\documentclass[a4paper]{amsart}
\usepackage[pdfauthor={Yiannis N. Petridis, Nicole Raulf, 
Morten S. Risager}, pdftitle={Double Dirichlet series and  quantum unique ergodicity of  weight 1/2 Eisenstein series}, pdfkeywords={Quantum Limits, 
Eisenstein series, Multiple Dirichlet series}, 
pdfcreator={Pdflatex},colorlinks=true]{hyperref}
\usepackage{amsmath,amssymb}
\usepackage{tikz} 
\usetikzlibrary{arrows,decorations.pathmorphing,backgrounds,fit,patterns,shadows,fadings} 
 
\DeclareMathOperator*{\res}{res}
\newtheorem{theorem}{Theorem}[section]
\newtheorem{lemma}[theorem]{Lemma}
\theoremstyle{definition}

\theoremstyle{remark}
\newtheorem{remark}[theorem]{Remark}

\newtheorem{proposition}[theorem]{Proposition}    

\usepackage{marginnote}
\usepackage[normalem]{ulem}

\numberwithin{equation}{section}

\def \parity {{p}}
\def \hyp {{}_3F_2} 
\def \l {{\lambda}}
\def \a {{\alpha}}

\def \e {{\epsilon}}

\def \g {{\gamma}}
\def \G {{\Gamma}}
\def \R {{\mathbb R}}
\def \H {{\mathbb H}}
\def \C {{\mathbb C}}
\def \Z {{\mathbb Z}}
\def \N {{\mathbb N}}

\def \Q {{\mathbb Q}}
\def \GmodH {{\Gamma\backslash\H}}

\newcommand{\vol}[1]{\textup{vol}( #1 )}
\newcommand{\abs}[1]{\left\lvert #1 \right\rvert}

\newcommand{\jacobi}[2]{\left(\frac{#1}{#2} \right)}

\author{Yiannis N. Petridis}
\address{Department of Mathematics, University College London, Gower Street, London WC1E 6BT, United Kingdom}
\email{i.petridis@ucl.ac.uk}
\author{Nicole Raulf}
\address{Laboratoire Paul Painlev\'e ,
U.F.R. de Math\'ematiques, Universit\'e Lille 1
Sciences et Technologies,
59 655 Villeneuve d'Ascq C\'edex, France 
}
\email{raulf@math.univ-lille1.fr}
\author{Morten S. Risager}
\address{Department of Mathematical
  Sciences, University of Copenhagen, Universitetsparken 5, 2100
  Copenhagen \O, Denmark}
\email{risager@math.ku.dk}


\title[Double Dirichlet series and weight 1/2 Eisenstein
series.]{Double Dirichlet series and  quantum unique ergodicity of
  weight 1/2 Eisenstein series }  

\thanks{
The first author was supported by NSF grant DMS-0401318, the second author was partly supported by the Labex CEMPI 
(ANR-11-LABX-0007-01),  and the third author was supported by a Sapere
  Aude grant from The Danish
Natural Science Research Council}

\date{\today}
\begin{document}
\begin{abstract}
The problem of quantum unique ergodicity (QUE) of weight $1/2$ Eisenstein series for $\Gamma_0(4)$ leads to the study of certain double Dirichlet series involving $\hbox{GL}_2$ automorphic forms and Dirichlet characters. We study the analytic properties of this family of double Dirichlet series (analytic continuation, convexity estimate) and prove that a subconvex estimate implies the  QUE result. 
\end{abstract}
\maketitle

\section{Introduction}
An important problem of quantum chaos is to describe the behavior of
eigenfunctions of Laplacians $\phi_{\lambda}$ with eigenvalue $\lambda$, as $\lambda\to\infty$. This
problem has a rich and interesting history
(see e.g. \cite{Snirelcprimeman:1974, Zelditch:1987, Colin-de-Verdiere:1985,
  Zelditch:1992a, Lindenstrauss:2006a, Soundararajan:2010a}). For
the weight $0$ Eisenstein series  $E(z,s)$ on the surface
$\hbox{SL}_2(\Z)\backslash \H$ Luo and Sarnak \cite{LuoSarnak:1995a}
determined the asymptotic behavior of the measures 
$$d\mu_t(z)=\abs{E\left(z,\frac{1}{2}+it\right)}^2d \mu(z)$$
on compact sets. Here $ d\mu(z) = \frac{dx^2 dy^2}{y^2} $ denotes 
the volume element corresponding to the hyperbolic metric on the 
upper half-plane $ \H $.
The main input in doing so was subconvex bounds on certain
standard $\hbox{GL}_1$ and $\hbox{GL}_2$ $L$-functions, namely the Riemann 
zeta function and the $L$-function of a Maa{\ss} cusp form. Their work was 
later generalized to the corresponding micro-local lifts
\cite{Jakobson:1994a},  and other arithmetic symmetric
spaces \cite{Koyama:2000, Truelsen:2011b}. Also for these
generalizations subconvex bounds were at the heart of the proofs.
 In
\cite{PetridisRaulfRisager:2013a} we studied similar questions for
scattering states.

In this paper we study the analogous problem for Eisenstein series of
weight 1/2. To be precise:  Let $E(z,s,1/2)$ be the
weight 1/2 Eisenstein series at the cusp infinity for the group
$\Gamma=\Gamma_0(4)$  (see
Section~\ref{Sec:weight-half-eisensteinseries}). We study
the limiting behavior as $\abs{t}\to\infty$ of 
\begin{equation} \label{changed-measure} 
d\mu_t(z)=\abs{E\left(z,\frac{1}{2}+it,1/2\right)}^2d \mu(z). 
\end{equation}
Since the Fourier coefficients $ \phi_n(s, 1/2) $ 
of $E(z, 1/2+it, 1/2)$ are essentially  
values of Dirichlet $L$-functions on the critical line, see 
\eqref{good-expression}, and, therefore, are not multiplicative, the 
problem is much harder.
The Rankin-Selberg convolutions that appear are not factored into 
standard $L$-functions. Instead, we find that  certain double Dirichlet 
series play a crucial
role. 
The relevant
double Dirichlet series are the following:

Let $\chi$, $\chi'$ be characters mod~$8$, and let $t_n$ be
either the eigenvalue of the Hecke operator $ T_n $ for a weight 
$ 0 $ Maa{\ss} form $ \psi $ on $\Gamma_0(4)\backslash\H$ or $t_n=\tau(n)$ 
be the divisor function. Let $s_0(1-s_0)$ be the corresponding 
Laplace eigenvalue of $ \psi $, with $\Re(s_0)\geq 1/2$, and if 
$t_n=\tau(n)$ let $s_0=1/2$.

We then define 
\begin{equation} \label{introduction-therightseries}
   Z(s,w,\chi,\chi') = 
   \zeta_2(4s-1)\sum_{\substack{n=1\\(n,2)=1}}^\infty \frac{\chi(n) t_{ n} 
   L^*(2w-1/2,n,\chi')}{{n}^{s-w+1/2}}  ,
\end{equation}
where $L^*(w,n,\chi)=q(w,n, \chi)L_2(w,\chi_{n_0}\chi)$. Here
$n_0$ is the squarefree part of $n$, $\chi_{n_0}(c)=\jacobi{n_0}{c}$ 
and $L_2(w,\chi_{n_0}\chi)$ is the standard $L$-function with the
2-factor removed.  The functions $q(w,n, \chi)$ are explicitly given 
so-called {\lq correction polynomials\rq}, see
\eqref{correction-polynomials-definition} below. 
The function $L^*(w,n,\chi)$ may seem strange at first, but it 
occurs naturally as the $n$-th Fourier coefficient of the Eisenstein 
series of weight 1/2, and it has many nice properties. See
e.g.\ \cite{Shimura:1973a} or Section \ref{Sec:weight-half-eisensteinseries} 
below. 

Friedberg and Hoffstein \cite{FriedbergHoffstein:1995a} have studied a
Rankin-Selberg integral (see \eqref{friedberghoffstein} below) which turns out to be a linear combination of
$Z(s,w,\chi,\chi')$ and $Z(s,w,\chi,\chi_4\chi')$, where $\chi_4$ is
the primitive character mod 4. They observed that this admits
meromorphic continuation and that certain linear combinations have a
pole at $(s,w)=(3/4,3/4)$ (in our normalization). They did this in order to
prove  non-vanishing of quadratic twists of 
$\hbox{GL}_2$-$L$-functions at the central point. 

Furthermore similar
series  with higher order twists instead of the quadratic characters
$\chi_{n_0}$ were studied by Brubaker, Bucur, Chinta, Frechette, and
Hoffstein \cite{BrubakerBucurChintaFrechetteHoffstein:2004} in order
to prove non-vanishing of  higher order twists. To understand the new series $Z(s,w,\chi,\chi')$ we follow essentially
the program introduced in \cite{BumpFriedbergHoffstein:1996a} to prove
the following: 

The series defining $Z(s,w,\chi,\chi')$ converges absolutely and 
uniformly in certain regions in $\C^2$, and hence defines an analytic 
function there. 
The functions $Z(s,w,\chi,\chi')$ admit meromorphic continuation 
to $\C^2$ and they satisfy a group of functional equations generated by 
$$\alpha: (s,w)\mapsto (s,1-w),\quad \quad \beta: (s,w)\mapsto (w,s)$$
The functions $ Z(s,w,\chi,\chi')$ grow at most
polynomially for $(\Re(s),\Re(w))$ in compact sets.
For the precise form of the functional equations we refer to Theorems
\ref{first-functional-equation} and \ref{second-functional-equation}. 
The group of functional equations is isomorphic to the dihedral group 
of order 8. A similar result for higher order twists may be found 
in \cite{BrubakerBucurChintaFrechetteHoffstein:2004}.

We want to investigate the growth of $Z(s,w,\chi,\chi')$ in $s$ and
$w$. The notions of analytic conductor and subconvexity are not
completely well-established for general multiple Dirichlet
series. Certain cases are dealt with in 
\cite{ Blomer:2011a, BlomerGoldmakherLouvel:2014} but a general 
theory is missing. 

To define these notions in the present case we note that when $\Re(s), \Re(w)>3/4$ the
function $Z(s,w,\chi,\chi')$  has a representation 
\begin{equation}\label{intro-another-nice-representation}
 Z(s,w,\chi,\chi')=\sum_{\substack{c=1\\(c,2)=1}}^\infty\frac{\chi'(c)L^{**}(s-w+1/2,\psi,c,\chi)}{c^{2w-1/2}},
\end{equation}
where $L^{**}(s,\psi,c,\chi)=Q^*(s,c,\chi)L_2(s,\psi\otimes\tilde\chi_
{c_0}\chi)$ (See
\eqref{another-nice-representation} and Theorem  \ref{second-functional-equation}). Here $c_0$ is the squarefree part of $c$,
$\tilde{\chi}_{c_0}(n)=\jacobi{n}{c_0}$ and  $L_2(s,\psi\otimes\tilde\chi_
{c_0}\chi) $ is the standard $L$-function with the 2-factor
removed. The functions $Q^*(s,c,\chi)$ are explicitly given so-called
\lq correction polynomials'', see \eqref{Q*-polynomials} below.

When proving bounds on standard $L$-functions one usually normalizes
the coefficients to be essentially bounded, at least on average. In our case it is not so
clear how to do that since the true size of $L^{**}(s,\psi,c,\chi)$
is known only conjecturally. If the generalized Lindel\"of hypothesis is true the
coefficients of the series \eqref{intro-another-nice-representation} are essentially bounded. We investigate what happens when this is true on average (over $c$). 
To be precise: we want to know what bound on
$Z(s,w,\chi,\chi')$  can be  proved, if we assume that  the coefficients are
essentially bounded, i.e. if 
\begin{equation}
  \label{lindelof-on-average} 
  \sum_{\substack{c\leq
      X\\(c,2)=1}}\abs{L^{**}(s,\psi,c,\chi)}=O(X^{1+\e}(1+\abs{s})^{\e}) .
\end{equation}
Using the properties of $Q^*(s,c,\chi)$ we will see that this follows
from assuming
\begin{equation} \label{LIND_AVERAGE}
  \sum_{\substack{1\leq c_0\leq X\\c_0 \textrm{odd}\\\textrm{squarefree}}}\abs{L_2(s,\psi\otimes\tilde\chi_
{c_0}\chi)}^2=O(X^{1+\e}(1+\abs{s})^{\e}) \textrm{ when }\Re(s)=1/2. 
\end{equation}
Also it is easy to see that \eqref{lindelof-on-average} implies
\eqref{LIND_AVERAGE} with the exponent 2 replaced by a 1. In particular it implies the generalized Lindel\"of hypothesis in
the $t$ parameter.

We now define the 
analytic conductor of $Z(1/2+it,1/2+iu,\chi,\chi')$ to be 
\begin{equation}\label{conductor-def}
\mathfrak q(t,u)=(1+\abs{t})(1+\abs{t+u})^{2}(1+\abs{u}). 
\end{equation}
Using an approximate functional equation argument for 
$Z(s,w,\chi,\chi')$ we can prove the following bound on the critical 
line: 
\begin{theorem}\label{intro-convexity-bound} 
Assume \eqref{lindelof-on-average}. Then 
\begin{equation} \label{convexity-bound}
Z(1/2+it,1/2+iu,\chi,\chi')=O_{\psi}(\mathfrak q(t,u)^{1/4+\varepsilon}). 
\end{equation}
Unconditionally 
\begin{equation}\label{trivial-bound} 
Z(1/2+it,1/2+iu,\chi,\chi')=O_{\psi}((\mathfrak q(t,u)(1+\abs{t-u})^2)^{1/4+\varepsilon}). 
\end{equation}
\end{theorem}
\begin{remark}
We call the unconditional bound \eqref{trivial-bound} 
the \emph{trivial} bound. The conditional bound
\eqref{convexity-bound} is called the
\emph{convexity} bound. Any bound
$O(\mathfrak q(t,u)^{1/4-\delta})$, with $\delta>0$ is called a \emph{subconvex}
bound with saving $\delta$. If $\delta=1/4-\varepsilon$ is permitted we say that
$Z(s,w,\chi,\chi')$ admits a Lindel\"of type bound.  In the
theory of $L$-functions, the notion of convexity and subconvexity 
is standard and has numerous applications,  see 
e.g. \cite{IwaniecKowalski:2004a}.
\end{remark}
\begin{remark}
 We note that even proving the  \emph{trivial} bound,  requires strong input. In particular, in order to prove
Theorem \ref{intro-convexity-bound} \eqref{trivial-bound}, we need the Lindel\"of hypothesis
on average in the conductor aspect for $L(s,\chi_n)$, and the convexity
estimate in the $s$ aspect. This bound is available, as follows
from Heath-Brown's famous large sieve inequality for quadratic characters 
\eqref{Heath-Brown-large-sieve}, see Theorem \ref{heath-brown-average}
below.

Also we note that we can prove unconditionally (See Lemma \ref{lincombZs} below)
that if $\{t_n\}$ comes from a cusp form $$Z(1/2+it,1/2-it,\chi,\chi')+b
Z(1/2+it,1/2-it,\chi,\chi')=O_{\psi}(\mathfrak
q(t,-t)^{1/4+\varepsilon}).$$ 
Here $b$ is the product of the sign of  $\chi$ and the sign of the
cusp form. We note that this is of the same order as the convexity
estimate above \emph{without} assuming  \eqref{lindelof-on-average}.

\end{remark}

\begin{remark}For special
  configurations of $s,w$  (in our case $s-w$ constant) the trivial
  bound and the convexity bound coincide. This is because in this case
  \eqref{LIND_AVERAGE} follows from Heath-Brown's estimate
  (see Theorem \ref{heath-brown-average} below).

 We emphasize that our notion of convexity is \emph{different} from
  that of
  Blomer, Goldmakher, and Louvel \cite{Blomer:2011a},
  \cite{BlomerGoldmakherLouvel:2014}. What we call the trivial bound
  corresponds to what they call the convexity bound. 
\end{remark}

\begin{remark} Even though we cannot prove it, it is
 not unreasonable to expect subconvexity for $Z(s,w,\chi,\chi')$!
Double Dirichlet series similar to $Z(s,w,\chi,\chi')$ -- with degree-one $L$-functions as
  coefficients -- \emph{are} known to satisfy subconvex bounds due to
  Blomer, Goldmakher, and Louvel \cite{Blomer:2011a},
  \cite{BlomerGoldmakherLouvel:2014}. (In 
  \cite{BlomerGoldmakherLouvel:2014} the authors consider a configuration such that the bound they prove
 would  be considered a subconvex
  bound also by our definition.  Likewise the bound proved in  \cite[Theorem 1]{Blomer:2011a} is a
  subconvex bound by our definition if one restricts to
  $s=1/2$ or $w=1/2$.)

 Furthermore it is known that on average the double Dirichlet series
 considered by Blomer admits  Lindel\"of type bounds \cite[Theorem 2]{Blomer:2011a}) in the $(s,w)$ aspect. In the conductor aspect
  (which is here the conductor related to the form with eigenvalues $\{t_n\}$)
  Hoffstein and Kontorovich
  \cite[(1.23)]{HoffsteinKontorovich:2010} conjecture Lindel\"of
  type bounds to hold.
\end{remark}

\begin{theorem}\label{QUE} Assume that for all $\chi$,$\chi', \{t_n\}$ the function $
  Z (s,w,\chi,\chi')$ admits a subconvex bound. Then
  for any compact Jordan measurable subsets $A$ and $B$ of $\Gamma\backslash \H$
  we have
 \begin{equation}\label{QUE-limit}\frac{\int_{A}\abs{E(z,1/2+it,1/2)}^2d
     \mu(z)}{\int_{B}\abs{E(z,1/2+it,1/2)}^2d
     \mu(z)}\to\frac{\vol{A}}{\vol B},\end{equation} as $\abs{t}\to\infty$.  
\end{theorem}

\begin{remark}\label{LuoSarnak-analogue}
Theorem \ref{QUE} is the analogue of the  Luo--Sarnak theorem 
\cite{LuoSarnak:1995a} for the weight $0$ Eisenstein series. Their theorem, 
however, is unconditional as in their case 
 subconvex
bounds for standard $\hbox{GL}_1$ and $\hbox{GL}_2$-$L$-functions
are readily available. As in \cite{LuoSarnak:1995a} 
we really prove -- conditionally on any subconvex bound -- 
the asymptotic result
\begin{equation} \label{againagain} 
\int_{A}\abs{E(z,1/2+it,1/2)}^2d \mu(z)\sim\frac{4}{\vol{\GmodH}}\vol{A}\log
\abs{t},\quad\quad \abs{t}\to\infty. 
\end{equation}
In contrast to the case of quantum unique ergodicity of Maa\ss \ cusp 
forms, the rate of convergence in (\ref{againagain}) is very slow. As in 
\cite{LuoSarnak:1995a} one can prove $ O(\log t/ \log \log t) $.

It is understood in many arithmetic cases that the equidistribution of masses  is implied by subconvexity bounds
for appropriate $L$-functions of degree 8 (see
 e.g. \cite{Sarnak:2011},
  \cite{Soundararajan:2010}, \cite{NelsonPitaleSaha:2014}). 
\end{remark}
\begin{remark}
The structure of the paper is as follows: In Section\ \ref{sec:doubledirichletseries} 
we study the double Dirichlet series
$Z(s,w,\chi,\chi')$ which arise when we address QUE of the weight 1/2
Eisenstein series $E(z,s,1/2)$. In Section\ \ref{Sec:weight-half-eisensteinseries}
we review the theory for $E(z,s,1/2)$ with explicit computations.  
In Section\ \ref{sec:main-theorem}, which is the main section of the
paper, we analyze \eqref{againagain} by splitting it in a cuspidal 
contribution and incomplete Eisenstein series contributions. E.g. in
the cuspidal space we find that for a cusp form $\psi$ with eigenvalue
$s_0(1-s_0)$, the integral
\begin{equation}
\int_{\GmodH}\psi(z)\abs{E(z,1/2+it,1/2)}^2d \mu(z)\label{unfoldable}
\end{equation}
 equals a linear
combination of terms of the form
\begin{equation}\label{unfolded}
c_{\chi,\chi'_\pm}(s,w)Z(s,w,\chi,\chi')\frac{1}{\Gamma(w\pm 1/4)}\int_0^\infty
W_{0,s_0-1/2}(2y)W_{\pm1/4,w-1/2} (2y)y^{s-1}\frac{dy}{y}
\end{equation}
evaluated at $(s,w)=(1/2+it,1/2-it)$. Here $c_{\chi,\chi'_\pm}(s,w)$
are functions which can easily be understood  when
$\Re(w)=\Re(s)=1/2$, and $W_{\mu,\nu}$ are Whittaker functions. In the
appendix we analyze the Mellin transform of the product of Whittaker
functions.

We can then 
deal with  \eqref{unfolded} using bounds on $Z(s,w,\chi,\chi')$. 
To deal with the cuspidal space we need subconvexity for $ Z(1/2+it,  1/2-it, \chi, \chi') $, with 
$t_n$ corresponding to Hecke eigenvalues for Maa{\ss} forms. 
For the 
incomplete Eisenstein series a similar analysis shows that we need the
same type of bound for $t_n=\tau(n)$ the divisor 
function for all configurations of $s$ and $w$. 
 We also use Zagier's theory of Rankin-Selberg 
integrals for 
functions not of rapid decay.
\end{remark}
\begin{remark}  Although the analytic continuation of 
  \begin{equation*}
I(s,w)=\int_{\Gamma\setminus\H}\psi (z)E(z,w,1/2)
\overline{E(z,\overline s,1/2)}d\mu(z)    
  \end{equation*}
(which specializes to \eqref{unfoldable} for special configuration)
follows from the well known analytic properties of $E(z,w,1/2)$ its
growth/decay properties jointly in $(s,w)$ are less clear. This is why we have to unfold and
eventually analyze $Z(s,w,\chi,\chi')$ to see that the above integral
 is $O(\abs{t}^{-\delta})$ for $s=1-w=1/2+it$ when
$\abs{t}\to\infty$ assuming subconvexity with saving $\delta$. The Maa\ss-Selberg relation
gives an upper bound (see e.g. Lemma \ref{MS-bound} below), but this is not good enough to prove Theorem \ref{QUE}.

\end{remark}
\begin{remark} One could speculate whether the implication in Theorem
  \ref{QUE} could be reversed, i.e. to what extent bounds on integrals
  like  \eqref{unfoldable} would
  imply bounds on $Z(s,w,\chi,\chi')$, via the expression \eqref{unfolded}. Such speculations are problematic at least for
  the following reason: We have good control over the
  asymptotics of the Mellin transform (see e.g.\ Lemma
  \ref{nicoleresult}) but since integrals like \eqref{unfoldable} are \emph{linear
    combinations} of terms of the form \eqref{unfolded}, we cannot
  conclude from bounds on integrals like \eqref{unfoldable} the same bounds on the
  individual summands. We elaborate on this in Lemma
  \ref{lincombZs} and Remark
  \ref{finalremarks} below.
 \end{remark}

\section{A double Dirichlet series}\label{sec:doubledirichletseries}
In this section we define and prove various properties of the double
Dirichlet series .  To prove its
meromorphic continuation and functional equation we proceed as in
\cite{BrubakerBucurChintaFrechetteHoffstein:2004} but with some
simplifications and refinements. We show, for instance, that knowing optimal bounds
towards the Ramanujan-Petersson conjecture is not necessary to get
optimal regions of convergence. To prove the convexity bounds we use a
combination of techniques from \cite{BlomerGoldmakherLouvel:2014} and
\cite{Blomer:2011a}. Although the techniques we use are certainly
known to the experts in the field, we were not able to find precise
enough statements in the existing literature for the double Dirichlet
series \eqref{introduction-therightseries}. 

We start by 
introducing some notation and deriving some basic results about Gauss
sums and Dirichlet series involving Gauss sums.

Let $\{t_n\}_{n\in \mathbb{N}}$ be the coefficients of the normalized
$L$-function
of a self dual $\hbox{GL}_2$ automorphic form $ \psi $. For good
primes, and  we  assume that only $p=2$ could potentially be a bad prime,
 the Satake parameters $\alpha_p$, $\beta_p$ satisfy
$\alpha_p+\beta_p=t_{p}$, $\alpha_p\cdot\beta_p=1$
and \begin{equation}\label{satake}t_{p^\lambda}=\sum_{j=0}^\lambda
  \alpha_p^j\beta_p^{\lambda-j}=\frac{\alpha_p^{\lambda+1}-\beta_p^{\lambda+1}}{\alpha_p-\beta_p}.\end{equation}
The Fourier coefficients satisfy the Ramanujan-Petersson conjecture on
average, since the
Rankin-Selberg method gives
\begin{equation}
  \label{RPC-average}
  \sum_{\abs{n}\leq X } \abs{t_n}^2\sim C X,
\end{equation}
as $X\to \infty$. Here $C$ is an explicit constant, see e.g.\ \cite[(8.15)]{Iwaniec:2002a}. 
The corresponding $p$-factor, i.e.\ the local $L$-function
is given by  $$L^{(p)}(s,\psi)=\sum_{\lambda=0}^\infty
\frac{t_{p^{{\lambda}}}}{p^{\lambda s}}=(1-t_pp^{-s}+p^{-2s})^{-1}=(1-\alpha_pp^{-s})^{-1}(1-\beta_pp^{-s})^{-1}.$$
Similar but easier identities and estimates are true for the divisor function $t_n=\tau(n)$ where
$\alpha_p=\beta_p=1$.

 For any $L$-function we will
write $L^{(p)}(s)$ for its corresponding $p$-factor and $L_2(s)$ for the
$L$-function with the 2-factor removed.

\subsection{Gauss sums and some related series}\label{intro-things}
We now recall a few basic relevant results about Gauss sums for real characters. 
Let $n$, $d$ be  integers with $d$ odd and positive and let $\jacobi{n}{d}$ be 
the Jacobi-Legendre symbol, i.e.
$$\jacobi{n}{d}=\prod_{p^v\|d}\jacobi{n}{p}^v,$$ where for an odd prime $p$ we
denote by $\jacobi{n}{p}$  the usual Legendre symbol. The symbol
$\jacobi{n}{d}$ is then extended to all odd $d\in \mathbb{Z}$ as in
\cite[p.~442]{Shimura:1973a}, see also \cite[p.~147, 187-188]{Koblitz:1984b}.

For an integer $n$ and a positive odd integer $d$ we define Gauss sums \begin{equation}
G_n(d):=\sum_{m\!\!\!\!\pmod d}\jacobi{m}{d}e\left(\frac{nm}{d}\right).
\end{equation}
Here $e(x)=e^{2\pi i x}$. Gauss ingeniously proved  that for odd 
squarefree $d$ we have
$G_{1}(d)=\varepsilon_{d}\sqrt{d}$ where
$\varepsilon_{d}=1$  if $d\equiv 1 (4)$ and $\varepsilon_{d}=i$  if $d\equiv -1 (4)$. Quadratic reciprocity states that
for $n,d$ relatively prime odd positive integers \begin{equation}\label{quad-red}
\Big(\frac{n}{d}\Big) \Big(\frac{d}{n}\Big)=(-1)^{\frac{n-1}{2}\frac{d-1}{2}}.
\end{equation}
It is elementary to verify that the right-hand side equals $\varepsilon_{n}\varepsilon_{d}/\varepsilon_{nd}
$. 
 For odd $d$ it turns out to be convenient to consider 
$$H_n(d):=\varepsilon_d^{-1}G_n(d).$$
\begin{proposition}\label{bigger-prop} The function $H_n(d)$ has the following properties:
\begin{enumerate}
\item \label{factor} For fixed $n$, $H_n(d)$ is multiplicative, i.e.\ if $d_1,d_2$ are
  coprime odd positive integers, then
  \begin{equation*}
    H_n(d_1d_2)=H_n(d_1)H_n(d_2).
  \end{equation*}
\item \label{anotherfactor}
If $(n_1,d)=1$, then
\begin{equation*}
H_{n_1n_2}(d)=\jacobi{n_1}{d}H_{n_2}(d).
\end{equation*}
\item \label{Gauss-basics}Let $\alpha,\beta$ be non-negative integers, and let $p$ be an
  odd prime. Then
  \begin{equation*}
H _{p^\alpha}(p^\beta)= \begin{cases}
\phi(p^\beta),&\textrm{if }\alpha\geq \beta, \beta\equiv 0 (2) ,\\
p^{\beta-1/2}(\delta_{\beta\equiv 1 (2)} -p^{-1/2} \delta_{\beta\equiv 0 (2)}),&\textrm{if }\alpha= \beta-1,\\
0,&\textrm{otherwise.}
  \end{cases}
\end{equation*}
\end{enumerate}
\end{proposition}
\begin{proof} \eqref{factor} follows from the Chinese remainder
  theorem and quadratic reciprocity, \eqref{anotherfactor} from
  the fact that  if $(n_1,d)=1$ then $n_1m$ runs through a set of
representatives mod~$d$, and \eqref{Gauss-basics} from elementary considerations.
\end{proof}
We now compute
\begin{equation}\sum_{\substack{c=1\\(c,2)=1}}^\infty\frac{
  \chi(c)H_n(c)}{c^{ 2s}}\quad\textrm{ and }\quad \sum_{\substack{n=1\\(n,2)=1}}^\infty\frac{t_{ n}\chi(n)H_{ 
  n}(c)}{n^{s}}
\end{equation}
where $\chi$ is a character mod~$q$, where $q\vert 8$.  As we shall 
see later these sums occur naturally in the Fourier coefficients
of the weight $1/2$ Eisenstein series of $\Gamma_0(4)$, and in
Rankin-Selberg-type integrals formed from these Eisenstein series.

For $n$ odd and positive we denote
 $$\tilde
\chi_{n}(c)=\jacobi{c}{n}$$ which is a character mod~$n$. When $n$ is
squarefree its conductor is $n$.

For $c$ odd we denote $$\chi_n(c)=\jacobi{n}{c}$$ which for $n$
odd and squarefree has an extension to all $c$ which is a character of
conductor $\abs{n}$ if $n\equiv 1 (4)$ and $4\abs{n}$ if $n\equiv 3 (4)$. See \cite[p.~147, 187-188]{Koblitz:1984b}.

By quadratic reciprocity \eqref{quad-red}
we have for odd positive $m,n$
\begin{equation}\label{quadrep-characters}\chi_n(m)=\tilde\chi_n(m)\begin{cases}1,& \textrm{ if }n\equiv
  1(4),\\\chi_4(m),& \textrm{ if }n\equiv 3(4),\end{cases}\end{equation}
where $\chi_4$ is the primitive character mod~$4$. We can write any non-zero integer $n$ uniquely as $n=n_0n_1^2$, where $n_0$ is
squarefree and $n_1>0$.  We define correction polynomials as
\begin{equation}\label{correction-polynomials-definition}
q(s,n,\chi)=\prod_{2\neq p\vert n_1}\sum_{\beta=0}^{v_p(n_1)}\frac{(1-\delta_{\beta<v_p(n_1)} \chi_{n_0}{(p)}\chi(p)p^{-s})}{p^{2\beta
  (s-1/2)}},
\end{equation}
where $v_p$ is the $p$-adic valuation.
 When
$\chi=1$ we sometimes write $q(s,n)=q(s,n,\chi)$.

We define \begin{equation}\label{L*}L^*(s,n,\chi)=q(s,n,\chi)L_2(s,\chi_{n_0}\chi).\end{equation}
\begin{lemma}\label{goodsplitting} 
We have 
$$\sum_{\substack{c=1\\(c,2)=1}}^\infty\frac{\chi(c)H_n(c)}{c^{2s}}=\frac{L^*(2s-1/2,n,\chi)}{\zeta_2(4s-1)}.$$
\end{lemma}
\begin{proof}
Using multiplicativity of $H_n(d)$ (Proposition \ref{bigger-prop}) we
see that the sum factors into local factors.
For a prime $p\neq 2$ we compute the corresponding
factor  
$$R_p(s)=\sum_{\beta=0}^\infty\frac{
  \chi(p^\beta)H_n(p^\beta)}{p^{\beta 2s}}.$$
Write $n=n'p^\alpha$ where $(n',p)=1$. Then using Proposition~\ref{bigger-prop} 
\eqref{anotherfactor}, \eqref{Gauss-basics} we have 
\begin{align} 
R_p(s) &= \sum_{\beta=0}^\infty \frac{\jacobi{n'}{p^{\beta}} \chi(p^\beta) 
H_{p^{\alpha}}(p^\beta)}{p^{\beta2s}}
\label{general-computation}= 
\sum_{\substack{\beta=0\\\beta\equiv 0 (2)}}^{\alpha}\frac{\phi(p^\beta)}{p^{\beta2s}} 
+ \frac{\jacobi{n'}{p^{{\alpha+1}}}\chi(p^{\alpha+1}) H_{p^{\alpha}}(p^{\alpha+1})}{p^{(\alpha+1)2s}}.
\intertext{Consider first $\alpha$ even, in which case $\alpha=2v_p(n_1)$. Then we find}
\nonumber 
R_p(s) &= 1+\sum_{\substack{\beta=1\\\beta\equiv 0 (2)}}^{\alpha}\frac{p^{\beta-1}(p-1)}{p^{\beta2s}} 
+ \frac{\chi_{n_0}(p)\chi(p)p^{\alpha+1/2}}{p^{(\alpha+1)2s}}, \\ 
\intertext{noting that $\chi_{n_0}(p)= \jacobi{n'}{p}$. By induction we find}
\nonumber 
R_p(s) &= \frac{L^{(p)}(2s-1/2,\chi_{n_0}\chi)}{\zeta^{(p)}(4s-1)} 
\left(\sum_{\substack{\beta=0\\\beta\equiv 0 (2)}}^{\alpha} p^{\beta(1-2s)} 
- \sum_{\substack{\beta=0\\\beta\equiv 0 (2)}}^{\alpha-2} {\chi_{n_0}(p) \chi(p) 
p^{-(2s-1/2)} {p^{\beta(1-2s)}}}\right).
\end{align}
Here we have used $\chi_{n_0}^2(p)=1$.

Returning now to (\ref{general-computation}) we assume  $\alpha$
is odd, in which case $\alpha-1=2v_p(n_1)$. We find that in this case
\begin{align} \nonumber 
R_p(s) &=1+\sum_{\substack{\beta=1\\\beta\equiv 0 (2)}}^{\alpha} 
\frac{p^{\beta-1}(p-1)}{p^{\beta2s}}+\frac{-p^{\alpha}}{p^{(\alpha+1)2s}} \\
\nonumber 
&=(1-p^{-(4s-1)})\sum_{\substack{\beta=0\\\beta\equiv 0 (2)}}^{\alpha-1}p^{\beta(1-2s)},
\intertext{where again we have used induction. Using
that for $\alpha$ odd $\chi_{n_0}(p)=0$ we may write this as }  
\nonumber 
R_p(s) &=\frac{L^{(p)}(2s-1/2,\chi_{n_0}\chi)}{\zeta^{(p)}(4s-1)} 
\sum_{\substack{\beta=0\\\beta\equiv 0 (2)}}^{\alpha-1}p^{\beta(1-2s)}.
\end{align} 
Since $\chi_{n_0}{(p)}=0$ we arrive at the desired result.
\end{proof}
\begin{proposition}\label{correction-polynomials} 
The function $q(s,n,\chi)$ has the following properties:
\begin{enumerate}
\item \label{prop1} If $n$ is squarefree, then $q(s,n,\chi)=1$.
\item \label{prop2} If $n=n_0n_1^2$ with $n_0$
  squarefree and $n_0,n_1$ odd, then 
$$q(s,n,\chi)=(n_1^2)^{1/2-s}q(1-s,n,\chi). $$
\item \label{prop3} If $\Re(s)\geq1/2$, then $q(s,n,\chi)=O(n^\epsilon)$
uniformly in $\Re(s)$.
\end{enumerate}
\end{proposition}
\begin{proof}
  These statements are all straightforward to verify from the
  definition. (\ref{prop1}) is clear and (\ref{prop2}) is easily
  verified by considering factors. Trivial estimates for \mbox{$\Re(s) 
  \geq 1/2$} lead to  $\abs{q(s,n,\chi)} \leq 2^{\#\{p\vert n\}} \tau(n)$ 
  which gives (\ref{prop3}).
\end{proof}

Write $c=c_0c_1^2$ with $c_0$ squarefree and set $v=v_p(c_1)$. We then
define, for odd $c$
\begin{equation}
  Q_\psi(s,c,\chi) = 
  \prod_{p \vert c_1} \frac{t_{p^{2v}}-t_{p^{2v-1}} \tilde \chi_{c_0}(p) \chi(p) 
  \big(\frac{p^{1-s}+p^{s}}{p}\big)+{t_{p^{2v-2}}{\tilde\chi_{c_0}(p)^2}/{p}}}{p^{2v(s-1/2)}}.
\end{equation}
Since $\psi$ is fixed, we shall often omit it from the notation and
simply write $Q(s,c,\chi)$. We define 
\begin{equation} 
L^*(s,c,\psi,\chi):=Q_\psi(s,c,\chi)L_2(s,\psi\otimes\tilde\chi_ {c_0}\chi). 
\end{equation}

\begin{lemma} \label{intermediate-computed}
Let $c$ be an odd natural number. Then
$$  
\sum_{\substack{n=1\\(n,2)=1}}^\infty\frac{t_{ n}\chi(n)H_{ n}(c)}{n^{s}}
= \sqrt{c}L^*(s,c,\psi,\chi). 
$$
\end{lemma}

\begin{proof} A similar computation can e.g.\ be found in 
\cite[Sec.~3]{BrubakerBucurChintaFrechetteHoffstein:2004}. 
We first show that the Dirichlet series factors into local 
factors. For $p$ an odd prime write  $c=c'p^l$ with $(c',p)=1$, 
and $m=p^{v_p{(m)}}\frac{m}{p^{v_p(m)}}$. Then using 
Propositions~\ref{bigger-prop} \eqref{factor} and
\eqref{anotherfactor} we find
\begin{align}
\label{factoringG} 
\nonumber
H_m(c) = 
\jacobi{m/p^{v_p(m)}}{p^l}\jacobi{p^{v_p(m)}}{c'}H_{p^{v_p(m)}}(p^l)H_{m/p^{v_p(m)}}(c'). 
\end{align}
Writing $m=np^\lambda$, we can write the Dirichlet series as
\begin{align*}
\sum_{\substack{n=1\\(n,2p)=1}}^\infty & \sum_{\lambda=0}^\infty 
\frac{t_{np^\lambda}\chi({np^\lambda})}{(np^\lambda)^{s}} 
H_{np^\lambda}(c'p^l)\\
&=\sum_{\substack{n=1\\(n,2p)=1}}^\infty 
\frac{t_{n}\chi(n)H_{n}(c')\jacobi{n}{p^l}}{n^{s}}\left(
\sum_{\lambda=0}^\infty \frac{t_{p^\lambda}\chi(p^\l)}{p^{\lambda s}} 
H_{p^\lambda}(p^l)\jacobi{p^\lambda}{c'}\right). 
\end{align*} 
Repeating this argument for every prime $p$ it follows that
the series factors as \begin{equation} 
  \label{factored-expression}
\prod_{p\neq 2}\left(\sum_{\lambda=0}^\infty
\frac{t_{p^\lambda}}{p^{\lambda s}}H_{p^\lambda}(p^{v_p{(c)}})\jacobi{p^\lambda}{c/p^{v_p{(c)}}}\chi(p^\l)\right).
\end{equation}
We now compute the local factors of
(\ref{factored-expression}) i.e.\ we compute, for $p\neq 2$
\begin{equation}\label{good-primes-again}
\sum_{\lambda=0}^\infty
\frac{t_{p^\lambda}\chi(p^\l)}{p^{\lambda s}}H_{p^\lambda}(p^l)\jacobi{p^\lambda}{c'},\end{equation}
where $l=v_p(c)$ and $c'=c/p^{v_p(c)}$. If $l=0$ the sum reduces to 
\begin{align*}
\sum_{\lambda=0}^\infty
\frac{t_{p^\lambda}\chi(p^\l)}{p^{\lambda s}}\jacobi{p^\lambda}{c}
&=L^{(p)}(s,\psi\otimes
\tilde
\chi_ {c_0}\chi),\end{align*}
where we have used that $\tilde \chi_ {c}(p)=\tilde \chi _{c_0}(p)$ if 
$(p,c)=1$. Here $c_0$ denotes the squarefree part of $c$.

If $l>0$ is even we use Proposition~\ref{bigger-prop}~\eqref{Gauss-basics}  
to see that in this case \eqref{good-primes-again} is equal to
\begin{equation}\label{sunshine}
\left(-\frac{t_{p^{l-1}}p^{l-1}\chi(p^{l-1})}
{p^{(l-1) s}}\jacobi{p^{l-1}}{c'}+
\sum_{\l=l}^\infty\frac{t_{p^\l}p^{l-1}(p-1)\chi(p^\l)}
{p^{\l s}}\jacobi{p^\lambda}{c'}\right).
\end{equation}
For $ t_n$ being a Hecke eigenvalue we can use the Satake parameters 
and evaluate the resulting geometric sums to see that 
\begin{align}
\sum_{\l=l}^\infty & \frac{t_{p^\l}\chi(p^\l)}{p^{\l s}} \jacobi{p^\lambda}{c'} \\  
\nonumber
&=\frac{1}{\alpha_p-\beta_p} \sum_{\l=l}^\infty \frac{\alpha_p^{\lambda+1}-
\beta_p^{\lambda+1}}{p^{\l s}} \chi(p^\l) \jacobi{p^\lambda}{c'} \\  
\nonumber 
&=\frac{1}{\alpha_p-\beta_p} \left(\frac{\alpha_p^{l+1}}{p^{l s}} 
(1-\alpha_p\jacobi{p}{c'} \chi(p)p^{- s})^{-1}-\frac{\beta_p^{l+1}}{p^{l s}} 
(1-\beta_p\jacobi{p}{c'}\chi(p)p^{- s})^{-1}\right),\\  
\intertext{where we have used $(\jacobi{p}{c'}\chi(p))^l=1$. Now the sum becomes}
\nonumber 
&=\frac{L^{(p)}(s,\psi\otimes \tilde \chi_ {c_0}\chi)}{p^{l s}}\\
\nonumber&\qquad\cdot\frac{1}{\alpha_p-\beta_p}\left(\alpha_p^{l+1} 
(1-\beta_p\jacobi{p}{c'}\chi(p)p^{- s})-\beta_p^{l+1}(1-\alpha_p 
\jacobi{p}{c'}\chi(p)p^{- s})\right)\\  
\nonumber 
&=\frac{L^{(p)}(s,\psi\otimes \tilde \chi_ {c_0}\chi)}{p^{l s}}
(t_{p^l}-t_{p^{l-1}} \jacobi{p}{c'}\chi(p)p^{- s}).
\end{align}
This is also true when $t_n=\tau(n)$ from a similar computation, which
we omit.

It follows that  \eqref{sunshine} can be written as
\begin{align*}
p^{l-1}&[\frac{-t_{p^{l-1}}}{p^{(l-1) s}}\tilde\chi_{c_0}(p^{l-1})\chi(p^{l-1})+\frac{L^{(p)}(s,\psi\otimes
\tilde\chi_ {c_0}\chi
)}{p^{l s}}
(p-1)(t_{p^l}-t_{p^{l-1}}\jacobi{p}{c'}\chi(p)p^{- s})]\\
&=p^{l-1}\frac{L^{(p)}(s,\psi\otimes
\tilde
\chi_ {c_0}\chi
)}{p^{l s}}[\frac{-t_{p^{l-1}}}{p^{- s}}\tilde\chi_{c_0}(p^{l-1})\chi(p^{l-1})(1-t_p\tilde\chi_{c_0}(p)\chi(p)p^{- s}+p^{-2 s})
\\
&\qquad \quad
+(p-1)(t_{p^l}-t_{p^{l-1}}\jacobi{p}{c'}\chi(p)p^{- s})].\\
\intertext{We use that the Hecke-eigenvalues satisfy
$t_{p^{l-1}}t_p=t_{p^{l}}+t_{p^{l-2}}$ to get} 
&=p^{l/2}\frac{L^{(p)}(s,\psi\otimes
\tilde
\chi_ {c_0}\chi
)}{p^{l(s-1/2)+1}}[pt_{p^l}-t_{p^{l-1}}\tilde\chi_{c_0}(p)\chi(p)(p^{1-s}+p^{s})+t_{p^{l-2}}].
\end{align*}
If instead $l>0$ is odd we can again use  Proposition
\ref{bigger-prop} \eqref{Gauss-basics}
and we find that in this case \eqref{good-primes-again} is equal to
\begin{equation*}
\frac{t_{p^{l-1}}}{p^{(l-1)s}}p^{l-1/2}\jacobi{p^{l-1}}{c'}\chi(p^{l-1})=\frac{t_{p^{l-1}}}{p^{(l-1)(s-1)-1/2}}.
\end{equation*}
We note also that $\tilde\chi_{c_0}(p)=\jacobi{p}{c_0}=0$ since by $l$
being odd we may conclude that $c_0$ is divisible by $p$. It follows
that in this case $L^{(p)}(s,\psi\otimes
\tilde
\chi_ {c_0}\chi
)=1$, and we conclude that \eqref{good-primes-again} can be written as
$$\frac{p^{l/2}t_{p^{l-1}}}{p^{(l-1)(s-1/2)}} L^{(p)}(s,\psi\otimes
\tilde
\chi_ {c_0}\chi
),$$ which gives the desired result in this case.
\end{proof}
\begin{proposition}\label{More-correction-polynomials}
The function $Q(s,c,\chi)$ has the following properties:
\begin{enumerate}
\item \label{prop21} If $c$ is squarefree, then $Q(s,c,\chi)=1$.
\item \label{prop22} If $c=c_0c_1^2$ with $c_0$ squarefree and $c_0, c_1$ odd, then 
$$(c_1^2)^{1-2s}Q(1-s,c,\chi)=Q(s,c,\chi). $$
\end{enumerate}
\end{proposition}
\begin{proof}
Statement (\ref{prop21}) is clear and  (\ref{prop22}) is easily
  verified by considering factors. 
\end{proof}
We would like to have bounds analogous to Proposition~\ref{correction-polynomials}~\eqref{prop3}. 
Any bound of the form  $\abs{t_{p^{l}}}\leq \tau({p^{l}})p^{\theta l}$ implies that,
 when $\Re(s)\geq 1/2$, 
\begin{equation}\label{standard-rpc}\abs{Q(s,c,\chi)}\leq \tau(c)4^{\#\{p\vert
  c\}}c^\theta=O(c^{\theta+\e}).\end{equation}
  The Ramanujan-Petersson conjecture will give the strongest bound with $\theta=0$.
Since the Ramanujan-Petersson conjecture is true on average
\eqref{RPC-average} we can prove that $Q(s,c,\chi)$ is bounded on
average:
\begin{lemma}\label{Q-on-square-average}For $\Re(s)\geq 1/2$ we have  
  \begin{equation*}
    \sum_{\substack{c\leq X, \\ c \, \textup{odd}}} 
    \abs{Q(s,c,\chi)}^2=O(X^{1+\varepsilon}),
  \end{equation*}
uniformly in $s$.
\end{lemma}
\begin{proof}
Write $c=c_0c_1^2$ with $c_0$ squarefree and $ c $ odd. It is easy to 
see that 
\begin{align*}\abs{Q(s,c,\chi)}&\leq \prod_{p\vert
    c_1}\abs{t_{p^{2v_p(c_1)}}}+2\abs{t_{p^{2v_p(c_1)-1}}}+\abs{t_{p^{2v_p(c_1)-2}}}\\
&\leq  \prod_{p\vert
    c_1}4\max_{i=0,1,2}\abs{t_{p^{2v_p(c_1)-i}}}\\
&=4^{\#\{p\vert c_1\}}\abs{t_{d_0}}, \quad\quad \textrm{where $d_0$ is some
  divisor of $c_1^2$.}
\end{align*}
It follows that $$\abs{Q(s,c,\chi)}^2\leq 16^{\#\{p\vert
  c_1\}}\abs{t_{d_0}}^2\leq 16^{\#\{p\vert
  c\}}\sum_{d\vert c}\abs{t_{d}}^2.$$
Using the Ramanujan-Petersson conjecture on average \eqref{RPC-average}
and $16^{\#\{p\vert
  c\}}=O(c^\varepsilon)$
we find
\begin{align*}
   \sum_{c\leq X}\abs{Q(s,c,\chi)}^2&=
   O(X^\varepsilon\sum_{c\leq X}\sum_{d\vert
     c}\abs{t_{d}}^2)\\
&=O(X^\varepsilon\sum_{d\leq X}\abs{t_d}^2\#\{c\leq X\vert
\quad d\textrm{ divides }c\})\\
&=O(X^{1+\varepsilon}\sum_{d\leq X}\frac{\abs{t_d}^2}{d})=O(X^{1+\varepsilon}).
\end{align*}
\end{proof}
We are now ready to define the double Dirichlet series. 
Let $\chi_4$ be the
primitive character mod~$4$, i.e.\ $\chi_4(n)=\jacobi{-1}{n}=(-1)^{(n-1)/2}$ for
$(n,2)=1$, and let
$\chi_8$ be the primitive character mod~$8$ given by
$\chi_8(n)=\jacobi{2}{n}=(-1)^{\frac{1}{8}(n-1)(n+1)}$ for $(n,2)=1$.  Let $\chi$,
$\chi'$ be characters mod~$8$, i.e $\chi$,
$\chi'$ are induced from 1, $\chi_4$, $\chi_8$, or $\chi_4\chi_8$. 
We then define 
\begin{equation}\label{therightseries}
  Z(s,w,\chi,\chi')=\zeta_2(4s-1)\sum_{\substack{n=1\\(n,2)=1}}^\infty\frac{\chi(n)t_{ n}L^*(2w-1/2,n,\chi')}{{n}^{s-w+1/2}}  .
\end{equation}
It is easy to see -- using Proposition~\ref{correction-polynomials}
\eqref{prop3}  and (\ref{L*}) -- that  for $\Re(2w-1/2)$,
$\Re(s-w+1/2)$ large enough the series is absolutely and
locally uniformly convergent. 

By Lemma \ref{goodsplitting} we see that 
$$Z(s,w,\chi,\chi')=\zeta_2(4s-1)\zeta_2(4w-1)\sum_{\substack{n=1\\(n,2)=1}}^\infty
    \frac{t_{n}\chi(n)}{n^{s-w+1/2}}\sum_{\substack{c=1\\(c,2)=1}}^\infty\frac{\chi'(c)H_n(c)}{c^{2w}}.$$
Interchanging summations and using Lemma \ref{intermediate-computed}
we see that this equals
\begin{equation}\label{full-intermediate-expression}
Z(s,w,\chi,\chi')=\zeta_2(4s-1)\zeta_2(4w-1)\sum_{\substack{c=1\\(c,2)=1}}^\infty\frac{\chi'(c)L^*(s-w+1/2,c,\psi,\chi)}{c^{2w-1/2}}.
\end{equation}
 Note
that, since $$\zeta_2(4s-1)\zeta_2(4w-1)=\sum_{\substack{n=1\\(n,2)=1}}^\infty\frac{\sigma_{2-4(s-w+1/2)}(n)}{n^{2(2w-1/2)}},$$
 we also have the series representation 
\begin{equation}\label{another-nice-representation}
 Z(s,w,\chi,\chi')=\sum_{\substack{c=1\\(c,2)=1}}^\infty\frac{\chi'(c)L^{**}(s-w+1/2,\psi,c,\chi)}{c^{2w-1/2}},
\end{equation}
where
$$L^{**}(s,\psi,c,\chi)=Q^*(s,c,\chi)L_2(s,\psi\otimes\tilde\chi_ {c_0}\chi)$$
with \begin{equation}\label{Q*-polynomials}Q^*(s,c,\chi)=\sum_{l^2\vert c}\sigma_{2-4s}(l)Q(s,c/l^2,\chi).\end{equation}
\begin{remark}
 The two representations \eqref{therightseries}, 
\eqref{full-intermediate-expression} will be instrumental in proving
meromorphic continuation of $Z(s,w,\chi,\chi')$ to $\C^2$. The proof follows the strategy outlined in
\cite{BumpFriedbergHoffstein:1996a}, \cite{DiaconuGoldfeldHoffstein:2003a}.
The choice of arguments in the definition of
  \eqref{therightseries}, i.e.
$2w-1/2$ and $s-w+1/2$,
might seem a bit strange, but for the purpose we have in mind it is the
most natural one. We shall see that with this choice the
functional equations are especially simple.
\end{remark}

\subsection{Functional equations of the standard $L$-functions}
We now recall the functional equations for the two $L$-functions
$L(s,{\chi_{n_0}\chi})$ and $L(s,\psi\otimes\tilde\chi_ {c_0}\chi)$.

\subsubsection{$\hbox{GL}_1$} \label{gl1twists} We will use the functional equation for $L_2(s,\chi_{n_0}\chi)$ for
$n_0$ a squarefree odd natural number, and $\chi$ mod~$8$: Let $\chi_0^8$ be the trivial character mod~$8$.
We have that $\chi_{n_0}\chi$ is odd precisely if
$\chi=\chi_4\chi_0^8$ or $\chi=\chi_4\chi_8$. Also it is known  (see
e.g.\ \cite[Ch.~5]{Davenport:2000}) that  $\chi_{n_0}\chi$ is induced
from the primitive character 
\begin{equation*}
{(\chi_{n_0}\chi)^
*} =\begin{cases}
 \chi_{n_0} ,&\textrm{if }n_0\equiv 1 (4), \chi=\chi_0^8,\\
\chi_4\chi_{-n_0} ,&\textrm{if }n_0\not \equiv 1 (4), \chi=\chi_0^8,\\
\chi_4 \chi_{n_0}, &\textrm{if }n_0\equiv 1 (4), \chi=\chi_4\chi_0^8,\\
\chi_{-n_0} ,&\textrm{if }n_0\not \equiv 1 (4), \chi=\chi_4\chi_0^8,\\
\chi_8\chi_{n_0}, &\textrm{if }n_0\equiv 1 (4), \chi=\chi_8\chi_0^8,\\
\chi_4\chi_8\chi_{-n_0} ,&\textrm{if }n_0\not \equiv 1 (4), \chi=\chi_8\chi_0^8,\\
\chi_4\chi_8\chi_{n_0}, &\textrm{if }n_0\equiv 1 (4), \chi=\chi_4\chi_8\chi_0^8,\\
\chi_8\chi_{-n_0} ,&\textrm{if }n_0\not \equiv 1 (4), \chi=\chi_4\chi_8\chi_0^8.\\
\end{cases}
\end{equation*}
It follows that
\begin{equation}\label{standard-functional-equation}L(s,{(\chi_{n_0}\chi})^*)=\left(\frac{\delta_{n_0,\chi}}{\pi}\right)^{1/2-s}\frac{\Gamma\left(\frac{1-s+\kappa_{\chi}}{2}\right)}{\Gamma\left(\frac{s+\kappa_{\chi}}{2}\right)}L(1-s,{(\chi_{n_0}\chi)^*})\end{equation}
where 
\begin{equation}\label{parity}
  \kappa_\chi=\begin{cases}
0,&\textrm{ if }\chi=\chi_{0}^8,\chi_8,\\
1,&\textrm{ if }\chi=\chi_4\chi_0^8,\chi_4\chi_8.
\end{cases}
\end{equation}
\begin{equation*}
  \delta_{n_0,\chi}=\begin{cases}
n_0,&\textrm{ if }\chi=\chi_0^8, n_0\equiv 1(4)\textrm{ or }\chi=\chi_4\chi_0^8, n_0\not\equiv 1(4),\\
4n_0,&\textrm{ if }\chi=\chi_0^8, n_0\not \equiv 1(4)\textrm{ or }\chi=\chi_4\chi_0^8, n_0\equiv 1(4),\\
8n_0 ,&\textrm{ if } \chi=\chi_8, \chi_4\chi_8.
\end{cases}
\end{equation*}
Note that all the functional equations are even,
i.e.\ $\frac{G_1({(\chi_{n_0}\chi})^*)}{i^{\kappa_\chi}\sqrt{\delta_{n_0,\chi}}}=1$.
                               
We have 
\begin{align*}L(s,\chi_{n_0}\chi)&=\prod_{p\vert\frac{8n_0}{\delta_{n_0,\chi}}}(1-{(\chi_{n_0}\chi)^*}(p)p^{-s})L(s,{(\chi_{n_0}\chi)^*})\end{align*}
and, also
 \begin{align}\nonumber L_2(s,\chi_{n_0}\chi)&
  =L_2(s,{(\chi_{n_0}\chi)^*})\\  
\label{h2-def}&=L(s,{(\chi_{n_0}\chi})^*) h_2(s,n_0,\chi).
\end{align}
where $h_2(s,n_0,\chi)$ is either 1, $1-2^{-s}$, or $1+2^{-s}$. Since
$(\chi_{n_0}\chi)^*(2)$ depends only on $\chi$ and $n_0$ mod~$8$, $h_2$
has the same dependence. 

\subsubsection{$\hbox{GL}_2$}\label{sec:gl2twists}
 We now turn to $L_2(s,\psi\otimes\tilde\chi_ {c_0}\chi)$ for
$c_0$ a squarefree odd natural number, and $\chi$ mod~$8$:
The
character $\tilde \chi_{c_0}$ is primitive of conductor $c_0$, and  is even precisely when
$\tilde\chi_{c_0}(-1)=\chi_4(c_0)=1$, i.e.\ when $c_0\equiv 1 (4)$. A reference on twisting of automorphic forms (at least for modular
forms) is \cite[Sec.~14.8]{IwaniecKowalski:2004a}. 

We need to take special care of 2-factors. For any primitive
automorphic form $f$ for $GL_2$ we define a polynomial $p_{2,f}(z)$ of degree 1 or 2
depending on whether 2 is ramified or not, by
\begin{equation}\label{p2poly}\frac{1}{p_{2,f}(z)}=\sum_{j=0}^\infty
  t_{2^j}(f)z^j,\end{equation} where $t_n(f)$ are the coefficients of
$L(s,f)$. In particular the 2-factor of  $L(s,f)$ equals
$p^{-1}_{2,f}(2^{-s})$.    If $p_{2,f}$ is of degree $2$,
$p_{2,f}(z)=(1-\alpha_2 z)(1-\beta_2 z)$,  the estimate $\abs{\alpha_2},
\abs{\beta_2}<2^{1/5}$ \cite[p.~549]{Shahidi:1988}, shows that $p_{2,f}(\pm
2^{-s})$ is uniformly bounded away from 0 at $\Re (s)\geq 1/2$. If $p_{2,f}(z)$ is of degree
$1$, the explicit value of $t_2$ ($=0$ or $\pm1/\sqrt{2}$) shows that
$p_{2, f}(\pm 2^{-s})$ does not vanish on $\Re (s)\geq1/2$ and as 
a result
\begin{equation}
  \label{anotherboundagain}
  \frac{1}{p_{2,f}(\pm 2^{-s})}=O(1)
\end{equation} uniformly in $f$ when $\Re(s)\geq 1/2$.

We assume now that $\psi$ is primitive Maa{\ss} Hecke form for $\Gamma_0(4)$ with real
Fourier coefficients. The twisted function $\psi\otimes \chi$  is
still a Hecke form with trivial character $\chi^2$ but not necessarily
primitive. Let $g= (\psi\otimes\chi)^*$ be the primitive form whose Fourier coefficients agree
with those of  $\psi\otimes\chi$ except possibly at the
$2$-factor. This is a cusp form of level $N=N_{\psi,\chi}=2^j$, a divisor of $64$.
For fixed $\psi$ there are $4$ such forms $g$, as there are $4$
characters mod~$8$. We have  that $L_2(s, \psi\otimes
\chi)=L_2(s, g)$ since the Fourier coefficients of $g$ and $\psi\otimes \chi$
agree on odd numbers. 

We now twist $g$ by $\tilde\chi_{c_0}$. Since the conductor of $\tilde\chi_{c_0}$ is relatively prime to the level of $g$, the result is a primitive cusp form of level $N\cdot c_0^2$. The twisted $L$-function $L(s, \psi\otimes \tilde\chi_{c_0}\chi)$  agrees with $L(s, g\otimes \tilde\chi_{c_0})$ outside the prime $2$, so that
\begin{equation*}L_2(s, g\otimes \tilde\chi_{c_0})=L_2(s, \psi\otimes \tilde\chi_{c_0}\chi).\end{equation*}
We have the functional equation of $g\otimes \tilde\chi_{c_0}$:
\begin{equation}\label{gl2-functional-equation}
L(s, g\otimes \tilde\chi_{c_0})= \e(g, \tilde\chi_{c_0})\left(\frac{Nc_0^2}{\pi^2}\right)^{1/2-s} \prod_{\e\in\{\pm 1\}}\frac{
\Gamma\left(\frac{1-s+\kappa_{\chi,\psi,c_0}+\epsilon(s_0-1/2)}{2}\right)
}{\Gamma\left(\frac{s+\kappa_{\chi,\psi,c_0}+\epsilon(s_0-1/2)}{2}\right)}
L(1-s, g\otimes \tilde\chi_{c_0}).\end{equation}
 This functional equation involves the  root number $\e(g, \tilde\chi_{c_0})$ that  depends on $c_0$ mod~$8$
as it is given by $$\e(g)\chi^2(
c_0)\tilde\chi_{c_0}(2^j)G(\tilde\chi_{c_0})^2/c_0,$$ where $\e(g)$ is
the root number of $g$.
We have 
\begin{equation*}L_2(s, \psi\otimes \tilde\chi_{c_0}\chi)={H_2(s, g,
    c_0) }L(s, g\otimes \tilde\chi_{c_0})\end{equation*}
where
\begin{equation}\label{def-H2}
  H_2(s, g,
    c_0)=p_{2,g\otimes\tilde \chi_{c_0}}(2^{-s})=p_{2,g}(\tilde \chi_{c_0}(2)2^{-s}).
\end{equation}
The dependence of $H_2( s, g,
  c_0)$ on $c_0$ is only mod~$8$, as it involves
    $\tilde\chi_{c_0}(2)$. We note also that
    $\kappa_{\chi,\psi,c_0}=\kappa_{\chi,\psi}\tilde\chi_{c_0}(-1)$
    depends only on $c_0$ mod~$4$ since $\tilde\chi_{c_0}(-1)=\chi_{4}(c_0)$.
\begin{remark}
In the $GL_1\times GL_1$ case, i.e.\ if $\psi=\psi_\tau$ and $t_n=\tau(n)$,
we have
\begin{equation*}
L(s,\psi_\tau\otimes\tilde\chi_{c_0}\chi) := \sum_{n=1}^\infty 
\frac{\tau(n)\tilde\chi_{c_0}\chi(n)}{n^s}=L(s,\tilde \chi_{c_0}\chi)^2.
\end{equation*}
We see (after using quadratic reciprocity) that the analogues of
the results of this section follow from section \ref{gl1twists}.
\end{remark}

\subsection{Average bounds on twisted $L$-functions}

Before we can give the proof of the meromorphic continuation we recall
a few facts concerning the involved $L$-series. We first recall an
average bound on $L$-functions twisted with quadratic characters. The
main ingredient in proving such a bound is Heath-Brown's large sieve estimate for quadratic characters. He
proves  \cite[Theorem 1]{Heath-Brown:1995} that for any positive
$\varepsilon>0$ there exists a constant $C>0$ such that for any positive integers $M$,
$N$  and for arbitrary complex numbers $a_1,\ldots, a_N$ we have 
\begin{equation}\label{Heath-Brown-large-sieve}
  \sum_{m\leq M}\!\!\!\hbox{}^*\abs{\sum_{n\leq
      N}\!\!\!\hbox{}^*a_n\jacobi{n}{m}}^2\leq
  C(MN)^\varepsilon(M+N)\sum_{n\leq N}\!\!\!\hbox{}^*\abs{a_n}^2. 
\end{equation}
Here a $*$ means summation over positive odd squarefree integers. From
this one can prove the following
\begin{theorem}\label{average}
  For $\Re(s)\geq1/2$ 
\begin{equation}\label{heath-brown-average}
\sum_{\substack{1<d_0\leq X\\
d_0\textrm{ odd} 
\\ \textrm{squarefree}}}\abs{L(s,\chi_{ d_0}\chi)}^4   =O((X\abs{s})^{1+\e}), 
\end{equation}
\begin{equation}\label{generalized-average}
\sum_{\substack{1<d_0\leq X\\
d_0\textrm{ odd}
    \\\textrm{squarefree}}}\abs{L(s,\psi\otimes\tilde \chi_{ d_0}\chi)}^2   =O((X\abs{s})^{1+\e}) .
  \end{equation}
\end{theorem}
The bound
(\ref{heath-brown-average}) is already in \cite[Theorem~2]{Heath-Brown:1995}  
and (\ref{generalized-average}) is essentially proved in the same way. See also
\cite[Sec.~2.3]{SoundararajanYoung:2010a} and \cite[Lemma~3.2]{ChintaDiaconu:2005}. 
These bounds give 
the Lindel\"of hypothesis on average in the character aspect, while
keeping the convexity bound in the $s$ aspect when $\Re(s)=1/2$.
\begin{remark} By considering 2-factors it is straightforward to see
  that the above bounds, i.e.\ \eqref{heath-brown-average} and
  \eqref{generalized-average} are true also if we remove 2-factors, 
  i.e.\ replace $L$ by $L_2$.
\end{remark}
\subsection{Meromorphic continuation and functional equations of
  $Z(s,w,\chi,\chi')$}
We first analyze $Z(s,w,\chi,\chi')$ from the representation \eqref{therightseries}.
\begin{theorem}\label{first-functional-equation} The function $(w-3/4)Z(s,w,\chi,\chi')$ is analytic in
  $$D_1=\{(s, w):\Re(s-w)>1/2, \Re(s+w)>3/2\},$$ and satisfies a functional
  equation $\alpha: (s,w)\mapsto (s, 1-w)$ given by
  \begin{equation*}
   (1-2^{-(3-4w)}) Z(s,w,\chi,\chi') = 
   \frac{\Gamma\left(\frac{3/2-2w+\kappa_{\chi'}}{2}\right)}{\Gamma\left(\frac{2w-1/2+\kappa_{\chi'}}{2}\right)} 
   \sum_{\chi'' \!\!\!\!\! \mod 8}p_{\chi''}(w)Z(s,1-w,\chi'',\chi').
  \end{equation*}
Here the $p_{\chi''}(w)$ are polynomials in $2^{-w}$. In
  particular they are  bounded in vertical strips. Furthermore, away from $w=3/4$
\begin{align*}
Z(s,&w,\chi,\chi')=\\&=\begin{cases}O((\abs{w}+1)^{1/4+\varepsilon}), &\textrm{
    for } 1/2\leq \Re{w}\leq K \textrm{ and } \Re(s-w)\geq 1/2+\delta,\\
O((\abs{w}+1)^{1/4 +1-2\Re(w)+\varepsilon}), &\textrm{
    for } -K \leq \Re{w}\leq 1/2 \textrm{ and }\Re(s+w)\geq 3/2+\delta
\end{cases}
\end{align*}
for any fixed $K>1/2$ and $\delta>0$.
\end{theorem}
\begin{remark} We shall see in the proof that the factor $(w-3/4)$
  is only necessary when $\chi'$ is trivial. We note also that the
  implied constant may depend on $\psi$. Moreover, the bounds given 
  above are not necessarily optimal. All we need for Theorem
  \ref{full-continuation} and Lemma
  \ref{approximate-functional-equation} below is polynomial control.
\end{remark}
\begin{proof}
We remark that the factor $ \zeta_2(4s-1) $ appearing in 
(\ref{therightseries}) does not have a pole in the region $ D_1 $. Thus 
we only have to study the series from (\ref{therightseries}) to prove the 
analytic properties of $ (w-3/4) Z(s, w, \chi, \chi') $.

We consider the regions where the series representation
(\ref{therightseries}) is absolutely convergent. We consider first the
sum over all \emph{non}-perfect squares ($n\ne m^2$).

If
$\Re(w)\geq 1/2$ (which corresponds to $\Re(2w-1/2)\geq 1/2$) 
we use
(\ref{RPC-average}), Theorem~\ref{average}, 
Proposition~\ref{correction-polynomials}~(\ref{prop3}), and 
Cauchy-Schwarz to see that away from $w=3/4$
\begin{equation}
  \label{bounding-coeffients}
\sum_{\substack{n\leq X\\ n\neq m^2  }}\abs{t_{ n}\chi(n)q(2w-1/2,
  n, \chi')L_2(2w-1/2,\chi_{n_0}\chi')}=O(X^{1+\varepsilon}\abs{w}^{1/4+\varepsilon}).
\end{equation}
It follows that the non-perfect square contribution is convergent for
$\Re(s-w)\geq 1/2+\delta$ and $ \Re (w) \geq 1/2 $ 
and that in the region $\Re(s-w)\geq 1/2+\delta$, $\Re(w)\geq 1/2$
it is analytic and
bounded by $O(\abs{w}^{1/4+\varepsilon})$. 

For $\Re(w)\leq 1/2$ we use Proposition~\ref{correction-polynomials}~(\ref{prop2}) 
and the functional equation for $L_2(2w-1/2,\chi_{n_0}\chi')$ to see that the product 
$q(2w-1/2,n,\chi')L_2(2w-1/2,\chi_{n_0}\chi')$ equals 
\begin{align*}
n^{1-2w} \left(\frac{\delta_{n_0,\chi'}}{n_0\pi}\right)^{1-2w} 
\frac{\Gamma\left(\frac{3/2-2w+\kappa_{\chi'}}{2}\right)}{\Gamma\left(\frac{2w-1/2+
\kappa_{\chi'}}{2}\right)}q(2(1-w)-1/2,n,\chi') L_2(2(1-w)-1/2,\chi_{n_0}\chi')
\end{align*}
times a factor $h_2(2w-1/2,n_0,\chi)/h_2(2(1-w)-1/2,n_0,\chi)$ which
is bounded when $\Re(w)\leq 1/2$ (recall \eqref{h2-def} for the
definition of $h_2$).
 We notice that $\delta_{n_0,\chi'}/n_0$ is 1, 4, or 8, and that in
 bounded $w$-strips the quotient of $\Gamma$-factors is
 $O(\abs{w}^{1-2\Re(w)})$. It follows that in bounded $w$-strips and
 for $\Re(w)\leq 1/2$ we have
\begin{align*}
  \label{anotherbounding-coeffients}
\sum_{\substack{n\leq X\\n \neq m^2}}&\abs{n^{2w-1}t_{ n}\chi(n)q(2w-1/2,
  n, \chi')L_2(2w-1/2,\chi_{n_0}\chi')}\\
&=O(\abs{w}^{1-2\Re(w)})\sum_{n\leq X}\abs{t_{ n}\chi(n)q(2(1-w)-1/2,
  n, \chi')L_2(2(1-w)-1/2,\chi_{n_0}\chi')}\\
&=O(\abs{w}^{1/4+1-2\Re(w) + \epsilon}X^{1+\varepsilon}),
\end{align*}
where in the last line we have used the same argument as used to bound
(\ref{bounding-coeffients}). It follows that when $\Re(s+w)\geq 3/2+\delta$, $\Re(w) 
\leq 1/2$ the non-square contribution from the 
series in (\ref{therightseries}) converges absolutely and that in this
region this contribution  is analytic and
bounded by $O(\abs{w}^{1/4+1-2\Re(w)+\varepsilon})$. 

We next consider the sum over all perfect squares $n=m^2$, 
\begin{equation*}
L_2(2w-1/2,\chi') \sum_{\substack{m=1\\(m,2)=1} }^\infty 
\frac{t_n\chi(n)q(2w-1/2,m^2,\chi')}{m^{2(s-w+1/2)}}. 
\end{equation*}
Using Proposition~\ref{correction-polynomials} and \eqref{RPC-average} we easily see that the
sum is convergent in $$\{(s, w):\Re(s-w)>0, \Re (s+w)>1\},$$ and that the
factor in front has a simple pole at $w=3/4$ if $\chi'$ is
trivial. That this contribution has the desired growth properties
follows from the convexity estimate on $L_2(2w-1/2,\chi')$.  

Having established that $(w-3/4)Z(s,w,\chi,\chi')$ is analytic in $D_1$, 
we now show that it satisfies a functional equation here. For $(s,w)$ in 
this region we use the functional equation \eqref{standard-functional-equation} 
and Proposition~\ref{correction-polynomials} and the subsequent discussion to 
see that $\frac{Z(s,w,\chi,\chi')}{\zeta_2(4s-1)}$ equals
\allowdisplaybreaks\begin{align*}
 \sum_{\substack{n=1\\(n,2)=1}}^\infty&\frac{t_{ n}\chi(n)q(2w-1/2,
  n, \chi')L_2(2w-1/2,\chi_{n_0}\chi')}{{n}^{s-w+1/2}} \\
=&\sum_{\substack{n=1\\(n,2)=1}}^\infty
 n^{1-2w} \left(\frac{\delta_{n_0,\chi'}}{n_0\pi}\right)^{1-2w}\frac{h_2(2w-1/2,n_0,\chi')}{h_2(2(1-w)-1/2,n_0,\chi')}
\frac{\Gamma\left(\frac{3/2-2w+\kappa_{\chi'}}{2}\right)}{\Gamma\left(\frac{2w-1/2+\kappa_{\chi'}}{2}\right)}\\
&\quad\quad 
\cdot \frac{t_{n}\chi(n)q(2(1-w)-1/2,
  n, \chi')L_2(2(1-w)-1/2,\chi_{n_0}\chi')}{{n}^{s-w+1/2}} \\
=&\pi^{2w-1}\frac{\Gamma\left(\frac{3/2-2w+\kappa_{\chi'}}{2}\right)}{\Gamma\left(\frac{2w-1/2+\kappa_{\chi'}}{2}\right)}\cdot \sum_{\substack{n=1\\(n,2)=1}}^\infty\frac{h_2(2w-1/2,n_0,\chi')}{h_2(2(1-w)-1/2,n_0,\chi')}
\left(\frac{\delta_{n_0,\chi'}}{n_0}\right)^{1-2w}
\\ 
& \quad\quad\cdot\frac{t_{ n}\chi(n)q(2(1-w)-1/2,
  n, \chi')L_2(2(1-w)-1/2,\chi_{n_0}\chi')}{{n}^{s-(1-w)+1/2}} .
\end{align*}
We split the sum according to $n$ mod~$8$, and notice that for fixed $\chi'$ 
the function $\frac{h_2(2w-1/2,n_0,\chi')}{h_2(2(1-w)-1/2,n_0,\chi')}
\left(\frac{\delta_{n_0,\chi'}}{n_0}\right)^{1-2w}$
is the same fraction of Dirichlet polynomials in $2^{-w}$ throughout each
of these sums, so that we can put them outside the sums. 
Using again that the indicator function of residue class
mod~$8$ can be written as a linear combination of characters mod~$8$ (at
least on the odd numbers) we arrive at the functional equation for $Z(s,w,\chi,\chi')$. 
We note that the factor $1-2^{-(3-4w)}$ is the product of all possible $h_2(2(1-w)-1/2,  n_0, \chi')$. 
This shows that the $p_{\chi''}(w)$ are in fact polynomials
  in $2^{-w}$.
\end{proof}

We now apply the same type of analysis to the second series
representation of $Z(s,w,\chi,\chi')$  given in
\eqref{full-intermediate-expression}.
Recall from Section~\ref{sec:gl2twists} that $g$ denotes
$(\psi\otimes\chi)^*$ where $\chi$ is one of the 4 characters mod~$8$. 
Let 
\begin{equation*}
  V(s,w)=\prod_{g}p_{2,g}(2^{-(w-s+1/2)})p_{2,g}(-2^{-(w-s+1/2)}).
\end{equation*}
where $p_{2,g}(z)$ is as in \eqref{p2poly}.

\begin{theorem}\label{second-functional-equation} 
  The function $(s-w-1/2)^2Z(s,w,\chi,\chi')$ is analytic in
  $$D_2=\{(s, w): \Re(s)>3/4, \Re(w)>3/4\},$$ and satisfies a functional
  equation $\beta: (s,w)\mapsto (w,s)$ given by
  \begin{align*}
 V(s&,w){ Z(s,w,\chi,\chi')}\\&=
 \sum_{\stackrel{k=0, 1} {\chi''
      \!\!\!\!\!\mod 8}} 
\prod_{\epsilon\in\{\pm 1\}}\frac{
\Gamma\left(\frac{1-(s-w+1/2)+k+\epsilon(s_0-1/2)}{2}\right)}{\Gamma\left(\frac{(s-w+1/2)+k+\epsilon(s_0-1/2)}{2}\right)}
   P_{\psi,\chi,\chi''}(s,w) Z(w,s,\chi,\chi'').
  \end{align*}
Here the $P_{\psi,\chi,\chi''}(s,w)$ are polynomials in
  $2^{-(s-w)}$. In particular they are functions bounded in vertical strips. 
Furthermore, away from $(s-w-1/2)=0$, 
\begin{equation*}
Z(s,w,\chi,\chi')=\begin{cases}O((\abs{s-w}+1)^{1/2+\varepsilon}), &\textrm{
    for } 3/4+\delta\leq \Re{w}\leq \Re(s)\leq K,\\
O((\abs{s-w}+1)^{3/2-2\Re(s-w+1/2+\varepsilon)}), &\textrm{
    for } 3/4+\delta \leq \Re{s}\leq \Re(w)\leq K,
\end{cases}
\end{equation*}
where $K$ is any constant with $K>3/4$ and any $\delta>0$. 
\end{theorem}
\begin{remark} We shall see in the proof that the factor $(s-w-1/2)^2$
  is only necessary when $\psi$ is $\hbox{GL}_1\times \hbox{GL}_1$ and
  $\chi$ is trivial.  We note also that the implied constant
  may depend on $\psi$. Moreover, as before the bounds given 
  above are not necessarily optimal as all we need for Theorem
  \ref{full-continuation} and Lemma
  \ref{approximate-functional-equation} below is polynomial control. 
\end{remark}
\begin{proof}
We now want to find the region of absolute convergence of   \eqref{full-intermediate-expression}. 
 Consider first the
region $\Re(s-w+1/2)\geq 1/2$. We can use Cauchy-Schwarz, Theorem~\ref{average}, and 
Lemma~\ref{Q-on-square-average} to see that the sum over non-perfect squares satisfies 
\begin{equation*}
\sum_{\substack{c\leq X\\c \neq r^2, c \textrm{ odd}}}\!\!\! \abs{\chi'(c)Q(s-w+1/2,c,\chi) L_2(s-w+1/2,\psi\otimes\tilde\chi_{c_0}\chi)} = O(X^{1+\varepsilon}(1+\abs{s-w})^{\frac{1}{2}+\varepsilon}). 
\end{equation*}
Hence the sum over these terms is absolutely convergent when
$\Re(2w-1/2)\geq 1+\delta$. 
The sum over the perfect squares potentially has a double pole at $s-w+1/2=1$: 
For $t_n=\tau(n)$ we
have $L_2(s,{\psi\otimes\chi_{0}^8 })=\zeta_2^2(s)$. The sum over perfect
squares is 
$$L_2(s-w+1/2,{\psi\otimes\chi })\sum_{\substack{c=1\\c =r^2}}^\infty\frac{\chi'(c)Q(s-w+1/2,c,\chi)}{c^{2w-1/2}},$$
where the sum   is again absolutely convergent
for $\Re(2w-1/2)\geq 1+\delta$, using Cauchy--Schwarz and Lemma~\ref{Q-on-square-average}.
It follows that, when $\Re(s-w+1/2)\geq 1/2$, the sums are
convergent for $\Re(w)\geq  3/4+\delta$, and hence $Z(s,w,\chi,\chi')$ is
analytic in this region except for a potential double polar line at
$s-w+1/2=1$. We also find that in this region we have the bound
$Z(s,w,\chi,\chi')=O((1+\abs{s-w})^{1/2+\varepsilon})$.

Turning now to $\Re(s-w+1/2)\leq  1/2$ we use the functional equation
\eqref{gl2-functional-equation} and 
Proposition~\ref{More-correction-polynomials}~\eqref{prop22} to move to a 
region where we can use the same bounds as for $\Re(s-w+1/2)\geq 1/2$: 

\allowdisplaybreaks\begin{align}
\nonumber & \frac{Z(s,w,\chi,\chi')}{\zeta_2(4s-1)\zeta_2(4w-1)} 
= \sum_{\substack{c=1\\(c,2)=1}}^\infty \frac{\chi'(c)Q(s-w+1/2,c,\chi) L_2(s-w+1/2, 
\psi\otimes\tilde\chi_{c_0}\chi)}{c^{2w-1/2}}\\
\label{bigfat-equation}&=\sum_{\substack{c=1\\(c,2)=1}}^\infty \frac{\chi'(c)c^{-2(s-w)}Q(1-(s-w+1/2),c,\chi) 
\epsilon{(\psi, \tilde\chi_{c_0}\chi)\left(\frac{N_1}{\pi^2}\right)^{-(s-w)}}}{c^{2w-1/2}}\\
\nonumber& \quad\quad
\times \frac{\Gamma\left(\frac{1-(s-w+1/2)+\kappa_{\chi,\psi,c_0}+(s_0-1/2)}{2}\right) 
\Gamma\left(\frac{1-(s-w+1/2)+\kappa_{\chi,\psi,c_0}-(s_0-1/2)}{2}\right)}{\Gamma\left(\frac{(s-w+1/2) 
+\kappa_{\chi,\psi,c_0}+(s_0-1/2)}{2}\right)\Gamma\left(\frac{(s-w+1/2)+\kappa_{\chi,\psi,c_0}-(s_0-1/2)}{2}\right)}\\
\nonumber &\quad\quad\times{\frac{H_2(s-w+1/2,g_1, c_0 )}{H_2(1-(s-w+1/2),g_1, c_0 )}} L_2(1-(s-w+1/2),\psi\otimes\tilde\chi_{c_0}\chi)
\end{align}
where $g_1=(\psi\otimes \tilde\chi_{c_0}\chi)^*$ with level
$N_1c_0^2$ where $N_1$ a divisor of 64 depending on $\chi, 
\psi$ (recall \eqref{def-H2} for the definition of $H_2$).  Using the same trick as before 
with splitting the sum into perfect squares and non-perfect squares, and using the 
bounds from  Lemma~\ref{Q-on-square-average} and Theorem~\ref{average}, as
well as the Stirling bound on the Gamma factors and a trivial bound on
the 2-factors we find that $Z(s,w,\chi,\chi')$ is analytic in
$$\{(s, w):\Re(s-w+1/2)\leq 1/2, \Re{s}\geq 3/4+\delta\}$$ and bounded as
$Z(s,w,\chi,\chi')=O(1+\abs{s-w}^{1/2+\varepsilon}\abs{s-w}^{1-2\Re(s-w+1/2)})$
for $\Re(s), \Re(w)$ bounded in this region.

We have established that $Z(s,w,\chi,\chi')$ is analytic in $D_2$.
 We now show that it also satisfies
a functional equation in this region. 
Consider \eqref{bigfat-equation}:
We noticed  that  $\epsilon{(\psi,\tilde\chi_0\chi)}$,   $\kappa_{\chi,\psi,c_0}$,
and $H_2(s,g_1, c_0 )$ depend only on $c_0$ modulo 8 (see Section \ref{sec:gl2twists}). We
split the sum into residue classes modulo 8 and
we can put these data outside the sum. Since $H_2(1-(s-w+1/2), g_1, c_0)$ can have zeros in
the region we multiply the left-hand side with all possible
expressions of it, which is $V(s,w)$ and arrive at the desired
functional equation.
\end{proof}
Using the two previous theorems we can now show that
$Z(s,w,\chi,\chi')$ admits a meromorphic continuation to all of $\C^2$.
\begin{theorem} \label{full-continuation}
  The function 
  \begin{equation}\label{zeta-with-polar}
  Z^*(s,w,\chi,\chi')=(s-w-1/2)^2(s+w-3/2)^2(w-3/4)(s-3/4)Z(s,w,\chi,\chi')
  \end{equation}  
  admits an analytic continuation to
  $(s,w)\in \C^2$ with at most polynomial growth for $\Re(s)$,
  $\Re(w)$ in bounded regions.  
\end{theorem}
\begin{proof}
We use repeatedly the functional
equations in Theorems~\ref{first-functional-equation} and
\ref{second-functional-equation}. We notice that these two theorems
show that  $Z^*(s,w,\chi,\chi')$ 
is analytic in the union of the two overlapping sets 
\begin{align*} 
D_1 &= \{(s, w): \Re(s-w)>1/2, \Re(s+w)>3/2\} \\ 
\intertext{and } 
D_2 &= \{(s, w):\Re(s)>3/4, \Re(w)>3/4\}.
\end{align*} 
since $(w-3/4)Z(s,w,\chi,\chi')$ is analytic in $D_1$ and
$(s-w-1/2)^2Z(s,w,\chi,\chi')$ is analytic in $D_2$.
\begin{figure}\begin{tikzpicture}
[scale=1,transform shape]
\draw[-stealth] (-1,0) -- (3,0);
\draw[](1,0.1)--(1,-0.1);
\node at (1,0.3) {${1}$};
\draw[-stealth] (0,-1) -- (0,3);
\draw[](0.1,1)--(-0.1,1);
\node at (0.3,1) {${1}$};
\node at (0.5,3) {${\Re(w)}$};
\node at (3,0.5) {${\Re(s)}$};

\draw[-] (2/2,1/2) -- (2/2+2/2,2/2+1/2);
\draw[-] (2/2, 1/2) -- (2/2+2/2,-2/2+1/2);
\draw[dashed] (2/2,1/2) -- (2/2+2/2,1/2);

\draw[-] (3/4,3/4) -- (3/4+2.828/2,3/4);
\draw[-] (3/4,3/4) -- (3/4,3/4+2.828/2);
\draw[dashed] (3/4,3/4) -- (1.414/2+3/4,1.414/2+3/4);
\end{tikzpicture}
 \caption{$D_1\cup D_2$}
\end{figure}
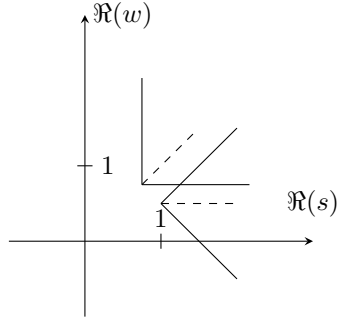
We now
use the group of functional equations generated by the two functional
equations
$$\alpha: (s,w)\mapsto (s,1-w), \quad\quad \beta: \quad(s,w)\mapsto (w,s).$$
They generate a group of order $8$ isomorphic to the dihedral group
$D_4$ of order $8$. We note that $\alpha^2=\beta^2=Id$. Using $\beta$
we see that $(s-3/4)Z(s,w,\chi,\chi')$ is
a holomorphic function of at most bounded polynomial growth (bounding
the ratio of Gamma functions using Stirling asymptotics) in
$D_3=\beta D_1$ which then extends $Z^*(s,w,\chi,\chi')$
to $D_1\cup D_2 \cup D_3$. We notice that the Gamma factor on the right-hand side 
of the functional equation in (\ref{second-functional-equation}) and $V(s,w)^{-1}$ does not 
have poles when $\Re(w-s)>0$ (by  \eqref{anotherboundagain} and properties
of the Gamma function).

We then use $\alpha$ to extend $Z^*(s,w,\chi,\chi')$ analytically to $D_1\cup
D_2 \cup\beta D_1\cup 
\alpha D_2 \cup\alpha \beta D_1$. We notice that the $2$ factor
$(1-2^{-(3-4w)})^{-1}$ and the 
Gamma factor in Theorem~\ref{first-functional-equation} are analytic
when $\Re(w)<3/4$.   The reflection $\alpha$ of the double polar line $s-w=1/2$ in $D_2$ 
produces the double polar line $s+w=3/2$ in $\alpha D_2$.
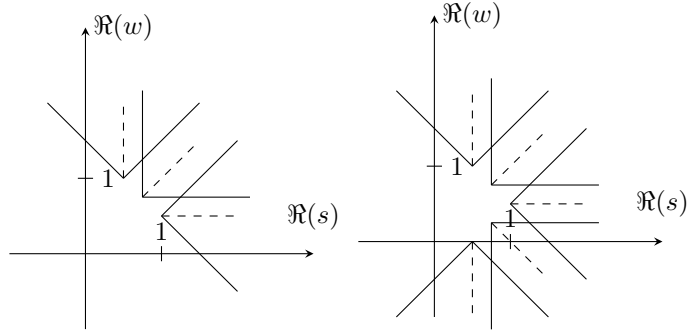
\begin{center}
\begin{tikzpicture}
[scale=1,transform shape]
\draw[-stealth] (-1,0) -- (3,0);
\draw[](1,0.1)--(1,-0.1);
\node at (1,0.3) {${1}$};
\draw[-stealth] (0,-1) -- (0,3);
\draw[](0.1,1)--(-0.1,1);
\node at (0.3,1) {${1}$};
\node at (0.5,3) {${\Re(w)}$};
\node at (3,0.5) {${\Re(s)}$};

\draw[-] (1/2,2/2) -- (1+1/2,2/2+2/2);
\draw[-] (1/2,2/2) -- (-1+1/2,2/2+2/2);
\draw[dashed] (1/2,2/2) -- (1/2,2/2+2/2);

\draw[-] (3/4,3/4) -- (3/4+2.828/2,3/4);
\draw[-] (3/4,3/4) -- (3/4,3/4+2.828/2);
\draw[dashed] (3/4,3/4) -- (1.414/2+3/4,1.414/2+3/4);

\draw[-] (2/2,1/2) -- (2/2+2/2,2/2+1/2);
\draw[-] (2/2,1/2) -- (2/2+2/2,-2/2+1/2);
\draw[dashed] (2/2,1/2) -- (2/2+2/2,1/2);
\end{tikzpicture}
\begin{tikzpicture}
[scale=1,transform shape]
\draw[-stealth] (-1,0) -- (3,0);
\draw[](1,0.1)--(1,-0.1);
\node at (1,0.3) {${1}$};
\draw[-stealth] (0,-1) -- (0,3);
\draw[](0.1,1)--(-0.1,1);
\node at (0.3,1) {${1}$};
\node at (0.5,3) {${\Re(w)}$};
\node at (3,0.5) {${\Re(s)}$};

\draw[-] (1/2,2/2) -- (2/2+1/2,2/2+2/2);
\draw[-] (1/2,2/2) -- (-2/2+1/2,2/2+2/2);
\draw[dashed] (1/2,2/2) -- (1/2,2/2+2/2);

\draw[-] (3/4,3/4) -- (3/4+2.828/2,3/4);
\draw[-] (3/4,3/4) -- (3/4,3/4+2.828/2);
\draw[dashed] (3/4,3/4) -- (1.414/2+3/4,1.414/2+3/4);
\draw[-] (2/2,1/2) -- (2/2+2/2,2/2+1/2);
\draw[-] (2/2,1/2) -- (2/2+2/2,-2/2+1/2);
\draw[dashed] (2/2,1/2) -- (2/2+2/2,1/2);

\draw[-] (1/2,1-2/2) -- (2/2+1/2,1-2/2-2/2);
\draw[-] (1/2,1-2/2) -- (-2/2+1/2,1-2/2-2/2);
\draw[dashed] (1/2,1-2/2) -- (1/2,1-2/2-2/2);

\draw[-] (3/4,1-3/4) -- (3/4,1-3/4-2.828/2) ;
\draw[-] (3/4,1-3/4) -- (3/4+2.828/2,1-3/4);
\draw[dashed] (3/4,1-3/4) -- (1.414/2+3/4,1-1.414/2-3/4)  ;
\end{tikzpicture}
\end{center}

The regions $D_4=\beta\alpha D_2$,  $D_5=\beta\alpha\beta D_1$, and 
$D_6=\alpha\beta\alpha D_2=\alpha D_4$ can be dealt with using 
Theorems~\ref{first-functional-equation} and \ref{second-functional-equation} 
in the same way and no new polar lines are introduced, neither due 
to the $2$ factors, nor the Gamma factors. 

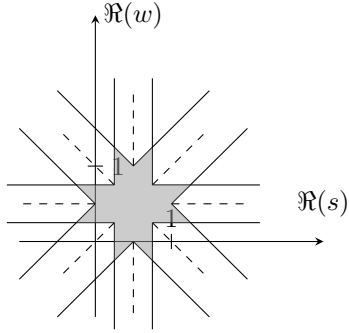
\begin{center}
\begin{tikzpicture}
[scale=1,transform shape]
\draw[-stealth] (-1,0) -- (3,0);
\draw[](1,0.1)--(1,-0.1);
\node at (1,0.3) {${1}$};
\draw[-stealth] (0,-1) -- (0,3);
\draw[](0.1,1)--(-0.1,1);
\node at (0.3,1) {${1}$};
\node at (0.5,3) {${\Re(w)}$};
\node at (3,0.5) {${\Re(s)}$};

\draw[-] (1/2,2/2) -- (2/2+1/2,2/2+2/2);
\draw[-] (1/2,2/2) -- (-2/2+1/2,2/2+2/2);
\draw[dashed] (1/2,2/2) -- (1/2,2/2+2/2);

\draw[-] (3/4,3/4) -- (3/4+2.828/2,3/4);
\draw[-] (3/4,3/4) -- (3/4,3/4+2.828/2);
\draw[dashed] (3/4,3/4) -- (1.414/2+3/4,1.414/2+3/4);

\draw[-] (2/2,1/2) -- (2/2+2/2,2/2+1/2);
\draw[-] (2/2,1/2) -- (2/2+2/2,-2/2+1/2);
\draw[dashed] (2/2,1/2) -- (2/2+2/2,1/2);

\draw[-] (1/2,1-2/2) -- (2/2+1/2,1-2/2-2/2);
\draw[-] (1/2,1-2/2) -- (-2/2+1/2,1-2/2-2/2);
\draw[dashed] (1/2,1-2/2) -- (1/2,1-2/2-2/2);

\draw[-] (1-2/2,1/2) -- (1-2/2-2/2,2/2+1/2);
\draw[-] (1-2/2,1/2) -- (1-2/2-2/2,-2/2+1/2);
\draw[dashed] (1-2/2,1/2) -- (1-2/2-2/2,1/2);

\draw[-] (1-3/4,3/4) -- (1-3/4-2.828/2,3/4);
\draw[-] (1-3/4,3/4) -- (1-3/4,3/4+2.828/2);
\draw[dashed] (1-3/4,3/4) -- (1-1.414/2-3/4,1.414/2+3/4);

\draw[-] (1-3/4,1-3/4) -- (1-3/4-2.828/2,1-3/4);
\draw[-] (1-3/4,1-3/4) -- (1-3/4,1-3/4-2.828/2);
\draw[dashed] (1-3/4,1-3/4) -- (1-1.414/2-3/4,1-1.414/2-3/4);

\draw[-] (3/4,1-3/4) -- (3/4+2.828/2,1-3/4);
\draw[-] (3/4,1-3/4) -- (3/4,1-3/4-2.828/2);
\draw[dashed] (3/4,1-3/4) -- (1.414/2+3/4,1-1.414/2-3/4);

\path[fill=black!40, semitransparent](1/2,2/2) --(3/4,5/4)--(3/4,3/4) --(5/4,3/4)
--(2/2,1/2) --(5/4,1/4)--(3/4,1/4)--(3/4,-1/4)--(1/2,-0/2)--(1/4,-1/4)
--(1/4,1/4)--(-1/4,1/4)--(-0/2,1/2)--(-1/4,3/4)--(1/4,3/4)-- (1/4,5/4)-- cycle;
\end{tikzpicture}
\end{center}
The function in (\ref{zeta-with-polar}) is now extended to a holomorphic function on the
complement of the domain with tube given by the shaded region. It is bounded polynomially 
for $\Re(w)$, $\Re(s)$ bounded. We can therefore use Bochner's tube theorem
(see \cite{DiaconuGoldfeldHoffstein:2003a}, Propositions~4.6 and 4.7 and the argument on 
p.~341) to extend the holomorphic function
to the convex hull of this region (which is $\C^2$) with at most polynomial bounds for
$(\Re(s),\Re(w))$ in compact sets. Therefore, $Z(s, w, \chi, \chi')$ has the same properties, 
apart from being meromorphic with the specified polar lines in (\ref{zeta-with-polar}).
\end{proof}

\begin{remark}
Combining Theorems~\ref{first-functional-equation} and
\ref{second-functional-equation} we note that 
$\alpha\circ \beta\circ\alpha\circ\beta:(s,w)\mapsto(1-s,1-w)$, and it
follows that there exist functions
$\alpha_{\rho,\rho',\chi,\chi'}(s,w)$ bounded in vertical strips such that
\begin{align}  
 \nonumber
F(s, w) & {Z}(s,w,\chi,\chi') = \\ 
& 
\label{double-functional-eq} \sum_{\overline k \in\{0,1\}^4}\frac{G(1-s,1-w,\overline k)}{G(s,w,\overline
  k)}\sum_{\rho,\rho' \!\!\!\!\mod 8}\alpha_{\overline k, \rho,\rho',
  \chi,\chi'}(s,w) Z (1-s,1-w,\rho,\rho') 
\end{align}
where 
\begin{align}\nonumber
G(s,w,\overline k)
:=\Gamma&\left(\frac{2w-1/2+k_1}{2}\right)\prod_{\epsilon_1\in \{\pm 1\}} 
\Gamma\left(\frac{s+w-1/2+k_2+\epsilon_1(s_0-1/2)}{2}\right)\\
&\cdot\Gamma\left(\frac{2s-1/2+k_3}{2}\right)\prod_{\epsilon_2\in \{\pm 1\}} 
\Gamma\left(\frac{s-w+1/2+k_4+\epsilon_2(s_0-1/2)}{2}\right)
\end{align} 
and 
\begin{align*} 
F(s, w) & :=  
\big(1-2^{-(3-4w)}\big) \big(1-2^{-(3-4s)}\big) V(s,w)V(w,1-s).  
\end{align*} 
Using $\frac{\Gamma(\frac{1-z+1}{2})}{\Gamma(\frac{z+1}{2})}=\frac{\Gamma(\frac{1-z}{2})}{\Gamma(\frac{z}{2})}\cot(\frac{\pi
  z}{2})$ we see that $$\frac{G(1-s,1-w,\overline k)}{G(s,w,\overline k)}=\frac{G(1-s,1-w,0)}{G(s,w,0)}\cot_{\overline k}(s,w)$$
where 
\begin{align*}
\cot_{\overline k}(s,w) &= \cot^{k_1} \left(\frac{\pi(2w-1/2)}{2}\right) 
\prod_{\e_1\in\{\pm 1\}} \cot^{k_2} \left(\frac{\pi(s+w-1/2+\e_1(s_0-1/2))}{2}\right) \\
& \quad \cdot \cot^{k_3}\left(\frac{\pi(2s-1/2)}{2}\right)\prod_{\epsilon_2\in\{\pm 1\}} 
\cot^{k_4}\left(\frac{\pi(s-w+1/2+\epsilon_2(s_0-1/2))}{2}\right). 
\end{align*}
Since away from poles of $\cot$ we have uniform bounds
$\cot\left(\frac{\pi z }{2}\right)=i \textup{sign}(y)+O(e^{-\pi y})$, we see that
$\cot_{\overline k}(s,w)$ is bounded in vertical strips (for the
arguments away from the poles of $\cot_{\overline k}$). It follows that
the functional equation \eqref{double-functional-eq} can be written
simply as 
\begin{equation}
  \label{simplified-double-functional-eq} 
F(s,w){Z}(s,w,\chi,\chi')=
\frac{G(1-s,1-w,0)}{G(s,w,\overline 0)} 
\sum_{\substack{\rho,\rho' \mod 8\\ \overline k \in \{0,1\}^4}} 
\beta_{\overline k, \rho,\rho',\chi,\chi'}(s,w) Z (1-s,1-w,\rho,\rho')
\end{equation}
where the functions $\beta_{\overline k, \rho,\rho', \chi,\chi'}(s,w)$ 
are bounded in vertical strips (away from any poles).
\end{remark}

\subsection{Bounds on $Z(s,w,\chi,\chi')$}

In this section we bound $Z(s,w,\chi,\chi')$ when
$\Re(s)=\Re(w)=1/2$. Recall that we defined (See
\eqref{conductor-def}) the analytic conductor to be
\begin{equation}\label{analytic-conductor}\mathfrak
q(t,u):=(1+\abs{t})(1+\abs{t+u})^{2}(1+\abs{u}).\end{equation}

\begin{theorem}\label{convexity-independent}
Assume \eqref{lindelof-on-average}. Then
$$
Z(1/2+it,1/2+iu,\chi,\chi')=O(\mathfrak q(t,u)^{1/4+\varepsilon}).$$
Unconditionally $$
Z(1/2+it,1/2+iu,\chi,\chi')=O(\mathfrak (q(t,u)(1+\abs{t-u})^2)^{1/4+\varepsilon}).$$
\end{theorem}

We call the bound obtained in Theorem \ref{convexity-independent} the
\emph{convexity} bound. Any bound of the form
$O(\mathfrak q(t,u)^{1/4-\delta})$ is called a \emph{subconvex}
bound. 

To prove Theorem~\ref{convexity-independent} we first prove an
approximate functional equation similar to the one in \cite[Lemma~4.2]{BlomerGoldmakherLouvel:2014}. 
\begin{lemma}\label{approximate-functional-equation} 
Let $t,u\in\R$ and $\chi$, $\chi'$  mod~$8$. There exist smooth functions
  $W_{\pm}:\R_+\to\C$ depending on $u,t$, and the characters
  satisfying  $$y^j\frac{d^j}{dy^j}W_{\pm}(y)=O (1+y)^{-A}$$
for all $j, A\in \mathbb N_0$, uniformly in $u,t$, such that 
$$ Z(1/2+it,1/2+iu,\chi,\chi')=\sum_{\rho, \rho' \!\!\!\!\mod
  8}\sum_{\pm}\sum_{c=1}^\infty
{\frac{\rho'(c)L^{**}(1/2\pm i(t-u),\psi,
    c,\rho)}{c^{1/2\pm
        2iu}}W_\pm\left(\frac{c}{\sqrt{\mathfrak q(t,u)}}\right)}.$$
\end{lemma}
\begin{proof} Recall that $1/\cos(z)$ is holomorphic in
  $\abs{\Re(z)}< \pi/2$ and satisfies
 $ 1/\cos(z)=O_{\epsilon_0}(e^{-\abs{z}})$ for  $\abs{\Re(z)}\leq
 \pi/2-\epsilon_0$. For  $\eta (\log 2)/(\pi i)$ bounded away from
 $\mathbb Z$ the function
 $P_\eta(z)=(1-2^{\eta-z})(1-2^{\eta+z})/(1-2^\eta)^2$ is uniformly bounded
 in vertical strips, holomorphic in $\C$, even in $z$, with a simple
 zero at $\eta$, and satisfies $P_\eta(0)=1$. For a given multiset
 $B$ let 
$$H_{B}(z)=(\cos(\frac{\pi z}{3A}))^{-12A}\prod_{\eta \in B}P_\eta (z)$$
which is $O(e^{-4\pi\abs{z}})$ for, say, $\abs{\Re(z)}\leq
(3/2-\delta)A$ with $\delta>0$ sufficiently small.
For an appropriate choice of multiset $B=B_{t,u}$ we set $ H_{t,u}(z) = 
H_{B_{t,u}}(z) $ so that the integrand of 
\begin{equation}\label{integral-afe}\frac{1}{2\pi i}\int_{(1)}
  \frac{F(s+z,w+z)}{F(s,w)}Z(s+z,w+z,\chi,\chi')\frac{G(s+z,w+z,0)}{G(s,w,0)} 
  H_{t,u}(z)\frac{dz}{z}\end{equation}
is holomorphic in the entire $z$-plane except for a simple pole 
at $z=0$. (The function $ H_{t,u} $ has been used to remove the poles 
of $ Z(s+z, w+z, \chi, \chi') $.) Also it has rapid decay in $z$ on
vertical lines due to Theorem~\ref{full-continuation}. Moving the line of 
integration to $\Re(s)=-1$ we see that \eqref{integral-afe} equals 
\begin{equation*} Z(s,w,\chi,\chi')+\frac{1}{2\pi i}\int_{(-1)}\!\!\!\frac{F(s+z,w+z)}{F(s,w)} Z(s+z,w+z,\chi,\chi')\frac{G(s+z,w+z,0)}{G(s,w,0)}H_{t,u}(z)\frac{dz}{z}.\end{equation*}
Using the functional equation \eqref{simplified-double-functional-eq} and the 
change of variable $z\mapsto -z$ the last integral equals 
\begin{align*} 
\sum_{\substack{\rho,\rho' \!\!\!\!\mod 8\\\overline k \in\{0,1\}^4}} 
\frac{1}{2\pi i} \int_{(1)} & \frac{\beta_{\overline k, \rho,\rho',\chi,\chi'} 
(s-z,w-z)}{F(s,w)} Z(1-s+z,1-w+z,\rho,\rho') \\ 
&\times\frac{G(1-s+z,1-w+z,0)}{G(s,w,0)}H_{t,u}(z)\frac{dz}{z}.
\end{align*}
It follows that there exist functions $ \gamma_{\overline k,\rho, 
\rho',\chi,\chi',\pm} (x, x') $ bounded if $ \Re(x)  
= \Re(x') = -1/2 $
(note that using \eqref{anotherboundagain} we see that $F(s,w)^{-1}$ is uniformly bounded) 
such that $ Z(\frac{1}{2}+it,\frac{1}{2}+iu,\chi,\chi') $ equals
\begin{align*}
\sum_{\substack{\pm \\ \rho,\rho' \!\!\!\!\mod 8 \\ \overline k \in\{0,1\}^4}} 
\frac{1}{2\pi i} \int_{(1)} \gamma_{\overline k, \rho,\rho',\chi,\chi',\pm} 
\Big(\frac{1}{2}\pm i t-z, & \frac{1}{2}\pm i u-z\Big)  
Z\Big(\frac{1}{2}\pm i t+z,\frac{1}{2}\pm i u+z,\rho,\rho'\Big) \\
& \times \frac{G\big(\frac{1}{2}\pm i t+z,\frac{1}{2}\pm i u+z, 
0\big)}{G\big(\frac{1}{2}+ i t,\frac{1}{2}+ i u,0\big)} H_{t,u}(z) \frac{dz}{z}.
\end{align*}
Using the series representation \eqref{another-nice-representation} we
arrive at the result with $W_\pm(y)$ equal to 
\begin{align*} 
\sum_{\overline k\in\{0,1\}^4} \frac{1}{2\pi i} \int_{(1)}  
\gamma_{\overline k, \rho,\rho',\chi,\chi',\pm} & \Big(\frac{1}{2}\pm i t-z,\frac{1}{2}\pm i u-z\Big) 
\big(y\sqrt{C'(t,u)}\big)^{-2z} \\ 
& \times \frac{G\big(\frac{1}{2}\pm i t+z,\frac{1}{2}\pm i u+z,0\big)}{G\big(\frac{1}{2} 
+ i t,\frac{1}{2}+ i u,0\big)} H_{t,u}(z) \frac{dz}{z}. 
\end{align*}
From Stirling's formula we find that
$\Gamma(s+z)/\Gamma(s)=O((1+\abs{s})^{\Re(z)}e^{\pi\abs{z}/2})$ 
uniformly for $s, z$ in  bounded strips away from poles. It follows
that
we have 
$$ 
\frac{G\big(\frac{1}{2}\pm i t+z,\frac{1}{2}\pm i 
u+z,0\big)}{G\big(\frac{1}{2}+ i t,\frac{1}{2}+ i u,0\big)} 
=O\big(\mathfrak q(u,t)^{\Re{(z)}}e^{2\pi \abs{z}}\big). 
$$
By shifting the contour to $\sigma$ and differentiating under the integral sign we see that
$$ 
y^j\frac{\partial^jW_\pm}{\partial^j y} 
=O\left(y^{-{2\sigma}}\int_{(\sigma)}e^{(-4\pi+2\pi)\abs{z}}\frac{(1+\abs{z})^j}{\abs{z}}dz+\delta_{j=0,\sigma<0}\right) 
$$
for $-\delta\leq \sigma<(3/2-\delta)A$. The last term comes from the
pole at $z=0$. For $y\leq 1$ we can choose $\sigma=-\delta/2$, and for
$y>1$ we choose $\sigma=A$ and find the desired bound.
\end{proof}

We can now prove Theorem~\ref{convexity-independent}:

\begin{proof} 
Let $\varepsilon >0$. For $\Re(z)=1/2$ we have, assuming \eqref{lindelof-on-average}, 
\begin{equation}\label{bound-on-coeff}\sum_{\substack{c\leq Y\\
      c\textrm{ odd} }}\abs{L^{**}(z,\psi,c,\rho)}=O(Y^{1+\e}(1+\abs{z})^{a+\e})\end{equation}
with $a=0$. Unconditionally
\eqref{bound-on-coeff} holds with $a=1/2$ as is straightforward to
verify from Lemma \ref{Q-on-square-average},
Theorem~\ref{average}~\eqref{generalized-average}, and
Cauchy-Schwartz.

It follows that for an appropriate choice of $A$
 in Lemma~\ref{approximate-functional-equation} we have 
$$\sum_{c>\mathfrak q(u,t)^{1/2+\varepsilon}} \!\!\!\!\!\!\!\!\! 
\frac{\abs{L^{**}(1/2\pm
    i(t-u),\psi,c,\rho)}}{c^{1/2}}\abs{W_{\pm}\left(\frac{c}{\sqrt{\mathfrak
        q(t,u)}}\right)}\leq
C_\varepsilon ((1+\abs{t-u})^{a}\mathfrak q(u,t)^{\e}).$$
It follows also that
$$\sum_{c\leq \mathfrak q(u,t)^{1/2+\varepsilon}}  
\frac{\abs{L^{**}(1/2\pm i(t-u),\psi,c,\rho)}}{c^{1/2}}=O(\mathfrak q(u,t)^{1/4+\e}(1+\abs{t-u})^{a+\e}).$$

The claim of Theorem~\ref{convexity-independent} now follows from the approximate functional equation. 
\end{proof}

\begin{remark}We notice that for the special configuration $w=1-s$ the
  conductor drops to essentially
$$(1+\abs{t})(1+\abs{u}).$$
This configuration will be the relevant one in Theorem
\ref{cuspidalcontribution} below.
\end{remark}

\begin{remark}One could speculate whether using another
  functional equation could lead to a smaller conductor.
During the proof of Theorem \ref{convexity-independent}, or more
precisely in the proof of the approximate functional equation Lemma \ref{approximate-functional-equation}, we have
made certain choices: we have chosen a particular
functional equation $(s,w)\to(1-s,1-w)$ and a particular series
representation \eqref{another-nice-representation}. In principle, there is
nothing that prohibits  running the same type
of argument with the other series representation \eqref{therightseries}
and/or another functional equation.  

Let us consider what happens if we make other choices.
If we use 
\eqref{therightseries} and if $\Re(z)=1$ and
$\Re(s)=\Re(w)=1/2$ then the function   $Z(s+z,w,\chi,\chi')$ in
\eqref{integral-afe} is evaluated in $D_1$ where the
series representation \eqref{therightseries} is convergent. Similarly
if we consider  \eqref{another-nice-representation} and  if $\Re(z)=1$ and
$\Re(s)=\Re(w)=1/2$ then the function 
$Z(s+z,w+z,\chi,\chi')$ is evaluated in $D_2$ where the
series representation \eqref{another-nice-representation} is
convergent. In  order for the argument in Lemma \ref{approximate-functional-equation} to work we need to use a functional equation $\gamma:\C^2\to\C^2$
with the property that, when $\Re(z)=1$ and
$\Re(s)=\Re(w)=1/2$, the numbers $\gamma(s-z,w-z)$/$\gamma(s-z,w)$  lie in $D_1$ or $D_2$. Only in this case is the integrand evaluated where
the double Dirichlet series has a series representation (after
moving the line of integration to $\Re(z)=-1$, using the functional equation and
making a change of variable $z\to -z$).

When we are using \eqref{another-nice-representation}  we assume
\eqref{lindelof-on-average}. When we are using
\eqref{therightseries} we make the similar assumption for this
series, namely that 
$$\sum_{\substack{n\leq X\\(n,2)=1}}\abs{t_{n}L^*(w,n,\chi')}=O(X^{1+\e}(1+\abs{w})^\e)\quad\textrm{ for }\Re(w)=1/2.$$
With these restrictions we list  the possible {\lq analytic conductors\rq} in Table \ref{table-conductors}.
\begin{table}[ht]
\caption{Different choices of analytic conductors}\label{table-conductors} 
\centering
\begin{tabular}{|c||c||c|}\hline
  Functional equation & Series repn. &Analytic conductor\\
  \hline \hline
  
$\beta\alpha\beta:(s,w)\to(1-s,w)$ & \eqref{therightseries}
&$(1+\abs{t+u})^2(1+\abs{t-u})^2(1+\abs{t})^2$\\
$\alpha\beta\alpha:(s,w)\to(1-w,1-s)$ & \eqref{another-nice-representation} &$(1+\abs{t})(1+\abs{t+u})^2(1+\abs{u})$\\
$\alpha\beta\alpha\beta:(s,w)\to(1-s,1-w)$ & \eqref{therightseries}
&$(1+\abs{t+u})^2(1+\abs{t-u})^2(1+\abs{t})^2$\\
$\alpha\beta\alpha\beta:(s,w)\to(1-s,1-w)$ & \eqref{another-nice-representation} &$(1+\abs{t})(1+\abs{t+u})^2(1+\abs{u})$\\\hline
\end{tabular}
\end{table} 
   Since for all $t,u\in \R$
$$(1+\abs{t})(1+\abs{t+u})^2(1+\abs{u})\leq (1+\abs{t+u})^2(1+\abs{t-u})^2(1+\abs{t})^2,$$
the conductor defined in \eqref{analytic-conductor} is the smallest
among these.
    \end{remark}

\subsection{Another double Dirichlet series}
It turns out that there is another double Dirichlet series which is
relevant in the applications to QUE. We now define it and then immediately 
show that it can be understood in terms of the series
$Z(s,w,\chi,\chi')$ which was analyzed in the previous sections. Let
\begin{equation}\label{onceagainaseries} 
\hat Z(s,w,\chi,\chi')=\sum_{\substack{c=1\\(c,2)=1}} 
\frac{\chi'(c)L^*(s-w+1/2,c,\chi)^2}{c^{2w-1/2}}. 
\end{equation}
In order to understand $\hat Z(s,w,\chi,\chi')$ we exhibit an 
interesting non-trivial relation between the $q$-polynomials
and the $Q$-polynomials in the case of the Eisenstein 
series, i.e.\ for $ t_n = \tau(n) $. Let $\hat Q$ be 
defined as $Q$ but with the one exception that we use 
$\chi_{c_0}$ instead of $\tilde\chi_{c_0}$, i.e.\ with $v=v_p(c_1)$,
\begin{equation*}
\hat Q(s,c_0c_1^2,\chi) = \prod_{p\vert c_1} \frac{t_{p^{2v}}-t_{p^{2v-1}} 
\chi_{c_0}(p) \chi(p)\left(\frac{p^{1-s}+p^{s}}{p}\right)+{t_{p^{2v-2}} 
{\chi_{c_0}(p)^2}/{p}}}{p^{2v(s-1/2)}},
\end{equation*}
defined for $c_0$, $c_1$ odd.
By \eqref{quadrep-characters} we see that 
\begin{equation} \label{relation-right}
  \hat Q(s,c,\chi)=\begin{cases}Q(s,c,\chi),&\textrm{ if }c_0\equiv 1
    (4),\\Q(s,c,\chi\chi_4),&\textrm { if }c_0\equiv 3 (4).\end{cases}
\end{equation}
\begin{lemma} \label{nontrivialrelation} 
Let $d_0$ be  an odd squarefree positive integer,
$d_1$ odd, and $t_n=\tau(n)$. Then 
\begin{equation*}
\sum_{d\vert d_1}d^{1-2s} \left(q(s,d_0\frac{d_1^2}{d^2},\chi)\right)^2 
= \sum_{d\vert d_1} \sigma_{2-4s}(d) \hat Q\bigg(s,d_0\frac{d_1^2}{d^2},\chi\bigg).
  \end{equation*}
\end{lemma}
\begin{proof} Since the involved arithmetical functions are
  multiplicative, it  is enough to verify the
  claim on prime powers $d_1=p^n$, i.e.\ we need to verify
$$
\sum_{i=0}^np^{2i(1/2-s)}q^2\big(s,d_0p^{2(n-i)}\big)
=\sum_{i=0}^n\sum_{j=0}^ip^{4j(1/2-s)}Q\big(s,d_0p^{2(n-i)},\chi\big). 
$$
Using the definitions of $q(s,d,\chi)$ and $Q(s,d,\chi)$ it is a
straightforward but tedious algebraic computation with sums and
products of geometric sums. The details are omitted.
\end{proof}
Using the above lemma we can now show that many properties of $\hat
Z(s,w,\chi,\chi')$ can be understood on the basis of the properties of
$Z(s,w,\chi,\chi')$. The following lemma implies in particular that
$\hat Z(s,w,\chi,\chi')$ admits a meromorphic continuation, and that any
bound we have on $Z_{\psi_\tau}(s,w,\chi,\chi')$ translates into a bound
for $\hat Z(s,w,\chi,\chi')$.

\begin{lemma}\label{ZhatZtilde} 
Assume that $\psi=\psi_\tau$, i.e.\ $t_n=\tau(n)$. Then 
\begin{align*} \hat
Z(s,w,\chi,\chi') =& \frac{1}{2\zeta_2(2s+2w-1)}
\big(Z_{\psi_\tau}(s,w,\chi,\chi')+Z_{\psi_\tau}(s,w,\chi\chi_4,\chi') \\ 
& \phantom{\frac{1}{2\zeta_2(2s+2w-1)} \big(} 
+Z_{\psi_\tau}(s,w,\chi,\chi'\chi_4)-Z_{\psi_\tau}(s,w,\chi\chi_4, 
\chi'\chi_4)\big).
\end{align*}
\end{lemma}
\begin{proof}
  We start by noticing that 
$L_2(s,\chi_{c_0}\chi)^2=L_2(s,\psi_\tau\otimes\chi_{c_0}\chi)$.
 Let now $d_0$
be an odd squarefree natural number. Then 
\begin{align*}
\zeta_2(2s+2w-1)\sum_{\substack{d_1=1\\d_1 \textrm{odd}}}^\infty 
\frac{q^2(s,d_0d_1^2,\chi)}{d_1^{2w}} 
&= \sum_{\substack{d, d_1=1\\d_1, \, d \textrm{ odd}}}^\infty 
\frac{d^{1-2s}q^2(s,d_0d_1^2,\chi)}{(dd_1)^{2w}}\\
&=\sum_{\substack{d_1=1\\d_1 \textrm{ odd}}}^\infty 
\frac{\sum_{d\vert d_1}d^{1-2s}q^2(s,d_0\frac{d_1^2}{d^2},\chi)}{d_1^{2w}}.
\end{align*}
We then use Lemma \ref{nontrivialrelation} and arrive at
\begin{align*}
  \sum_{\substack{
      d_1=1\\d_1 \textrm{
        odd}}}^\infty&\frac{\sum_{d\vert d_1}\sigma_{2-4s}(d)\hat
    Q(s,d_0\frac{d_1^2}{d^2},\chi)}{d_1^{2w}}=\sum_{\substack{l=1 \\l \textrm{
        odd}}}^\infty\frac{\sigma_{2-4s}(l)}{l^{2w}}\sum_{\substack{
      d_1=1\\d_1 \textrm{
        odd}}}^\infty\frac{\hat
    Q(s,d_0d_1^2,\chi)}{d_1^{2w}}\\
&=\zeta_2(4s+2w-2)\zeta_2(2w)\sum_{\substack{
      d_1=1\\d_1 \textrm{
        odd}}}^\infty\frac{\hat
    Q(s,d_0d_1^2,\chi)}{d_1^{2w}}.
\end{align*}
Multiply the first and last expression by
$\frac{\chi'(d_0)L_2(s,\chi_{d_0}\chi)^2}{d_0^w}$ and summing over all odd
squarefree natural numbers $d_0$ we get 
\begin{align*}
  \zeta_2(2s&+2w-1)\sum_{\substack{d=1\\(d,2)=1}}^\infty\frac{\chi'(d)q^2(s,d,\chi)L_2(s,\chi_{d_0}\chi)^2}{d^{w}}\\
&=\zeta_2(4s+2w-2)\zeta_2(2w)\sum_{\substack{
      d=1\\(d,2)=1}}^\infty\frac{\chi'(d)\hat
    Q(s,d,\chi)L_2(s,\psi_\tau\otimes\chi_{d_0}\chi)}{d^{w}}.
\end{align*}
By \eqref{quadrep-characters}  we see that 
$$\hat
    Q(s,d,\chi)L_2(s,\psi_\tau\otimes\chi_{d_0}\chi)= 
\begin{cases}
  Q_{\psi_\tau}(s,d,\chi)L_2(s,\psi_\tau\otimes\chi_{d_0}\chi),&\textrm{
  if }d\equiv 1 (4),\\
Q_{\psi_\tau}(s,d,\chi\chi_4)L_2(s,\psi_\tau\otimes\chi_{d_0}\chi\chi_4),&\textrm{
  if }d\equiv 3 (4).
\end{cases}$$
Substituting $(s-w+1/2,2w-1/2)$ for $(s, w)$ and comparing with 
\eqref{full-intermediate-expression} we obtain the desired result.
\end{proof}

\section{Eisenstein series}\label{Sec:weight-half-eisensteinseries}
We briefly recall a few facts about Eisenstein series with weights.  
For $\gamma\in \hbox{SL}_2(\R)$ and $z\in \H$ we define $j(\gamma,z)=cz+d$ 
and $j_\gamma(z)=\frac{cz+d}{\abs{cz+d}}$. We let $\arg$ denote the
principal argument and define $j_\gamma(z)^k=e^{ik\arg(cz+d)}$. Since
$j(\gamma_1\gamma_2,z)=j(\gamma_1,\gamma_2z)j(\gamma_2,z)$ 
$$ 
\tilde \omega(\gamma_1,\gamma_2) = 
\frac{1}{2\pi}\left(\hbox{arg}j(\gamma_1,\gamma_2z) 
+ \hbox{arg}j(\gamma_2,z)-\hbox{arg}j(\gamma_1\gamma_2,z)\right) 
$$
is an integer independent of $z$.
The factor system of weight $k\in \R$ is then defined as 
$$ 
\omega(\gamma_1,\gamma_2)=e(k\tilde\omega(\gamma_1,\gamma_2)). 
$$ 
Then we have $ \omega(\gamma_1,\gamma_2)j_{\gamma_1\gamma_2}(z)^{k} 
=j_{\gamma_1}(\gamma_2z)^{k}j_{\gamma_2}(z)^{k}$. We refer to 
\cite[Ch.~2.6, Ch.~3]{Iwaniec:1997a} for the basic properties 
of multiplier systems, as well as for further explanations of 
the generalities of Fourier expansions.

Let $\nu$ be a weight $k$ multiplier system, and let $\G$ be a cofinite subgroup 
of ${\hbox{SL}_2( {\mathbb R})} $. 
For an open cusp $\mathfrak a$, i.e.\ $\nu(\mathfrak a)=1$, we 
define the weight $k$ Eisenstein series for $\G$ by
\begin{equation*}
E_{\mathfrak a}(z,s,k) := \sum_{\gamma\in\Gamma_\mathfrak{a}\backslash\Gamma} 
\overline{\nu(\g)\omega(\sigma_{\mathfrak a}^{-1},\gamma)} j_{\sigma_{\mathfrak a}^{-1}\g}(z)^{-k} 
\Im(\sigma_{\mathfrak a}^{-1} \g z)^s \quad  \textrm { for $\Re(s)>1$,}
\end{equation*}
where $\sigma_{\mathfrak a}$ is a scaling matrix of the cusp 
$\mathfrak a$, i.e.\ $\sigma_{\mathfrak a}^{-1} \Gamma_{\mathfrak a} 
\sigma_{\mathfrak a} = \Gamma_\infty $, $ \Gamma_\infty $ 
being generated by $ \gamma_\infty = \left(\begin{smallmatrix} 
1 & 1 \\ 0 & 1 
\end{smallmatrix}\right) $ and $ - \gamma_\infty $ if $ -I \in \Gamma $. 
This function satisfies $E_{\mathfrak a}(\g z,s,k) = \nu(\g) j_\g^{k}(z) 
E_{\mathfrak a}(z,s,k)$ for $ \gamma \in \G $, is an eigenfunction of the 
weight $k$ Laplacian with eigenvalue $s(1-s)$, and admits a meromorphic
continuation to $s\in\C$. We now briefly recall how to find the Fourier 
coefficients of $ E_{\mathfrak a}(z,s,k)$ at an open cusp $\mathfrak b$. 
We have 
$$ 
j_{\sigma_{\mathfrak b}}(z)^{-k}E_{\mathfrak a}(\sigma_{\mathfrak b} z,s,k) 
= \sum_{\gamma\in\Gamma_\infty\backslash\sigma_{\mathfrak a}^{-1}\Gamma\sigma_{\mathfrak b}} 
\overline{\nu_{{\mathfrak a}{\mathfrak b}}(\gamma)}j_{\gamma}(z)^{-k}\Im(\gamma z)^s, 
$$ 
where
$\nu_{{\mathfrak a}{\mathfrak b}}(\gamma) = \nu(\sigma_{\mathfrak a} \gamma 
\sigma_{\mathfrak b}^{-1}) \omega(\sigma_{\mathfrak a}^{-1}, \sigma_{\mathfrak a} 
\gamma \sigma_{\mathfrak b}^{-1}) \omega(\gamma\sigma_{\mathfrak b}^{-1}, 
\sigma_{\mathfrak b})$. For the rest of the paper we can assume that $-I\in\Gamma$.
Summing over a set of representatives of
$\G_\infty\backslash\sigma_{\mathfrak a}^{-1} \Gamma\sigma_{\mathfrak b}/\G_\infty$,
which we can assume have $c_\g>0$ for $\g\not \in \Gamma_\infty$, we see that 
$$
j_{\sigma_{\mathfrak b}}(z)^{-k}E_{\mathfrak a}(\sigma_{\mathfrak b} z,s,k) 
= \delta_{{\mathfrak a}={\mathfrak b}}y^s 
+ \sum_{I\neq \gamma\in\G_\infty\backslash\sigma_{\mathfrak a}^{-1}\G\sigma_{\mathfrak b} /\G_\infty} 
\overline{\nu_{{\mathfrak a}{\mathfrak b}}(\g)}\sum_{l\in \Z}j_{\g\g_\infty^l}(z)^{-k} 
\Im(\g\gamma^l_\infty z)^s. 
$$
Therefore, by a familiar computation, we have
\allowdisplaybreaks\begin{align*} 
\int_0^1(&j_{\sigma_{\mathfrak b}}(z)^{-k}E_{\mathfrak a}(\sigma_{\mathfrak b}z,s,k) 
- \delta_{{\mathfrak a}={\mathfrak b}}y^s)e(-nx)dx\\
&=\sum_{I\neq \gamma\in\G_\infty\backslash\sigma_{\mathfrak a}^{-1}\Gamma\sigma_{\mathfrak b}/\G_\infty} 
\frac{\overline{\nu_{{\mathfrak a}{\mathfrak b}}(\g)}}{c^{2s}}e\left(n\frac{d}{c}\right) 
y^s\int_{-\infty}^\infty\left(\frac{z}{\abs{z}}\right)^{-k}\frac{1}{\abs{z}^{2s}}e(-nx)dx.
\end{align*}
Substituting $t=x/y$ in the last integral we see that 
\begin{align*}y^s\int_{-\infty}^\infty
\left(\frac{z}{\abs{z}}\right)^{-k}&\frac{e(-nx)}{\abs{z}^{2s}}dx
=y^{1-s}\int_{-\infty}^\infty \left(\frac{t+i}{\abs{t+i}}\right)^{-k} 
\frac{e(-nty)}{\abs{t+i}^{2s}}dt\\
&=e^{-ik\pi/2}y^{1-s}\int_{-\infty}^\infty
\left(\frac{1-it}{\abs{1-it}}\right)^{-k}\frac{e(-nty)}{\abs{1+it}^{2s}}dt\\
&=
\begin{cases}
\pi^se^{-ik\pi/2}\frac{\abs{n}^{s-1}}{\Gamma(s+\frac{kn}{2\abs{n}})} 
W_{\frac{kn}{2\abs{n}},s-1/2}(4\pi \abs{n}y),&\textrm{ if }n\neq 0,\\
\pi 4^{1-s}e^{-ik\pi/2}\frac{\G(2s-1)}{\G(s+k/2)\G(s-k/2)} y^{1-s},&\textrm{ if }n=0, 
\end{cases}
\end{align*}
where $W_{\mu,\nu}(y)$ is the Whittaker function, and where we have used 
\cite[3.384 (9), p.~349]{GradshteynRyzhik:2007a} for $n\neq 0$ and
\cite[p.~84--85]{Shimura:1975a} for $n=0$.

\subsection{Eisenstein series of level 4}
We now specialize to $\Gamma=\Gamma_0(4)$. In this case the Fourier 
coefficients of half-integral weight
Eisenstein series were originally studied by Shimura \cite{Shimura:1975a}.
We consider
the weight $1/2$ multiplier system
$\nu$ related to the theta series 
\begin{equation*}
  \theta(z):=y^{1/4}\sum_{m\in\Z}e(m^2z), 
\end{equation*} i.e.\ $\theta(\g z)=\nu(\g)j_\g(z)^{1/2}\theta(z)$ 
for $\g\in\G$. It is well known
that 
$$ 
\nu\left(\g\right)=\jacobi{c}{d}\varepsilon_d^{-1} \textrm{ for } 
\begin{pmatrix}a & b\\c& d\end{pmatrix}=\gamma\in\Gamma_0(4). 
$$
Here the Jacobi-Legendre symbol is extended as in \cite[p.~442]{Shimura:1973a}.
The group $\Gamma_0(4)$ has 3 cusps $\mathfrak{a}_1=\infty$,
$\mathfrak{a}_2=0$, $\mathfrak{a}_3=1/2$, with corresponding 
stabilizers $\Gamma_{\mathfrak{a}_i}$ generated by $ \pm
\gamma_{\mathfrak{a}_i}$ where 
$$
\gamma_{{\mathfrak{a}_1}}=\begin{pmatrix}1&1\\0&1\end{pmatrix},\quad
\gamma_{\mathfrak{a}_2}=\begin{pmatrix}1&0\\-4&1\end{pmatrix},\quad
\gamma_{\mathfrak{a}_3}=\begin{pmatrix}-1&1\\-4&3\end{pmatrix} 
$$
and we define scaling matrices 
$$ 
\sigma_{\mathfrak{a}_1}=\begin{pmatrix}1&0\\0&1\end{pmatrix}, \quad
\sigma_{\mathfrak{a}_2}=\begin{pmatrix}0&-1/2\\2&0\end{pmatrix}, \quad 
\sigma_{\mathfrak{a}_3}=\begin{pmatrix}1&-1/2\\2&0\end{pmatrix}. 
$$
Only the cusps $\infty$ and $0$ are open with respect to 
$\nu$ as 
$$ 
\nu(\gamma_{\mathfrak{a}_1})=\nu(\gamma_{\mathfrak{a}_2})=1, \quad 
\nu(\gamma_{\mathfrak{a}_3})=\jacobi{-4}{3}\epsilon^{-1}_3=i. 
$$

We now compute the Fourier expansion for the weight $1/2$ Eisenstein
series. We focus on the cusp at infinity but the analysis for the 
other cusps is similar, although slightly more technical. The main 
extra complication at the other cusps comes from the factor system. 
This can be dealt with as follows: For $k=1/2$ we can use 
$z=\gamma_{2}^{-1}i$ in the definition of the factor system to see that 
$$ 
\omega(\gamma_1,\gamma_2)= 
\begin{cases} 
1,&\textrm{if }
-\pi< \hbox{arg}(c_{\gamma_1}i+d_{\gamma_1})+\hbox{arg}(c_{\gamma_2}i+a_{\gamma_2}) 
\leq \pi,\\-1,&\textrm{otherwise.} 
\end{cases} 
$$ 
Using the properties of a multiplier system one finds (see
\cite[(3.5)]{Iwaniec:1997a}) that 
$$
\nu_{\mathfrak a \mathfrak b}(\gamma)=
\nu(\sigma_{\mathfrak{a}}\gamma\sigma_{\mathfrak{b}}^{-1})
\frac{\omega(\sigma_{\mathfrak{a}}\gamma\sigma_{\mathfrak{b}}^{-1}, 
\sigma_{\mathfrak{b}})}{\omega(\sigma_{\mathfrak{a}},\gamma)}. 
$$
This is explicit enough that one can do the computations also 
for the other cusps.

We now focus on 
$({\mathfrak{a}_1},{\mathfrak{a}_1})=(\infty,\infty)$, and omit 
the corresponding subscripts. Using that all the non-identity
elements of $\G_\infty\backslash\G/\G_\infty$ are parametrized by  
$\begin{pmatrix}
   *&*\\4c&d
 \end{pmatrix}$ with $c>0$,  $d$ mod~$4c$,
 $(d,4c)=1$, we find that 
$$E(z,s,1/2)=y^s+\phi(s,1/2)y^{1-s}+\sum_{n\neq
0}\phi_n(s,1/2)W_{\frac{n}{\abs{n}}\frac{1}{4},s-1/2}(4\pi \abs{n}y)e(nx)$$
with 
\begin{align}\label{fourier-vanilla} 
\phi_n(s,1/2) &= \frac{\pi^se^{-i\pi/4}\abs{n}^{s-1}}{\Gamma(s+\frac{n}{4\abs{n}})} 
\sum_{c=1}^\infty\frac{1}{(4c)^{2s}}\sum_{\substack{d\!\!\!\mod {4c}\\
 (d,4c)=1}}{\overline{\nu\begin{pmatrix}*&*\\4c&d\end{pmatrix}}e(nd/4c)}\\
\nonumber&=\frac{\pi^se^{-i\pi/4}\abs{n}^{s-1}}{\Gamma(s+\frac{n}{4\abs{n}})} 
\sum_{c=1}^\infty\frac{1}{(4c)^{2s}} \sum_{\substack{d\!\!\!\mod 4c}}{\varepsilon_d 
\jacobi{4c}{d}e(nd/4c)},
\end{align}
and
\begin{align*}
\phi(s,1/2)= \frac{\pi 4^{1-s}e^{-i\pi/4}\G(2s-1)}{\G(s+1/4)\G(s-1/4)} 
\sum_{c=1}^\infty\frac{1}{(4c)^{2s}}\sum_{\substack{d\!\!\!\mod 4c}}{\varepsilon_d 
\jacobi{4c}{d}}.
\end{align*}
If we write $4c=2^kc'$ with $c'$ odd
then Sturm
proved \cite[Lemma 1]{Sturm:1980} -- using quadratic reciprocity and
the Chinese remainder theorem --
that 
\begin{equation}\label{sturms-identity}\sum_{\substack{d\!\!\mod
    4c}}{\varepsilon_d\jacobi{4c}{d}e(nd/4c)}=H_n(c')\sum_{r
\!\!\!\mod 2^k}\jacobi{2^k}{r}\varepsilon_re(nr/2^k).\end{equation}
It follows that for $ n \not= 0 $ 
$$
\phi_n(s,1/2) = \frac{\pi^se^{-i\pi/4}\abs{n}^{s-1}}{\Gamma(s+\frac{n}{4\abs{n}})} 
\sum_{\substack{c'=1\\(c',2)=1}}^\infty \frac{H_n(c')}{c'^{2s}} \sum_{k=2}^\infty 
\frac{\sum_{r\!\!\!\mod 2^k}\jacobi{2^k}{r}\varepsilon_re(nr/2^k)}{2^{2ks}}, 
$$
which by Lemma \ref{goodsplitting} equals 
\begin{equation} \label{good-expression}
\frac{\pi^se^{-i\pi/4}\abs{n}^{s-1}}{\Gamma(s+\frac{n}{4\abs{n}})} 
\frac{L^*(2s-1/2,n,1)}{\zeta_2(4s-1)} r_2(s,n),
\end{equation}
where we have written 
\begin{equation} \label{r2neu} 
r_2(s,n) := \sum_{k=2}^\infty\frac{\sum_{r \!\!\!\mod 2^{k}} \jacobi{2^{k}}{r} 
\varepsilon_re(nr/2^{k})}{2^{k2s}}. 
\end{equation} 
The function $r_2(s,n)$ can also be computed. One uses that
$\varepsilon_d$ can be expressed as a sum of characters mod~$4$ as  
$$\varepsilon_d=\frac{1+i}{2}\chi_4^0(d)+\frac{1-i}{2}\chi_4(d).$$
Inserting this in \eqref{r2neu} the numerator becomes 
\begin{equation}\label{bad-primes-splitting} 
\frac{1+i}{2}G_n(\chi_8^k\chi_{2^k}^0)+\frac{1-i}{2}G_n(\chi_8^k\chi_4\chi_{2^k}^0), 
\end{equation}
where $\chi_8$ is the primitive character mod~$8$ given by
$\chi_8(n)=(-1)^{\frac{1}{8}(n-1)(n+1)}$ for $(n,2)=1$, and $G_n$ denotes
the usual Gauss sum. Using \cite[Lemma 3]{Shimura:1975a} as well as
explicit computations of $G_1(\chi_1)$, $G_1(\chi_8)$, $G_1(\chi_4)$,
$G_1(\chi_4\chi_8)$  these can
all be computed and using the result one can compute $r_2(s,n)$. We omit the
details but state the result: Assume first $n\not \equiv  0 (4)$. Then
\begin{equation}\label{isuseful}
  r_2(s,n)=\frac{1+i}{4}\begin{cases}-\frac{1}{2^{2(2s-1)}},&n\not
    \equiv 1 (4),\\
\frac{1}{2^{2(2s-1)}}+\frac{\chi_8(n)\sqrt{2}}{2^{3(2s-1)}}, &n
    \equiv 1 (4).\\
\end{cases}
\end{equation}
More generally we find that if $n=4^rn_0$ where $n_0\not \equiv 0
(4)$, then 
\begin{equation}\label{isalsouseful}
r_2(s,n)=\frac{(1+i)}{4}u_r(2^{-(2s-1)})+4^{-r(2s-1)}r_2(s,n_0),
\end{equation}
where
\begin{equation}\label{def-u}u_r(x)=\frac{(x^2)^{r+1}-x^2}{x^2-1}.\end{equation}
We remark that $r_2(s, n)$ is entire.

\subsubsection{Scattering term}
We now compute the scattering term $\phi(s,1/2)$, which by 
\eqref{sturms-identity} equals
$$ 
\frac{\pi 4^{1-s}e^{-i\pi/4}\G(2s-1)}{\G(s+1/4)\G(s-1/4)}  
\sum_{\substack{c'=1\\(c',2)=1}}^\infty \frac{H_0(c')}{c'^{2s}} \sum_{k=2}^\infty 
\frac{\sum_{r\!\!\!\mod 2^k}\jacobi{2^k}{r}\varepsilon_r}{2^{2ks}}.
$$
The sum
$\displaystyle\sum_{\substack{c'=1\\(c',2)=1}}^\infty\frac{H_0(c')}{c'^{2s}}$
factors, and for an odd prime $p$ we observe that $$H_0(p^\beta)=\begin{cases} 
\varphi(p^\beta),& \textrm{ if }\beta \equiv 0 (2),\\0,&\textrm{
  otherwise.}\end{cases}$$ Here $\varphi$ is Euler's $\varphi$-function. Therefore 
\begin{align*} 
\sum_{\beta=0}^\infty & \frac{H_0(p^\beta)}{p^{\beta 2s}} 
=\sum_{\beta=0}^\infty \frac{\varphi(p^{2\beta})}{p^{2\beta 2s}} 
=\frac{\zeta^{(p)}(4s-2)}{\zeta^{(p)}(4s-1)}. 
\end{align*}
For the prime $2$ we note that for $k\geq 2$ we have
\begin{equation*}G_0(\chi_4\chi_{2^{k}}^0)=G_0(\chi_8\chi_{2^{k+1}}^0)=G_0(\chi_4\chi_8\chi_{2^{k+1}}^0)=0\end{equation*}
Using this we find, 
\begin{align*}\sum_{k=2}^\infty&\frac{\sum_{r
\!\!\!\mod 2^{k}}\jacobi{2^{k}}{r}\varepsilon_r}{2^{k2s}}\\ 
&=\sum_{\substack{k=2\\ k \equiv 0 (2)}}^\infty\frac{\frac{1+i}{2}G_0(\chi_4^0\chi_{2^k}^0)+\frac{1-i}{2}G_0(\chi_4\chi_{2^k}^0)}{2^{k2s}}+\sum_{\substack{k=2\\
    k \equiv 1 (2)}}^\infty\frac{\frac{1+i}{2}G_0(\chi_8\chi_{2^k}^0)+\frac{1-i}{2}G_0(\chi_8\chi_4\chi_{2^k}^0)}{2^{k2s}}\\
&=\sum_{\substack{k=2\\ k \equiv 0
    (2)}}^\infty\frac{\frac{1+i}{2}\varphi(2^k)}{2^{k2s}}={(1+i)}\frac{2^{-4s}}{1-2^{-(4s-2)}}.
\end{align*}
It follows that 
$$\phi(s,1/2)=\pi 4^{1-s}e^{-i\pi/4}\frac{\G(2s-1)
}{\G(s+1/4)\G(s-1/4)}\frac{(1+i)}{2^{4s}}\frac{\zeta(4s-2)}{\zeta_2(4s-1)}.$$
Using that $\G(s+1/4)\G(s-1/4)=\sqrt{\pi}2^{3/2-2s}\Gamma(2s-1/2)$
this simplifies to\begin{equation}\label{weight-1/2-scattering}
\frac{1}{2^{4s-1}-1}\frac{\xi(4s-2)}{\xi(4s-1)}
\end{equation}
where $\xi(s)=\pi^{-s/2}\Gamma(s/2)\zeta(s)$ (compare
\cite[p.~247-248]{Iwaniec:1997a}). The other entries in the scattering
matrix $\Phi(s,1/2)$ can be computed in a similar way and we find
\begin{equation} \label{scattering-explicit} 
  \Phi(s,1/2)=\begin{pmatrix}\frac{2^{-(4s-1)}}{1-2^{-(4s-2)}}&\frac{1-i}{2^{2s}}\\
 \frac{1+i}{2^{2s}} & \frac{2^{-(4s-1)}}{1-2^{-(4s-2)}}\end{pmatrix}\frac{1-2^{-(4s-2)}}{1-2^{-(4s-1)}}\frac{\xi(4s-2)}{\xi(4s-1)}.
\end{equation}
As a consistency check we note that a direct computation and the
functional equation for $\xi$ show that the scattering matrix 
verifies $\Phi(s,1/2)\Phi(1-s,1/2)=I$ as predicted by the general theory.

\subsection{Eisenstein series of level $2^n$.}
We now consider the group $\Gamma_0(N)$, where $N=2^n$ with $n\geq 2$. Let $\chi$
be a Dirichlet character modulo $N$, and consider the weight 1/2 multiplier
system
$$\nu(\gamma)=\chi(d)\jacobi{c}{d}\varepsilon_d^{-1} \textrm{ for } 
\begin{pmatrix}a & b\\c& d\end{pmatrix}=\gamma\in\Gamma_0(N). $$
We consider the corresponding Eisenstein series of weight 1/2 at the
cusp at 0, denoted by $$E_{0,\chi}(z,s,1/2).$$
Similarly one denotes $E_{\infty, \chi}(z, s, 1/2)$  the  corresponding Eisenstein series at the cusp $\infty$.
 The Fourier coefficients at infinity of the Eisenstein series at zero has a simpler 2-factor than the
Eisenstein series at infinity. The stabilizer at 0 is generated by
$\pm\gamma_0$ and has corresponding scaling matrix $\sigma_{0}$ where
\begin{equation*}
\gamma_0=\begin{pmatrix}1&0\\-2^n&1\end{pmatrix}, \quad \sigma_{0}=\begin{pmatrix}0&-1/\sqrt{2^{n}}\\\sqrt{2^{n}}&0\end{pmatrix}.
\end{equation*}
From the general considerations in the beginning of section \ref{Sec:weight-half-eisensteinseries} we find that the non-zero Fourier coefficients at infinity
equal
\begin{align*}
  \sum_{I\neq \gamma\in\G_\infty\backslash\sigma_{
      0}^{-1}\Gamma_0(N)/\G_\infty}
  \frac{\overline{\nu_{{ 0}{\infty}}(\g)}}{c^{2s}}e\left(n\frac{d}{c}\right)\pi^se^{-i\pi/4}\frac{\abs{n}^{s-1}}{\Gamma(s+\frac{n}{4\abs{n}})} 
W_{\frac{n}{4\abs{n}},s-1/2}(4\pi \abs{n}y).
\end{align*}
After some computations one finds 
\begin{align}
 \nonumber \sum_{I\neq \gamma\in\G_\infty\backslash\sigma_{
      0}^{-1}\Gamma_0(N)/\G_\infty}
  \frac{\overline{\nu_{{
          0}{\infty}}(\g)}}{c^{2s}}e\left(n\frac{d}{c}\right)&=\frac{i\chi(-1)}{N^{s}}\sum_{\substack{a=1\\(a,2)=1}}^\infty\frac{\chi(a)H_n(a)}{a^{2s}}\\
\label{willitend}&
=\frac{i\chi(-1)}{N^{s}}\frac{L^*(2s-1/2,n,\chi)}{\zeta_2(4s-1)},
\end{align}
where in the last equality we have used Lemma \ref{goodsplitting}.
Using this it is straightforward to see how $Z(s,w,\chi,\chi')$
relates directly to a Rankin-Selberg integral in the case where
$\{t_n\}$ comes from a cusp form. Let $\psi$ be a cuspidal Hecke newform of
weight zero, and trivial multiplier for
$\Gamma_0(2^k)$ with eigenvalue $s_0(1-s_0)$ 
and Fourier expansion
\begin{equation} \label{Hecke-Maass}
\psi (z)=\sum_{n\neq0} b_nW_{0,s_0-1/2}(4\pi \abs{n}y)e(nx). 
\end{equation} 
Let $\chi$ be a Dirichlet character mod 8. Consider the
twisted Maa\ss{} form
$$\psi\otimes \chi (z)=\sum_{n\neq0} \chi(n)b_nW_{0,s_0-1/2}(4\pi
\abs{n}y)e(nx),$$ which is a weight zero cusp form for some
$\Gamma_0(M)$ and character $\chi_0^M$  for some $ M\vert \hbox{lcm}(64,2^k)$ and $8\vert M$. Let $\chi'$ be another Dirichlet character mod 8. Consider now the Rankin-Selberg integral
\begin{align*}
I(\psi,\chi,\chi',s,w)=  \int_{\Gamma_0(M)\backslash \H}\psi\otimes
  \chi(z)E_{0,\chi_0^M\chi'}(z,w,1/2)\overline{E_{\infty,\chi_0^M\chi'}(z,\overline
    s,1/2)}d\mu(z).
\end{align*}
This is the integral studied by Friedberg and Hoffstein (See
\cite[(1.2), p.~388]{FriedbergHoffstein:1995a}).

Unfolding, using
$b_n=b_{\frac{n}{\abs{n}}}{\abs{n}}^{-1/2}t_{\abs{n}}$,
\eqref{willitend}, and $L^*(s,-n,\chi)=L^*(s,n,\chi_4\chi)$ we arrive at
\begin{align}
 \nonumber I(\psi,\chi,\chi',s,w)&=\frac{\pi^we^{-i\pi/4}i\chi'(-1)}{(2\pi)^{s-1}M^{w}\zeta_2(4w-1)}\sum_{\substack{n\neq
  0\\ (n,2)=1}}\frac{\chi(n)b_{\frac{n}{\abs{n}}}t_{\abs{n}}L^*(2w-1/2,n,\chi')}{\abs{n}^{s-w+1/2}}G_{\frac{n}{\abs{n}}}(w)
  \\ \label{friedberghoffstein}&=\frac{\pi^we^{-i\pi/4}i\chi'(-1)}{(2\pi)^{s-1}M^{w}\zeta_2(4w-1)\zeta_2(4s-1)}\\\nonumber
  &\quad\quad\times
\left[Z(s,w,\chi,\chi')G_+(w)+\chi(-1)b_{-1}Z(s,w,\chi,\chi_4\chi')G_-(w)\right],     
\end{align}
where
\begin{equation*}
  G_\pm(w)=\frac{1}{\Gamma(w\pm \frac{1}{4})}\int_{0}^\infty W_{\pm 1/4,w-1/2}(2y)  W_{0,s_0-1/2}(2y) y^{w-1}\frac{dy}{y}.
\end{equation*}
\begin{lemma}\label{bound-on-integrals}
$$  I(\psi,\chi,\chi',1/2+it,1/2+iu)=O(\log((2+\abs{t})(2+\abs{u})))$$
\end{lemma}
\begin{proof}
  This follows from the Maa\ss-Selberg relation, and known properties  of the relevant scattering matrix.
\end{proof}
It is tempting to speculate whether the above bound on $
I(\psi,\chi,\chi',1/2+it,1/2+iu)$ can be used to bound
$Z(s,w,\chi,\chi')$ through \eqref{friedberghoffstein}. What we can
prove is the following:

Denote the expression in the square brackets of
\eqref{friedberghoffstein} by $\tilde
I(\psi,\chi,\chi',s,w)$. We then find that
\begin{align}\nonumber\tilde I(\psi,&\chi,\chi',s,w)\pm\tilde
I(\psi,\chi,\chi_4\chi',s,w)\\ &=(Z(s,w,\chi,\chi')\pm
Z(s,w,\chi,\chi_4\chi'))(G_+(w)\pm \chi(-1)b_{-1}G_-(w)).\label{linearcomb-integrals}\end{align}

\begin{lemma} \label{lincombZs}Assume that $\psi$ is a cusp form. Then for $s=1-w=1/2+it$
\begin{equation}\label{MS-bound}Z(s,w,\chi,\chi')+\chi(-1)b_{-1}Z(s,w,\chi,\chi_4\chi')=O((1+\abs{t})^{1/2+\e}).\end{equation}
\end{lemma}
\begin{proof}From \eqref{linearcomb-integrals} we see that 
$$Z(s,w,\chi,\chi')+\chi(-1)b_{-1}Z(s,w,\chi,\chi_4\chi')(G_+(w)+G_-(w))$$
equals $$\tilde I(\psi,\chi,\chi',s,w)+\chi(-1)b_{-1}\tilde I(\psi,\chi,\chi_4\chi',s,w).$$
 The claim now follows from Lemma~\ref{bound-on-integrals} and
 Lemma~\ref{nicoleresult}, combined with Remark~\ref{order-of-growth}.
\end{proof}
\begin{remark}\label{finalremarks}
  We notice that with the restriction above on $s,w$  the conductor
$\mathfrak q(t,-t)$ is of the order $(1+\abs{t})^2$. So the right-hand side in 
\eqref{MS-bound} is of the order 
$\mathfrak{q}(t,-t)^{1/4+\e}$, i.e. for the linear combination $Z(s,w,\chi,\chi')+\chi(-1)b_{-1}Z(s,w,\chi,\chi_4\chi')$ we have  proved the convexity
estimate  unconditionally. Surprisingly this {\lq soft\rq} method of using the Maa{\ss}--Selberg relations gives much stronger bounds than the harder method using Heath-Brown's Theorem \ref{heath-brown-average} and approximate functional equations. Unfortunately we do not know how to
prove this unconditionally for
$Z(s,w,\chi,\chi')$ and $Z(s,w,\chi,\chi_4\chi')$  separately. The main reason for this is that 
$G_+(w)-G_-(w)$ decays much faster than $G_+(w)+G_-(w)$ so using a similar argument on
$$Z(s,w,\chi,\chi')-\chi(-1)b_{-1}Z(s,w,\chi,\chi_4\chi')$$ gives very
poor
bounds. 
\end{remark}

If we use \eqref{another-nice-representation} 
(i.e. interchange sums) we find, like \cite[(1.2) p.~389]{FriedbergHoffstein:1995a}), that $\tilde I(\psi,\chi,\chi',s,w)$ equals
\begin{align*}
  \sum_{\substack{c=1\\(c,2)=1}}&\frac{\chi'(c)L^{**}(s-w+1/2,\psi,c,\chi)}{c^{2w-1/2}}(G_+(w)+\chi(-1)b_{-1}\chi_4(c)G_-(w))
.
\end{align*}
By taking linear combinations over different $\chi'$ we can restrict
to $c$ in a specific residue class, as  in the work of
Friedberg and Hoffstein \cite{FriedbergHoffstein:1995a}.

\section{Limits of weight 1/2 Eisenstein series}\label{sec:main-theorem}
We consider
separately Maa\ss{} cusp forms and incomplete Eisenstein series,
i.e. analyze 
$$\int_\GmodH\psi(z) \abs{E(z,1/2+it,1/2)}^2d\mu(z),$$ where $\psi$ is
either a Maa\ss{} cusp form or an incomplete Eisenstein series.  
Then a standard approximation argument, see \cite[p.~217]{LuoSarnak:1995a}, implies the result \eqref{againagain}.

\subsection{The cuspidal contribution.}
Let $\psi$ be a cuspidal element of a weight zero Hecke basis for
$\Gamma_0(4)$ with eigenvalue $s_0(1-s_0)$ 
and Fourier expansion
\begin{equation*} 
\psi (z)=\sum_{n\neq0} b_nW_{0,s_0-1/2}(4\pi \abs{n}y)e(nx). 
\end{equation*} 
We will freely use that we can assume that the Fourier coefficients
are real.

We want to study 
\begin{equation*}
  \int_{\Gamma\setminus\H}\psi (z)\abs{E(z,s,1/2)}^2 d\mu(z) 
\end{equation*}
when  $\Re(s)=1/2$. It turns out to be convenient to  consider the 
slightly more general integral 
\begin{equation*}
I(s,w)=\int_{\Gamma\setminus\H}\psi (z)E(z,w,1/2)
\overline{E(z,\overline s,1/2)}d\mu(z). 
\end{equation*}
For sufficiently large $\Re (s)$, we can unfold to get
\begin{equation}
I(s, w)=\int_{\Gamma_\infty\setminus\H}\psi(z)E(z,w,1/2) y^s d\mu(z).
\end{equation} 
Using the Fourier expansions of $\psi$ and $E_\infty(z,w,1/2)$ 
and computing the $x$-integral we find
\begin{align}
I(s,w) &= 
\int_0^{\infty} \sum_{n\neq 0}b_n \phi_{-n}(w,1/2) W_{0,s_0-1/2}(4\pi \abs{n}y) 
W_{\frac{-n}{\abs{n}}\frac{1}{4},w-1/2}(4\pi \abs{n}y) y^{s-1}\frac{dy}{y} \nonumber \\
\label{longago} 
&=\sum_{n\neq 0}\frac{b_n \phi_{-n}(w,1/2)}{(2\pi\abs{n})^{s-1}} 
\int_0^{\infty}W_{0,s_0-1/2}(2y)W_{\frac{-n}{\abs{n}}\frac{1}{4},w-1/2}(2y) y^{s-1}\frac{dy}{y}.
\end{align} 
We consider the series 
$$ 
Z_\pm(s,w) := 
\frac{\Gamma(w\mp\frac{1}{4})}{\pi^we^{-i\pi/4}}\zeta_2(4s-1)\zeta_2(4w-1) 
\sum_{\pm n=1}^\infty\frac{b_{n}\phi_{- n}(w,1/2)}{\abs{n}^{s-1}}. 
$$
By \eqref{good-expression} we see that
\begin{equation}
  \label{double-dirichlet-first-representation}
Z_\pm(s,w) = 
\zeta_2(4s-1) \sum_{\pm n=1}^\infty\frac{b_{ n}r_2(w,- n)L^*(2w-1/2,-n,1)}{\abs{n}^{s-w}}  .
\end{equation}
The next proposition reduces many questions about $Z_\pm(s,w)$ to
questions about $Z(s,w,\chi,\chi')$. 
Consider the Dirichlet polynomial 
$$ 
T(s,w) := \prod_{\e\in\{\pm 1\}} p_2\big(\e 2^{-(s+w-1/2)}\big) 
p_2\big(\e 2^{-(s-w+1/2)}\big) 
$$
where $p_2(z)$ is defined in \eqref{p2poly}.
\begin{proposition}\label{simpler-series-translation} 
There exist functions $f_\pm(s,w, \chi,\chi')$ bounded in vertical strips such that 
  \begin{equation*}
  T(s,w)  Z_\pm(s,w)=\sum_{\chi,\chi'}f_\pm(s,w,\chi,\chi')Z(s,w,\chi,\chi'),
  \end{equation*}
where the sum is over all pairs of characters mod~$8$. 
\end{proposition}
\begin{proof}
We first assume that $\psi$ is a newform. Then we have 
$$ 
b_n=b_{\frac{n}{\abs{n}}}{\abs{n}}^{-1/2}t_{\abs{n}}, 
$$
where $\{t_n\}_{n\in \mathbb{N}}$ are the coefficients of $L(s,\psi)$. We note that if $m\geq 1$ is odd then 
$ \chi_{(\pm2^lm)_0}=\chi_{m_0}\chi$ where $ m_0 $ denotes 
the squarefree part of $ m $  
for some character $\chi$ whose 
conductor divides 8, namely 
$$ 
\chi(d)=\begin{cases}
\jacobi{\pm 2}{d}, &\textrm{if  } l \textrm{ odd,}\\
\jacobi{\pm 1}{d}, &\textrm{if  } l \textrm{ even.}
\end{cases} 
$$ 
Notice that $\chi$ depends only on $l$ mod~$2$ and the sign $\pm$.
For the same $\chi$ we have $ q(w,m,\chi)= q(w,\pm 2^lm) $. It
follows that $L^*(s,m,\chi)=L^*(s,\pm 2^lm,1)$.  We write 
the summation index $n$ in (\ref{double-dirichlet-first-representation}) 
as $n=2^lm$ where $m$ is odd and split the sum as 
$$ 
\sum_{\substack{l=0\\l \textrm{ odd }}}^\infty \sum_{\substack{\pm m=1\\(m,2)=1}}^\infty \cdots 
+ \sum_{\substack{l=0\\l \textrm{ even }}}^\infty \sum_{\substack{\pm m=1\\(m,2)=1}}^\infty\cdots. 
$$
We split the $m$ sum further according to $m\equiv 1, 3, 5, 7 \, (8)$, which can be done by using a linear combination of
characters. We then use the explicit formulae for $r_2(w, -n)$ in 
\eqref{isuseful}, \eqref{isalsouseful}, and that the Fourier coefficients 
satisfy the Hecke relations to see that $Z_\pm(s,w)$ can be written as a linear combination of $Z(s, w, \chi, \chi')$   with coefficients being functions bounded on vertical strips
times one of the following series:
\begin{equation}
  \label{firsttwo}
\sum_{j=0}^{\infty}\frac{t_{2^{2j}}}{2^{2j(s+w-1/2)}},\quad  \sum_{j=0}^{\infty}\frac{t_{2^{2j+1}}}{2^{(2j+1)(s+w-1/2)}}
  \end{equation}
and 
\begin{equation}
  \label{lasttwo}
\sum_{j=0}^{\infty}\frac{t_{2^{2j}}u_j(2^{-(2w-1)})}{2^{2j(s-w+1/2)}}, \quad \sum_{j=0}^{\infty}\frac{t_{2^{2j+1}}u_j(2^{-(2w-1)})}{2^{(2j+1)(s-w+1/2)}}. 
\end{equation}
 We easily see that $$2\sum_{j=0}^{\infty}\frac{t_{2^{2j}}}{2^{2js}}=\frac{1}{p_2(2^{-s})}+\frac{1}{p_2(-2^{-s})}, 
\quad 2\sum_{j=0}^{\infty}\frac{t_{2^{2j+1}}}{2^{(2j+1)s}}=\frac{1}{p_2(2^{-s})}-\frac{1}{p_2(-2^{-s})}.$$
We  see also that, using \eqref{def-u},
$$\sum_{j=0}^{\infty} \frac{t_{2^{2j}}u_j(x)}{2^{2js}}=\frac{x^2}{2(1-x^2)}\left(\frac{1}{p_2(2^{-s})}+\frac{1}{p_2(-2^{-s})}-\frac{1}{p_2(x2^{-s})}-\frac{1}{p_2(-x2^{-s})}\right),
$$
which has no poles coming out of $x^2-1$ in the denominator. Similarly we see that
$$\sum_{j=0}^{\infty} \frac{t_{2^{2j+1}}u_j(x)}{2^{(2j+1)s}}=\frac{x^2}{2(1-x^2)}\left(\frac{1}{p_2(2^{-s})}-\frac{1}{p_2(-2^{-s})}-\frac{1}{x}\left(\frac{1}{p_2(x2^{-s})}-\frac{1}{p_2(-x2^{-s})}\right)\right).
$$
We substitute in the last four equations $s+w-1/2$ or $s-w+1/2$ for
$s$ as required and $x=2^{-(2w-1)}$ to identify the possible
polynomials that appear in the denominators. These have product $T(s,
w)$. We now notice that multiplying any of the 4 functions in \eqref{firsttwo},
\eqref{lasttwo} by $T(s,w)$ we get holomorphic functions bounded on
vertical strips, which proves the claim.

If $\psi$ is an oldform with, say, $\psi=\psi_1(2^jz)$ with $\psi_1$ a
primitive form, and $j=1,2$, then the series in \eqref{longago} becomes 
\begin{equation*}
  \sum_{n\neq 0}\frac{b_n(\psi_1) \phi_{-2^jn}(w,1/2)}{(2\pi\abs{2^jn})^{s-1}}
\end{equation*}
which by the explicit expression for $\phi_{n}(w,1/2)$ can be
analyzed similarly to the newform case.
\end{proof}

\begin{remark}\label{itneverends}
  In Theorem \ref{cuspidalcontribution} below, we need to study
  $Z_{\pm}(1/2+it, 1/2-it)$. For $\Re(s)=\Re (w)=1/2$ we notice that
  by \eqref{anotherboundagain} we have $1/{T(s, w)}=O(1)$.
\end{remark}
\begin{theorem}\label{cuspidalcontribution} 
Assume that for any $\chi,\chi'$ mod~$8$ the function $Z(s,1-s,\chi,\chi')$
satisfies a subconvex bound. Then 
$$ 
\int_{\Gamma\setminus\H}\psi(z)\abs{E(z,1/2+it,1/2)}^2 d\mu(z) \to 0 
$$
as $\abs{t}\to\infty$. 
\end{theorem}
\begin{proof}
By Proposition \ref{simpler-series-translation} a subconvex bound 
with saving $\delta$ translates into a bound $Z_\pm(s,1-s) = 
O(\abs{t}^{2(1/4-\delta)})$ when $\Re(s)=1/2$. 
Combining this with the bound in Lemma~\ref{nicoleresult}, the 
estimate $1/\zeta(1+it)=O(\log |t|)$ \cite[Eq.~3.11.8]{Titchmarsh:1986a}, 
and the identity \eqref{longago} we see that $I(s,1-s)=O(\abs{t}^{2(1/4-\delta)-1/2+\varepsilon})$ 
 for any $\varepsilon>0$  when $\Re(s)=1/2$. 
Since 
$$ 
I(1/2+it,1/2-it) 
= \int_{\Gamma\setminus\H} \psi(z) \abs{E(z,1/2-it,1/2)}^2 d\mu(z), 
$$ 
we find  that, when $\delta>0$,  $I(1/2+it,1/2-it)\to 0$ as $\abs{t}\to \infty$. 
\end{proof}

\begin{remark} 
  In the proof above we see that the trivial bound from 
Theorem~\ref{convexity-independent} only gives $O(\abs{t}^{1/2+\varepsilon})$.
\end{remark} 

\subsection{The incomplete Eisenstein series contribution}
In the following we choose a fundamental domain of $ \Gamma $ such that 
\begin{displaymath} 
\mathcal{D} = 
\mathcal{D}_0 \cup \bigcup_{j = 1}^3 \sigma_{\mathfrak{a}_j} \mathcal{D}^Y 
\end{displaymath} 
where $ \mathcal{D}^Y := \{x + iy: 0<x<1, y>Y\} $, $ Y $ sufficiently 
large, $ \mathcal{D}_0 $ is a suitable compact set and, as before, 
$ \sigma_{\mathfrak{a}_j} $ denotes the scaling matrix of the cusp $ \mathfrak{a}_j $. 

In order to introduce the incomplete Eisenstein series let $ h(y) \in C^{\infty} 
(\R^+) $ be a function which decreases rapidly at $ 0 $ and $ \infty $, 
and whose derivatives are also of rapid decay. Its 
Mellin transform evaluated at $-s$ is
\begin{equation}
H(s)=\int_0^{\infty}h(y)y^{-s}\frac{dy}{y}
\end{equation}
and thus by the Mellin inversion formula we have 
\begin{equation}
h(y) = \frac{1}{2\pi i} \int_{\Re s = a} H(s) y^s \, ds
\end{equation}
for any $ a \in \R $. The function $ H(s) $ is entire and $ H(a + it) $ is in 
the Schwartz space in the $ t $ variable for any $ a \in \R $. The incomplete 
Eisenstein series corresponding to the cusp $ \mathfrak{a} $ is then given by 
\begin{equation}
F_h(z, \mathfrak{a}) = \sum_{\gamma\in\Gamma_{\mathfrak{a}}\setminus\Gamma} h(\Im 
\sigma^{-1}_{\mathfrak{a}} \gamma z)
= \frac{1}{2\pi i} \int_{\Re s = a > 1} H(s) E_{\mathfrak{a}}(z,s,0) \, ds.
\end{equation}
For $ i = 1, 2, 3 $ we are interested in the behavior of 
\begin{displaymath} 
J(t, \mathfrak{a}_i) = 
\int_{\GmodH} F_h(z, \mathfrak{a}_i) \left|E\left(z, \frac{1}{2} 
+ it, \frac{1}{2}\right)\right|^2 d\mu(z) \ \mbox{ as } \abs{t} \rightarrow \infty. 
\end{displaymath} 
In the following we only treat the contribution from the cusp at
infinity but the other contributions can be dealt with
similarly. Unfolding the incomplete Eisenstein series we find 
\begin{align*} 
J(t, \infty) &=  \int_{\GmodH} F_h(z, \infty) \, d\mu_t(z) 
=  J_1(t, \infty) + J_2(t, \infty) 
\end{align*} 
with 
\begin{align} 
\label{J1a} 
J_1(t, \infty) &:= \int_0^{\infty} h(y) \Big|y^{\frac{1}{2}+it}+\phi(\frac{1}{2}+it,1/2) 
y^{\frac{1}{2}-it}\Big|^2 \frac{dy}{y^2}, \\ 
\label{J2a} 
J_2(t, \infty) &:= \int_0^{\infty} h(y) 
\sum_{n \not= 0} \Big|\phi_n(1/2+it, 1/2)\Big|^2 \Big|W_{\frac{n}{4|n|}, it} 
(4 \pi |n|y)\Big|^2 \frac{dy}{y^2}. 
\end{align} 
The integral $ J_1(t, \infty) $ is easily dealt with. Namely, we obtain 
\begin{align}  
\nonumber J_1(t, \infty) 
& = 
(1+|\phi(1/2+it, 1/2)|^2) \int_0^{\infty} h(y) \, \frac{dy}{y} \\ 
\nonumber & \quad 
+ \phi(1/2+it, 1/2) \int_0^{\infty} h(y) y^{-2it} \, \frac{dy}{y} \\ 
\nonumber & \quad 
+ \overline{\phi(1/2+it, 1/2)} \int_0^{\infty} h(y) y^{2it} \, \frac{dy}{y} \\ 
\label{J1} & = 
(1+|\phi(1/2+it, 1/2)|^2) H(0) + \phi(1/2+it, 1/2) H(2it) \\ 
\nonumber& \quad 
+ \overline{\phi(1/2+it, 1/2)} H(-2it)\\
\nonumber &=O(1). 
\end{align} 
For the integral $ J_2(t, \infty) $ we find, using the rapid decay of the 
Whittaker function and the Mellin inversion formula, 
\begin{equation} \label{J2} 
J_2(t, \infty) = 
\frac{1}{2 \pi i} \int_{\Re s=a>1} H(s) R_1\Big(|E(z, 1/2+it, 1/2)|^2, s\Big) ds 
\end{equation} 
where 
\begin{align}\nonumber R_1\Big(|E(z, w, 1/2)|^2, s\Big) & =\sum_{n\neq
  0}\abs{\phi_n(w,1/2)}^2\int_{0}^\infty
\abs{W_{\frac{n}{4\abs{n}},w-1/2}(4\pi
  \abs{n}y)}^{2}y^{s-1}\frac{dy}{y}\\
\label{R1}
&= \sum_{n\neq0} \frac{\abs{\phi_n(w,1/2)}^2}{(2\pi \abs{n})^{s-1}} 
\int_{0}^\infty \abs{W_{\frac{n}{4\abs{n}},w-1/2}(2y)}^{2}y^{s-1}\frac{dy}{y}. 
\end{align}
In order to analyze the asymptotic behavior of $ J_2(t, \mathfrak{\infty})$
we need to understand the function $R_1(|E(z, w, 1/2)|^2, s)$. 
There are (at least) two ways to do this: To use properties of the double
Dirichlet series we defined in Section \ref{sec:doubledirichletseries}, 
or to use Zagier's theory of the Rankin-Selberg method for functions 
that are not of rapid decay but satisfy certain mild growth condition. 
We will actually use a combination of these two techniques. We want to 
shift the line of integration in (\ref{J2}) to $ \Re (s) = 1/2 $. For 
this we need to identify the poles, estimate them, see what the contribution 
of $ \int_{\Re (s) = 1/2} H(s) R_1(|E(z, w, 1/2)|^2, s) $ is to the 
asymptotics. For the first 
and third aspect we use the Rankin-Selberg approach and for the second 
aspect the multiple Dirichlet series approach works best. 

We first describe why double Dirichlet series techniques apply. The
growth of the Mellin transform of the absolute value of the Whittaker
function is analyzed in Lemma \ref{estimateM}. By combining
\eqref{good-expression}, \eqref{isuseful}, and \eqref{isalsouseful} 
we see that $\overline{\phi_n(\overline{w},1/2)}=\phi_n(w,1/2)$. 
This shows that when $\Re(w)=1/2$ we have
$$\abs{\phi_n(w,1/2)}^2=\phi_n(w,1/2)\phi_n(1-w,1/2).$$ The
right-hand side has the advantage of being meromorphic in $w$. 
We define
$$ 
\hat Z_{\pm}(s,w) = \frac{\Gamma(w\pm \frac{1}{4}) 
\Gamma(1-w\pm \frac{1}{4})}{\pi i} \sum_{\pm n=1}^\infty 
\frac{\phi_n(w,1/2)\phi_n(1-w,1/2)}{\abs{n}^{s-1}}, 
$$
which by \eqref{good-expression} equals
\begin{align*} 
& \frac{1}{\zeta_2(4w-1)\zeta_2(4(1-w)-1)} \cdot \\ 
& \qquad \qquad 
\sum_{\pm n=1}^\infty 
\frac{L^*(2w-1/2,n,1)L^*(2(1-w)-1/2,n,1)}{\abs{n}^s} r_2(w,n) r_2(1-w,n). 
\end{align*} 
We now show that $\hat Z_{\pm}(s,w)$ is directly related to the 
function $\hat Z\left(s, 
w,\chi,\chi'\right)$ defined in \eqref{onceagainaseries}. 
Let
\begin{equation*}
  U(s,w)=(1-2^{-(4w-1)})(1-2^{-2s})(1-2^{-(4w-2+2s)})(1-2^{-(-4w+2+2s)}).
\end{equation*}
\begin{proposition}\label{hatZtosomething} 
There exist functions $\hat f_{\pm,\kappa}(s,w,\chi,\chi')$ bounded 
in vertical strips such that 
\begin{align*} 
& U(s,w) \hat Z_{\pm}(s,w) = \frac{1}{\zeta_2(4w-1)\zeta_2(4(1-w)-1)} 
\times \\ 
& \times 
\sum_{\kappa\in\{0,1\}} \frac{\Gamma(\frac{2w-1/2+\kappa}{2})}{\Gamma 
(\frac{2(1-w)-1/2+\kappa}{2})} 
\sum_{\chi,\chi'}\hat f_{\pm,\kappa}(s,w,\chi,\chi') \hat Z\left(\frac{s+2w-1/2}{2}, 
\frac{s-2w+3/2}{2},\chi,\chi'\right). 
\end{align*}
\end{proposition}
\begin{proof}
As in  the proof of Proposition \ref{simpler-series-translation} we
write $n=2^lm$ and split into sums over $l$ even, odd respectively. We
then split the $m$ sum according to the residue class mod~$8$ which
is a linear combination over characters mod~$8$. Inserting the
explicit formulae for $r_2(w,n)$ \eqref{isuseful},
\eqref{isalsouseful} we are led to consider the series
\begin{equation*}
  \sum_{j=0}^\infty u_j(x)u_j(y)z^j,\quad \sum_{j=0}^\infty
  u_j(x)z^j,\quad \sum_{j=0}^\infty z^j
\end{equation*} with $x,y,z$ being appropriate powers of 2. Since
these are all sums of geometric series, see \eqref{def-u}, they are explicitly computable and
 after multiplying by
$(1-2^{-2s})(1-2^{-(4w-2+2s)})(1-2^{-(-4w+2+2s)})$ they become
Dirichlet polynomials in powers of 2, hence holomorphic and bounded in
vertical strips. Therefore 
$$ 
(1-2^{-2s})(1-2^{-(4w-2+2s)})(1-2^{-(-4w+2+2s)})\hat Z_{\pm}(s,w)=\sum_{\chi,\chi'}\tilde
f_{\pm}(s,w,\chi,\chi')\tilde Z(s,w,\chi,\chi') 
$$
where 
\begin{equation*} 
\begin{split} 
\tilde Z(s, & w, \chi, \chi') = \\ 
& \frac{1}{\zeta_2(4w-1)\zeta_2(4(1-w)-1)}\sum_{ n=1}^\infty
\frac{\chi'(n)L^*(2w-1/2,n,\chi)L^*(2(1-w)-1/2,n,\chi)}{n^s} 
\end{split} 
\end{equation*} 
and $\tilde
f_{\pm}(s,w,\chi,\chi')$ are bounded in vertical strips.
Using the functional equation  on $L^*(2(1-w)-1/2,n,\chi)$ we see -- as
in the proof of Theorem \ref{first-functional-equation} -- that 
$$ 
(1-2^{-(4w-1)})\sum_{ n=1}^\infty
\frac{\chi'(n)L^*(2w-1/2,n,\chi)L^*(2(1-w)-1/2,n,\chi)}{n^s} 
$$ 
equals 
$$ 
\sum_{\kappa\in\{0,1\}} 
\frac{\Gamma(\frac{2w-1/2+\kappa}{2})}{\Gamma(\frac{2(1-w)-1/2+\kappa}{2})} 
\sum_{\chi,\chi'}\tilde{\tilde{f}}_{\kappa}(x,y,\chi,\chi') 
\tilde{\tilde{ Z}}(s,w,\chi,\chi') 
$$
where $ \tilde{\tilde{f}}_{\kappa}(x,y,\chi,\chi') $ is another set of functions 
bounded in vertical strips and 
$$ 
\tilde{\tilde{Z}}(s,w,\chi,\chi')=\sum_{ n=1}^\infty
\frac{\chi'(n)L^*(2w-1/2,n,\chi)^2}{ n^{s-2w+1}}. 
$$
Combining the above equations  and comparing with
\eqref{onceagainaseries} finishes the proof.
\end{proof}
The above lemma implies that many questions about $R_1(|E(z, w,
1/2)|^2, s)$ can be dealt with using $Z(s,w,\chi,\chi')$. We now
describe a different method for understanding $R_1(|E(z,w,1/2)|^2, 
s)$, namely Zagier's Rankin-Selberg method for 
functions not of rapid decay. This method was introduced by Zagier 
for the group SL$_2 (\Z) $ in \cite{Zagier:1981a} and generalized 
by Kudla (unpublished), Dutta Gupta \cite{Dutta-Gupta:1997a}, 
and Mizuno \cite{Mizuno:2005a}. 
Its usefulness for determining the contribution of the incomplete 
Eisenstein series to the asymptotics can already be seen in
\cite{Zelditch:1991}. 
We introduce the generalized Rankin-Selberg transform following 
Zagier \cite{Zagier:1981a} and Mizuno \cite{Mizuno:2005a}. 
 We write $ e_{ij}(y,s,k) = \delta_{ij} y^s + \phi_{ij}(s, k) y^{1-s} $ for the 
zero Fourier coefficient of $ E_{\mathfrak{a}_i}(z,s,k) $ at $
\mathfrak{a}_j $ and we denote the scattering matrix by $ \Phi(s, k) =
(\phi_{ij}(s, k))$. We note that for $\Gamma_0(4)$ the matrix $
\Phi(s, 0)$ is $3\times 3$ whereas $ \Phi(s, 1/2)$ is $2\times 2$. 
For the weight $0$ Eisenstein series we use the notation 
$E_i(z,s,0)=E_{\mathfrak a_i}(z,s,0)$.

\begin{theorem} {\cite[Th.~2]{Mizuno:2005a}}\label{theorem:RankinSelberg} 
Let $ F $ be a continuous functions on $ \H $ that is $ \Gamma $-invariant and 
satisfies for $ i = 1, 2, 3 $ 
\begin{displaymath} 
F(\sigma_{{\mathfrak a}_i} z) = \psi_i(y) + O(y^{-N}) \quad 
\mbox{for all } N \mbox{ as } y \rightarrow \infty, 
\end{displaymath} 
where 
\begin{displaymath} 
\psi_i(y) = \sum_{j=1}^l \frac{c_{ij}}{n_{ij}!} \, y^{\a_{ij}} \log^{n_{ij}} y, \quad n_{ij}\in {\mathbb N}\cup \{0\}, \quad  
i = 1, 2, 3. 
\end{displaymath} 
For such a function $ F $ the Rankin-Selberg transform $ R_i (F, s) $ 
corresponding to the cusp $ \mathfrak{a}_i $, $ i = 1, 2, 3 $, is defined by 
\begin{displaymath} 
R_i (F, s) := 
\int_0^{\infty} \int_0^1 (F(\sigma_{{\mathfrak a}_i} z) - \psi_i(y)) y^s d\mu(z), 
\end{displaymath} 
for $ \Re s$ sufficiently large. Then we have 
\begin{align} \label{rankin} 
\nonumber R_i(F, s) & = 
\int_{\mathcal{D}_0} F(z) E_i(z, s, 0) d\mu(z) \\ 
\nonumber & \quad 
+ \sum_{j=1}^3 \int_{\mathcal{D}^Y} (F(\sigma_{{\mathfrak a}_j} z) E_i(\sigma_{{\mathfrak a}_j} z, s, 0) 
- \psi_j(y) e_{ij} (y, s, 0)) d\mu(z) \\ 
\nonumber & \quad 
+ \sum_{j=1}^3 \phi_{ij} (s, 0) \int_Y^{\infty} \psi_j(y) y^{-s-1} \, dy 
- \int_0^Y \psi_i(y) y^{s-2} \, dy \\ 
& = 
\int_{\mathcal{D}_0} F(z) E_i(z, s, 0) d\mu(z) \\ 
\nonumber & \quad 
+ \sum_{j=1}^3 \int_{\mathcal{D}^Y} (F(\sigma_{\mathfrak a_j} z)
E_i(\sigma_{\mathfrak a_j} z, s, 0) 
- \psi_j(y) e_{ij} (y, s, 0)) d\mu(z) \\ 
\nonumber & \quad 
- \sum_{j=1}^3 \phi_{ij} (s, 0) \widehat{\psi_j}(1-s, Y) - 
\widehat{\psi_i}(s, Y) ,
\end{align} 
where 
\begin{equation*} 
\widehat{\psi_i}(s, Y) = 
\sum_{j=1}^l c_{ij} \sum_{m=0}^{n_{ij}} 
\frac{(-1)^{n_{ij}-m}}{m!} \frac{Y^{s+\alpha_{ij}-1} \log^m Y}{(s+\alpha_{ij}
-1)^{n_{ij}-m+1}} .
\end{equation*} 
Furthermore, for each $ i = 1, 2, 3 $ the function $ R_i(F,s) $ can be 
meromorphically continued to $ \C $ and we have the functional equation 
\begin{equation*} 
\mathcal{R}(F,s) := {}^t (R_1(F,s), R_2(F,s), R_3(F,s)) 
= \Phi(s,0) \mathcal{R}(F,1-s). 
\end{equation*} 
\end{theorem} 
We want to move the line of integration in \eqref{J2} to $\Re(s)=1/2$ 
and Theorem~\ref{theorem:RankinSelberg} plays a major role,  
as it allows to identify the relevant poles and to calculate the 
corresponding residues. 
By the above theorem, in particular by (\ref{rankin}), we infer 
\begin{align} 
R_1\big(|E & (z, 1/2+it, 1/2)|^2, s\big) = 
\int_{\mathcal{D}_0} |E(z, 1/2+it, 1/2)|^2 E_1(z, s, 0) d\mu(z) \nonumber \\ 
\label{rankinE}& + \sum_{j=1}^3 \int_{\mathcal{D}^Y} \Big(|E(\sigma_{{\mathfrak a}_j} z, 1/2+it, 1/2)|^2 
E_1(\sigma_{{\mathfrak a}_j} z, s, 0) - \psi_j(y) e_{1j} (y, s, 0)\Big) d\mu(z) \\ 
& - \widehat{\psi(s, Y)}, \nonumber 
\end{align} 
where 
\begin{align*} 
\widehat{\psi(s, Y)} &= \widehat{\psi_1}(s, Y) + \phi_{11}(s, 0) 
\widehat{\psi_1}(1-s, Y) + \frac{Y^{1-s}}{1-s} \phi_{12}(s, 0) 
|\phi_{12}(1/2+it, 1/2)|^2, \\ 
\widehat{\psi_1}(s, Y) &= \frac{Y^s}{s} \Big(1+|\phi_{11}(1/2+it, 1/2)|^2\Big) 
+ \frac{Y^{s-2it}}{s-2it} \, \phi_{11}(1/2+it, 1/2) \\ 
& \quad 
+ \frac{Y^{s+2it}}{s+2it} \, \overline{\phi_{11}(1/2+it, 1/2)}, \\ 
\psi_j(y) &= \big|\delta_{1j} y^{1/2+it} + \phi_{1j} (1/2+it, 1/2)y^{1/2-it}\big|^2, \ 
j=1, 2, \\ 
\psi_3(y) &= 0. 
\end{align*} 
Thus we easily see that we pick up residues at $s=1$ and at $ s= 1 \pm 2it $ 
when we shift the line of integration. The pole at $ s=1 $ is responsible for 
the contribution of the $ \log \abs{t} $-term in \eqref{againagain} as we will see. 
We, therefore, examine $ H(s) R_1(|E(z, 1/2+it, 1/2)|^2, s) $ at $s=1$. In 
order to determine the order of the pole at $ s=1 $ and its residue we use 
the Laurent expansion of $ H(s) $ and $ R_1(|E(z, 1/2+it, 1/2)|^2, s) $. 
The first two terms of (\ref{rankinE}) are easily understood because of the 
Eisenstein series which has simple poles at $s=1$ and no other poles
in $\Re(s)\geq 1/2$. In order to treat the last term of (\ref{rankinE}) we 
write 
\begin{align*} 
\frac{Y^{1-s}}{1-s} &= - \frac{1}{s-1} + \log Y + O(|s-1|), \\ 
\phi_{1j}(s, 0) &= \frac{1}{\vol{\GmodH}} \frac{1}{s-1} 
+ b_0^{1j} + O(|s-1|). 
\end{align*} 
These expansions and the fact that the 
scattering matrix $ \Phi(s, 1/2) = (\phi_{ij}(s, 1/2))_{1 \leq 
i, j \leq 2} $  is unitary for $\Re (s)=1/2$ (cf. \cite[Lemma 10.5]{Roelcke:1966a}) yield 
\begin{align*}  
\widehat{\psi(s, Y)} &= 
- \frac{1}{\vol{\GmodH}} \Big(1 + \sum_{j=1}^2 |\phi_{1j} 
(1/2+it, 1/2)|^2\Big) \frac{1}{(s-1)^2} \\ 
& \quad + \bigg(\Big(1 + \sum_{j=1}^2 |\phi_{1j} (1/2+it, 1/2)|^2\Big) 
\frac{\log Y}{\vol{\GmodH}} \\ 
& \quad 
- \Big(1 + |\phi_{11} (1/2+it, 1/2)|^2\Big) b_0^{11} 
- |\phi_{12} (1/2+it, 1/2)|^2 b_0^{12} \\ 
& \quad 
+ \frac{1}{\vol{\GmodH}} \frac{\overline{\phi_{11}}(1/2+it, 
1/2) Y^{2it} - \phi_{11}(1/2+it, 1/2) Y^{-2it}}{2it}\bigg) \frac{1}{s-1} + O(1) \\ 
&= 
- \frac{2}{\vol{\GmodH}} \frac{1}{(s-1)^2} 
+ \bigg(\frac{2 \log Y}{\vol{\GmodH}} 
- \Big(1 + |\phi_{11} (1/2+it, 1/2)|^2\Big) b_0^{11} \\ 
& \quad 
- |\phi_{12} (1/2+it, 1/2)|^2 b_0^{12} \\ 
& \quad 
+ \frac{1}{\vol{\GmodH}} \frac{\overline{\phi_{11}}(1/2+it, 
1/2) Y^{2it} - \phi_{11}(1/2+it, 1/2) Y^{-2it}}{2it}\bigg) \frac{1}{s-1} + O(1) .
\end{align*} 
Consequently we see that $ R_1(|E(z, 1/2+ir, 1/2)|^2, s) $ has a pole of 
order $ 2 $ in $ s=1 $. Furthermore, 
\begin{align}
\nonumber\res_{s=1} H(s) & R_1\big(|E(z, 1/2+it, 1/2)|^2, s\big) \\ 
\nonumber&= 
\Bigg(\frac{1}{\vol{\GmodH}} 
\bigg(- 2 \log Y + 
\int_{\mathcal{D}_0} |E(z, 1/2+it, 1/2)|^2 d\mu(z) \\ 
\label{res} 
& \quad 
\phantom{\Bigg(\frac{1}{\vol{\GmodH}} \bigg(\int} 
+ \sum_{j=1}^3 \int_{\mathcal{D}^Y} \Big(|E(\sigma_{\mathfrak a_j} z, 1/2+it, 1/2)|^2 
- \psi_j(y) \Big) d\mu(z) \\ 
\nonumber& \quad 
\phantom{\Bigg(\frac{1}{\vol{\GmodH}} \bigg(\int} 
-  \frac{\overline{\phi_{11}}(1/2+it, 1/2) Y^{2it} - \phi_{11}(1/2+it, 1/2) 
Y^{-2it}}{2it} 
\bigg) \\ 
\nonumber& \quad 
\phantom{\Bigg( \Bigg)} 
+ b_0^{11} 
+ \sum_{j=1}^2 |\phi_{1j} (1/2+it, 1/2)|^2 b_0^{1j}\Bigg) H(1) 
+ \frac{2 H'(1)}{\vol{\GmodH}} \\ 
\nonumber&= 
\Bigg(-\frac{1}{\vol{\GmodH}} 
\sum_{j=1}^2 {\phi_{1j}}'(1/2+it, 1/2) \overline{\phi_{1j}} (1/2+it, 1/2) \\ 
\nonumber& \quad 
\phantom{\Bigg( \Bigg)} 
+ b_0^{11} + \sum_{j=1}^2 |\phi_{1j} (1/2+it, 1/2)|^2 b_0^{1j}\Bigg) H(1) 
+ \frac{2 H'(1)}{\vol{\GmodH}},
\end{align} 
where we used the Maa\ss -Selberg relations (see
e.g.\ \cite[Lemma 11.2]{Roelcke:1966a}). 
For the remaining poles at $ s = 1 \pm 2it $ we obtain 
\begin{equation*} 
\res_{s=1 + 2it} H(s) R_1\big(|E(z, 1/2+it, 1/2)|^2, s\big) = 
H(1+2it) \phi_{11}(1+2it, 0) \overline{\phi_{11}(1/2+it, 1/2)} ,
\end{equation*} 
and this expression is of rapid decay as $ \abs{t} \rightarrow \infty
$. This follows from the following general facts: the entries of the 
scattering matrix of weight zero are uniformly bounded for $ \Re(s) 
\geq 1/2$, $\abs{\Im(s)}\geq 1$ (see e.g.\ \cite[p.~655]{Selberg:1989a}), 
$ \phi_{11}(1/2 \pm it, 1/2) $ is bounded since $\Phi(1/2+it,k)$ is unitary, 
and the rapid decay of $ H(1 \pm 2it) $. The same bound holds for the residue of 
$ H(s) R_1(|E(z, 1/2+it, 1/2)|^2, s) $ at $ s = 1-2it $. 
We now want to shift the line of integration in \eqref{J2}. To do this we need
to control the growth of the $R_1(|E(z, 1/2+it, 1/2)|^2, s)$ apart 
from knowing the residues. 
\begin{lemma} \label{shift} 
Let $F(z)=|E_1 (z, 1/2+it, 1/2)|^2$. The function $ 
R_1\big(F(z), \sigma + iv\big) 
$ 
is of at most polynomial growth as $ |v| \rightarrow \infty $ for $ \sigma \geq 
1/2 $.  
\end{lemma} 
\begin{proof} 
In order to avoid the poles of the Eisenstein series coming from the zeros 
of the  zeta function in the critical strip we work with $ R^*_i(F, s) := 
\zeta(2s)R_i(F,s) $, $ i = 1, 2, 3 $. Then the function $ R^*_i(F, s) $ has 
only finitely many poles in the strip \mbox{$ 0 \leq \Re (s) \leq 1 $}. The 
estimates for the Eisenstein series and the scattering matrix imply that
\begin{equation*} 
R^*_i(F, s) = O(1) 
\end{equation*} 
as $ |\Im (s)| \rightarrow \infty $ for $ \Re (s) > 1 $, $ i = 1, 2, 3 $. 
Using the functional equation as well as explicit expressions for 
$ \phi_{1j} (s, 0) $ 
we then get
\begin{equation*} 
R^*_1(F,s) 
= \frac{\zeta(2s)}{\zeta(2(1-s))} \sum_{j=1}^3 \phi_{1j}(s,0) R^*_j(F,1-s) 
= O\big(|\Im (s)|^{1-2\sigma}\big) 
\end{equation*} 
as $ |\Im (s)| \rightarrow \infty $ for $ \sigma = \Re (s) < 0 $, $ i = 1, 
2, 3 $. Thus by the Phragm\'en-Lindel\"of principle we finally obtain that 
\begin{equation*} 
R_1(F, \sigma + iv) = O\big(|v|^k\big) 
\end{equation*} 
as $ |v| \rightarrow \infty $, $ \sigma \geq 1/2 $, for some $ k \in \N $. 
\end{proof} 
Now that polynomial growth has been established it follows then from (\ref{res}) that
\begin{align}
\nonumber 
J_2(t, \infty) &= 
\Bigg(-\frac{1}{\vol{\GmodH}} \sum_{j=1}^2 \frac{{\phi_{1j}}'}{\phi_{1j}}(1/2+it, 1/2) 
|{\phi_{1j}} (1/2+it, 1/2)|^2 + b_0^{11} \\ 
 \label{J2aftershift}& \quad 
\phantom{\Bigg( \Bigg) -\frac{1}{2 \pi}} 
+ \sum_{j=1}^2 |\phi_{1j} (1/2+it, 1/2)|^2 b_0^{1j}\Bigg) H(1) 
+ \frac{H'(1)}{\pi} \\ 
\nonumber& \quad \phantom{\Bigg( \Bigg)} + \frac{1}{2 \pi i} 
\int_{\Re s = 1/2} H(s) R_1\Big(|E(z, 1/2+it, 1/2)|^2, s\Big) ds 
+ O(1).
\end{align} 
In section \ref{Sec:weight-half-eisensteinseries} we saw that, 
up to constants and fractions of polynomials in powers of $ 2 $, the 
entries of the scattering matrix are equal to $ \xi(3-4s)/\xi(4s-1) $, 
see \eqref{scattering-explicit}. Hence in order to determine the 
asymptotic behavior of the first term in (\ref{J2aftershift}) with 
respect to the $ t $-variable we need to understand the logarithmic 
derivative of $ \xi(3-4s)/\xi(4s-1) $ at $ s=1/2 +it$. 
The contribution from the remaining terms is $ O(1) $.
We have 
\begin{align*} 
\bigg(\log \frac{\xi(3-4s)}{\xi(4s-1)}\bigg)'\Bigg|_{s=1/2+it} 
&= 
4 \log \pi - 2 \frac{\Gamma'}{\Gamma}\bigg(\frac{1}{2}-2it\bigg) - 
2 \frac{\Gamma'}{\Gamma}\bigg(\frac{1}{2}+2it\bigg) \\ 
& \quad 
- 4 \bigg(\frac{\zeta'}{\zeta}(1-4it) - \frac{1}{4it} 
+ \frac{\zeta'}{\zeta}(1+4it) + \frac{1}{4it}\bigg) \\ 
&= -4 \log |t| + o(\log |t|)  
\end{align*} 
by Stirling's formula and \cite[Theorem 5.17]{Titchmarsh:1986a}.
Since $ \Phi(s, \frac{1}{2}) $ is unitary for $ \Re s = \frac{1}{2} $, 
we finally arrive at
\begin{equation} 
J_2(t, \infty) = \frac{4 H(1)}{\vol{\GmodH}} \log \abs{t} + \frac{1}{2 \pi i} 
\int_{\Re s = 1/2} H(s) R_1\Big(|E(z, 1/2+it, 1/2)|^2, s\Big) ds 
+ o(\log \abs{t}) 
\end{equation} 
as $ \abs{t} \rightarrow \infty $. 
To treat the last integral we use again the connection to double
Dirichlet series.
\begin{lemma} \label{estimateJ3} 
Assume that for any $\chi,\chi'$ mod~$8$ the function $Z_{\psi_\tau}(s,1-s,\chi,\chi')$
  satisfies a subconvex bound with saving $\delta>0$. Then, as $\abs{t}\to\infty$, 
\begin{equation*}
\frac{1}{2 \pi i} 
\int_{\Re s = 1/2} H(s) R_1\Big(|E(z, 1/2+it, 1/2)|^2, s\Big) ds = o(1).
\end{equation*} 
\end{lemma} 
\begin{proof} 
By \eqref{R1}, Proposition~\ref{hatZtosomething} combined with
$U(s,w)^{-1}=O(1)$ when $\Re(s)=\Re(w)=1/2$, Lemma~\ref{ZhatZtilde},  
Stirling's formula, Lemma~\ref{estimateM} and $1/\zeta(1+it)=O(\log |t|)$ we find that 
\begin{align*} 
R_1\Big(|E(z, 1/2+it, & 1/2)|^2, \, 1/2+iu\Big) = \\ 
& 
O\left(\abs{t}^{-1/2+\varepsilon} \max_{\chi,\chi'} \abs{Z_{\psi_\tau}(\frac{1}{2}+i(u+2t), 
\frac{1}{2}+i(u-2t),\chi,\chi')}\right). 
\end{align*} 
Subconvexity implies that the $\max$ is 
$$ O\Big(\big((1+\abs{u+2t})( 1+\abs{u-2t})(1+2\abs{u})^2\big)^{1/4-\delta}\Big). $$

Using the rapid decay of $H(s)$ 
we finally obtain that
\begin{equation*} 
J_3(t, \infty) = O\Big(|t|^{-1/2+\varepsilon} |t|^{2(1/4-\delta)}\Big) 
= o(1).
\end{equation*}
\end{proof}

\begin{remark} 
In the above proof we see that, as in the cuspidal case, the 
trivial bound from Theorem~\ref{convexity-independent} only gives 
$O(\abs{t}^{1/2+\varepsilon})$. However, for a compact set $ A $ 
the Maa\ss -Selberg relations easily yield
\begin{equation*} 
\int_A |E(z, 1/2+it, 1/2)|^2 d\mu(z) = O(\log t). 
\end{equation*} 
\end{remark} 
 
To summarize we have proved:
\begin{theorem} \label{incompleteeisensteincontribution} 
Assume that for any $\chi,\chi'$ mod~$8$ the function $Z(s,1-s,\chi,\chi')$
satisfies a subconvex bound. Then, as $\abs{t}\to\infty$, 
\begin{equation*}
\int_{\Gamma\setminus\H} F_h(z)\abs{E_\infty(z,1/2+it,1/2)}^2 d\mu(z) 
= \frac{4}{\vol{\GmodH}}H(1)\log \abs{t}+o(\log\abs t). 
\end{equation*}
\end{theorem}
The asymptotics \eqref{againagain} and hence Theorem~\ref{QUE} now
follows from Theorems~\ref{cuspidalcontribution} and
\ref{incompleteeisensteincontribution} by an approximation argument 
as in \cite[p.~217]{LuoSarnak:1995a}.

\section{Appendix: Mellin transforms of products of Whittaker
  functions}
In this appendix we prove various bounds on Mellin transforms of
products of Whittaker functions that we have not been able to find in
the literature in the generality needed.
\begin{lemma} \label{nicoleresult} 
Let $\parity\in\{\pm 1\}$. For $s=1/2+it$, $w= 1-s$, and $s_0$ fixed, we have the following bound
$$ 
\frac{1}{\Gamma(w +\parity/4)} \int_{0}^\infty W_{0,{s_0-1/2}}(y) 
W_{\parity/4,w-1/2}(y) y^{s-1} \frac{dy}{y} = O((1+\abs{t})^{-1/2}) ,
$$ 
as $\abs{t}\to\infty$.
\end{lemma}
\begin{remark}\label{order-of-growth}
The estimate in Lemma \ref{nicoleresult} cannot be improved, as the
proof below shows that the estimate can be turned into an asymptotic 
rate of decay of the same order.
\end{remark}
\begin{proof} 
Using \cite[7.611 7., p.~821]{GradshteynRyzhik:2007a}, we obtain 
\begin{align} 
\nonumber \int_0^{\infty} W_{0,s_0-1/2}& (y) W_{ 
\frac{\parity}{4},w-1/2}(y) y^{s-1} \frac{dy}{y} \\ 
\nonumber&= \frac{\Gamma(s+w-s_0) \Gamma(s+w+s_0-1) 
\Gamma(1-2w)}{\Gamma(1-\frac{\parity}{4}-w) \Gamma(s+w)} 
\\ 
\label{nmellin}  & \quad \times 
{}_3F_2(s+w-s_0,s+w+s_0-1,w-\frac{\parity}{4};2w, s+w;1) \\ 
\nonumber& \quad + \frac{\Gamma(s-w+s_0) \Gamma(s-w-s_0+1) 
\Gamma(2w-1)}{\Gamma(w-\frac{\parity}{4}) \Gamma(s-w+1)} \\ 
\nonumber & \quad \times 
{}_3F_2(s-w+s_0,s-w-s_0+1,1-\frac{\parity}{4}-w;2-2w,s-w+1;1)  ,
\end{align} 
if $ |\Re(s_0-1/2)|+|\Re(w-1/2)|<\Re s $. 
The generalized hypergeometric series that appear in (\ref{nmellin}) 
converge for $ \Re s < 1+\frac{\parity}{4} $. 
We now set $ s=1/2+it $ and $ w=1/2-it $ and get 
\begin{align*} 
\int_0^{\infty} & W_{0,s_0-1/2}(y) W_{ 
\frac{\parity}{4},-it}(y) y^{s-1} \frac{dy}{y} \\ 
&= 
\frac{\Gamma(1-s_0) \Gamma(s_0) \Gamma(2it)}{\Gamma(\frac{1}{2}-
\frac{\parity}{4}+it) \Gamma(1)} 
{}_3F_2(1-s_0,s_0,\frac{1}{2}-\frac{\parity}{4}-it;1-2it,1;1) \\ 
& \quad + \frac{\Gamma(s_0+2it) \Gamma(1-s_0+2it) 
\Gamma(-2it)}{\Gamma(\frac{1}{2}-\frac{\parity}{4}-it) \Gamma(1+2it)} \\ 
& \quad \qquad \times 
{}_3F_2(s_0+2it,1-s_0+2it,\frac{1}{2}-\frac{\parity}{4}+it;1+2it,1+2it;1). 
\end{align*} 
Using \cite[p.~18]{Bailey:1964a}, we infer that 
(see also \cite[(2.9), p.~1491]{Jakobson:1994a}) 
\begin{align*} 
{}_3F_2(s_0+2it,&1-s_0+2it, \frac{1}{2}-\frac{\parity}{4}+it;1+2it,1+2it;1) \\ 
&= 
 \frac{\Gamma(\frac{1}{2}+
\frac{\parity}{4}-it) \Gamma(1+2it)}{\Gamma(\frac{1}{2}+\frac{\parity}{4} 
+it) \Gamma(1)}  
{}_3F_2(\frac{1}{2}-\frac{\parity}{4}+it, 1-s_0, s_0; 1+2it,1;1), 
\end{align*} 
and thus
\begin{align} 
\nonumber \int_0^{\infty} W&_{0,s_0-1/2}(y) W_{ 
\frac{\parity}{4},-it}(y) y^{s-1} \frac{dy}{y} \\ 
\label{nmn} &= \frac{\Gamma(1-s_0) \Gamma(s_0) \Gamma(2it)}{\Gamma(\frac{1}{2} 
-\frac{\parity}{4}+it)} 
{}_3F_2(1-s_0,s_0,\frac{1}{2}-\frac{\parity}{4}-it;1-2it,1;1) \\ 
\nonumber & \quad 
+ \frac{\Gamma(s_0+2it) \Gamma(1-s_0+2it)\Gamma(-2it) 
\Gamma(\frac{1}{2}+\frac{\parity}{4}-it)}{\Gamma(\frac{1}{2} 
-\frac{\parity}{4}-it) \Gamma(\frac{1}{2}+\frac{\parity}{4}+it)} \\ 
\nonumber& \quad \qquad \qquad \qquad \qquad \qquad \qquad \times 
{}_3F_2(1-s_0, s_0, \frac{1}{2}-\frac{\parity}{4}+it;1+2it,1;1). 
\end{align} 
We want to understand the asymptotic behavior of the hypergeometric 
series appearing in (\ref{nmn}). Since $ \Re (s)=1/2<1+\parity/4$ the hypergeometric series in 
(\ref{nmn}) converge absolutely. Moreover, the only difference between the 
two series is the sign of $it$ so that it suffices to treat the first 
series. The treatment of the second hypergeometric series appearing in 
(\ref{nmn}) is similar. Using the series representation for $ {}_3F_2 $ 
we see that 
\begin{equation} \label{nseries}
{}_3F_2(s_0,1-s_0,\frac{1}{2}-\frac{\parity}{4}-it;1-2it,1;1)= 
\sum_{n=0}^\infty \frac{(s_0)_n (1-s_0)_n (\frac{1}{2}-\frac{\parity}{4} 
-it)_n}{(1)_n (1-2it)_n} \frac{1}{n!}. 
\end{equation} 
In order to determine its asymptotic behavior as $ \abs{t} \rightarrow \infty $ 
we want to interchange the summation with the limit, i.e.\  we want to 
take the limit $ \abs{t} \rightarrow \infty $ in each term of the series separately. 
For this let $ \epsilon \in (0;1/4) $ be sufficiently small and rewrite the 
terms appearing in (\ref{nseries}) as 
\begin{equation*} 
\abs{\frac{(s_0)_n (1-s_0)_n (\frac{1}{2}-\frac{\parity}{4}-it)_n}{(1)_n 
(1-2it)_n}} = 
\abs{\frac{(s_0)_n (1-s_0)_n}{(1+\epsilon)_n} }\abs{\frac{(1+\epsilon)_n 
(\frac{1}{2}-\frac{\parity}{4}-it)_n}{(1)_n (1-2it)_n}}. 
\end{equation*} 
For $ 0 \leq l \leq n $ we have 
\begin{align*} 
& \abs{\frac{(l+1+\epsilon)(l+\frac{1}{2}-\frac{\parity}{4}-it)}{(l+1) 
(l+1-2it)}}^2 = \\ 
& \hspace{25 ex} 
\frac{(l^2+(\frac{3}{2}-\frac{\parity}{4}+\epsilon)l+(1+\epsilon)(\frac{1}{2} 
-\frac{\parity}{4}))^2+t^2(l+1+\epsilon)^2}{(l+1)^4+4t^2(l+1)^2}. 
\end{align*} 
Since $ 2(l+1)>l+1+\epsilon $ and 
\begin{equation} \label{condition1/2} 
0 \leq l^2+\bigg(\frac{3}{2}-\frac{\parity}{4}+\epsilon\bigg)l+(1+\epsilon)
\bigg(\frac{1}{2}-\frac{\parity}{4}\bigg) \leq (l+1)^2 ,
\end{equation} 
this implies that 
\begin{equation*} 
\abs{\frac{(s_0)_n (1-s_0)_n (\frac{1}{2}-\frac{\parity}{4}-it)_n}{(1)_n 
(1-2it)_n}} \leq   \abs{\frac{(s_0)_n (1-s_0)_n}{(1+\epsilon)_n} }
\end{equation*} 
for all $ n \geq 0 $. 
Furthermore, the hypergeometric series 
\begin{equation*} 
{}_2F_1(s_0,1-s_0;1+\epsilon;1) = 
\sum_{n=0}^{\infty} \frac{(s_0)_n (1-s_0)_n}{(1+\epsilon)_n} \frac{1}{n!} 
\end{equation*} 
converges absolutely and therefore, by the theorem of majorized convergence, we finally 
obtain 
\begin{equation*} 
\lim_{\abs{t} \rightarrow \infty} \hyp(s_0,1-s_0,\frac{1}{2}-\frac{\parity}{4}-it; 
1-2it,1;1) = {}_2F_1(s_0,1-s_0;1;1/2). 
\end{equation*} 
Thus only the Gamma factors appearing in (\ref{nmn}) determine the asymptotic 
behavior and using Stirling's formula we see that 
\begin{equation*} 
\int_0^{\infty} W_{0,s_0-1/2}(y) W_{ 
\frac{\parity}{4},-it}(y) y^{s-1} \frac{dy}{y} 
= O\big(\abs{t}^{-(\frac{1}{2}-\frac{\parity}{4})} e^{-\frac{\pi}{2} \abs{t}}\big) 
\end{equation*} 
as $ |t| \rightarrow \infty $. This implies the desired bound.
\end{proof} 

\begin{lemma} \label{estimate_hypergeometric} 
Let  $ \parity \in \{\pm 1\} $. We have 
\begin{displaymath} 
{}_3F_2 (1/2+\parity/4-it, 1/2+iu, 1/2-iu; 1, 1-2it; 1) \ll 
e^{\pi|u|} |u|^{-2\epsilon} ,
\end{displaymath} 
as $ \abs{u} \rightarrow \infty $, where the implied constant does not 
depend on $ t $. Furthermore, there exists a constant $ C $ independent 
of $ t $ such that 
\begin{equation*} 
{}_3F_2 (1/2+\parity/4-it, 1/2, 1/2; 1, 1-2it; 1) \leq C. 
\end{equation*} 
\end{lemma} 
\begin{proof} 
Since $ \Re (2-2it - (1+1/2+\parity/4-it))  > 0 $ 
the hypergeometric series $ {}_3F_2 (1/2+\parity/4-it, 1/2+iu, 1/2-iu; 
1, 1-2it; 1) $ converges. 
By the definition of the hypergeometric series we have 
\begin{displaymath} 
{}_3F_2 (\frac{1}{2}+\frac{\parity}{4}-it, s, 1-s; 1, 1 -2it; 1) = 
1 + \sum_{m = 1}^{\infty} \frac{\left(s\right)_m 
\left(1-s\right)_m}{(1)_m m!} \frac{\left(\frac{1}{2}+\frac{\parity}{4} 
-it\right)_m}{(1 -2it)_m} 
\end{displaymath} 
with $ s = \frac{1}{2} + iu $. We now determine the behavior of the series 
as $ \abs{u}\to \infty $. We use the same argumentation that 
was already useful in the proof of Lemma \ref{nicoleresult}. We write 
\begin{equation} \label{est1} 
\frac{\left(s\right)_m \left(1-s\right)_m}{(1)_m m!} 
\frac{\left(\frac{1}{2}+\frac{\parity}{4}-it\right)_m}{(1 -2it)_m} 
= 
\frac{\left(s\right)_m \left(1-s\right)_m}{(1+\epsilon)_m m!} 
\frac{\left(\frac{1}{2}+\frac{\parity}{4}-it\right)_m 
(1+\epsilon)_m}{(1)_m (1-2it)_m} 
\end{equation} 
with $ \epsilon > 0 $ sufficiently small. 
As before the second factor on the right-hand side can be bounded 
in norm by 1, and it is straightforward to see that the first factor
is real and positive
so 
\begin{equation*}
\abs{{}_3F_2 \bigg(\frac{1}{2}+\frac{\parity}{4}-it, s, 1-s; 1, 1 -2it; 1\bigg)} 
 \leq  
{}_2F_1 \bigg(s, 1-s; 1+\epsilon; 1\bigg).
 \end{equation*}
The last hypergeometric function equals (see \cite[(1), p.~2]{Bailey:1964a}) 
\begin{equation*}\frac{\Gamma(1+\epsilon) \Gamma(\epsilon)}{\Gamma(\frac{1}{2}+\epsilon+iu) 
\Gamma(\frac{1}{2}+\epsilon-iu)},\end{equation*}
 and the first statement now follows from Stirling's
formula. The second statement follows from plugging $u=0$ in the above
argument. 
\end{proof} 
\begin{remark} 
A similar bound is given in \cite{Jakobson:1994a}, Claim~3.4, p.~1499. 
\end{remark} 
\begin{lemma} \label{estimateM} 
Let $\parity\in\{\pm 1\}$. For  $ u, t \in \R $ 
we have 
\begin{equation*} 
 \frac{1}{|\Gamma(\frac{1}{2}+\frac{\parity}{4 }+it)|^2} 
\int_0^\infty y^{-1/2+iu} \left|W_{\frac{\parity}{4 }, it} (y)\right|^2 
\frac{dy}{y}=O((1+\abs{t})^{-1/2}),
\end{equation*} as $\abs{t}\to\infty$. The implied constant is uniform
in $u$.
\end{lemma} 
\begin{proof} 
Set 
\begin{displaymath} 
I_{\parity, t} (u) := 
\int_0^\infty y^{-1/2+iu} \left|W_{\frac{\parity}{4} , it} (y)\right|^2 
\frac{dy}{y}. 
\end{displaymath} 
Since $\abs{I_{n, t} (u)}\leq {I_{n, t}} (0)$, we assume that $u=0$. 
By \cite[Formula 7.611~7., p.~821]{GradshteynRyzhik:2007a} we get 
\begin{eqnarray*} 
I_{n, t} (0) 
& = & 
\frac{\Gamma(\frac{1}{2}-2it) \Gamma(\frac{1}{2}) \Gamma(2it)}{\Gamma(\frac{1}{2} - 
\frac{\parity}{4}  + it) \Gamma(1 - \frac{\parity}{4} 
 - it)} \\ 
& & \times 
\ {}_3F_2 (\frac{1}{2}-2it, \frac{1}{2}, \frac{1}{2} - \frac{\parity}{4}  - it; 
1 - 2it, 1 - \frac{\parity}{4}  - it; 1) \\ 
& & 
+ 
\frac{\Gamma(\frac{1}{2}+2it) \Gamma(\frac{1}{2}) \Gamma(-2it)}{\Gamma(\frac{1}{2} 
- \frac{\parity}{4}  - it) \Gamma(1 - \frac{\parity}{4} 
 + it)} \\ 
& & 
\times 
\ {}_3F_2 (\frac{1}{2}+2it, \frac{1}{2}, \frac{1}{2} - \frac{\parity}{4}  + it; 1 + 2it, 
1 - \frac{\parity}{4} + it; 1).  
\end{eqnarray*} 
It suffices to consider the first term since the second term differs 
from the first one only by the sign of $ t $. Using the transformation 
formulae of \cite[p.~18]{Bailey:1964a}, as in the proof of Lemma \ref{nicoleresult}  we see that 
\begin{align*} 
{}_3F_2 & (\frac{1}{2}-2it,\frac{1}{2} ,\frac{1}{2}  - \frac{\parity}{4}  - it; 
1 - 2it, 1 - \frac{\parity}{4}  - it; 1) \\ 
& = 
\frac{\Gamma(1 - \frac{\parity}{4}-it) 
\Gamma(\frac{1}{2})}{\Gamma(\frac{1}{2}-\frac{\parity}{4}-it)} 
{}_3F_2 (\frac{1}{2}+\frac{\parity}{4}-it, \frac{1}{2}, \frac{1}{2}; 1, 1 -2it; 1).
\end{align*} 
By the second part of Lemma \ref{estimate_hypergeometric} the
hypergeometric series is bounded and we find -- by bounding all the Gamma
functions using Stirling -- that
\begin{equation*}\abs{I_{\parity,t}(0)} =O\left(\frac{\Gamma(1 - \frac{\parity}{4}-it) 
}{\Gamma(\frac{1}{2}-\frac{\parity}{4}-it)} \frac{\Gamma(\frac{1}{2}-2it)  \Gamma(2it)}{\Gamma(\frac{1}{2} - 
\frac{\parity}{4}  + it) \Gamma(1 - \frac{\parity}{4} 
 - it)}\right)=O(e^{-\pi\abs{t}}\abs{t}^{-\frac{1}{2}+\frac{p}{2}})
\end{equation*} 
as $ \abs{t} \rightarrow \infty $, which gives the result.
\end{proof} 
{\bf{Acknowledgements. } }We thank Gautam Chinta, Adrian Diaconu and
Valentin Blomer for useful discussions about multiple Dirichlet
series. Also we thank the referees for several helpful comments.
\bibliographystyle{abbrv}
\def\cprime{$'$}

\end{document}